\documentclass[11pt, english]{amsart}

\usepackage[english]{babel}
\usepackage{mathrsfs}
\usepackage{amsmath,amssymb,mathtools}
\usepackage[all]{xy}
\usepackage{url}

\usepackage{smfthm}

\usepackage{bm}
\usepackage{mathptmx}

\usepackage{nohyperref}
\usepackage{geometry}
\newcommand{\Z}{\mathbf Z}
\newcommand{\Qp}{\mathbf Q_p}
\newcommand{\Ph}{\varphi}
\newcommand{\z}{\bold{z}}
\newcommand{\zn}{\zeta_{p^n}}

\newcommand{\Dc}{\mathbf{D}_{\mathrm{cris}}}
\newcommand{\La}{\Lambda}

\newcommand{\res}{\mathrm{res}}
\newcommand{\R}{\mathbf{R}}

\newcommand{\bD}{\mathbf{D}}

\newcommand{\cl}{\mathrm{cl}}

\newcommand{\Iw}{\mathrm{Iw}}
\newcommand{\Bd}{\mathbf{B}_{\mathrm{dR}}}
\newcommand{\Bc}{\mathbf{B}_{\mathrm{cris}}}
\newcommand{\Exp}{\mathrm{Exp}}

\newcommand{\Hom}{\mathrm{Hom}}

\newcommand{\F}{\mathrm{Fil}}
\newcommand{\gam}{\gamma}

\newcommand{\Ddagrig}{\mathbf{D}^{\dagger}_{\mathrm{rig}}}

\newcommand{\Gal}{\mathrm{Gal}}

\newcommand{\CR}{\mathcal{R}}
\newcommand{\CDcris}{\mathscr{D}_{\mathrm{cris}}}

\newcommand{\ep}{\varepsilon}

\newcommand{\Zp}{\mathbf Z_p}
\newcommand{\Q}{\mathbf Q}
\newcommand{\CH}{\mathscr H}

\newcommand{\gr}{\text{\rm gr}}
\newcommand{\CE}{\mathcal E}
\newcommand{\CO}{\mathcal O}
\newcommand{\DdagrigA}{\bD^{\dagger}_{\mathrm{rig},A}}
\newcommand{\DdagrigAf}{\bD^{\dagger}_{\mathrm{rig},A_{\f}}} 
\newcommand{\DdagrigAg}{\bD^{\dagger}_{\mathrm{rig},A_{\g}}}

\newcommand\Fr{\textrm{Fr}}

\newcommand\BQ{\mathbf Q}
\newcommand\BC{\mathbf C}

\newcommand{\bM}{\mathbf M}
\newcommand\im{\mathrm{Im}}

\newcommand\pr{\mathrm{pr}}
\newcommand\Pro{\mathrm{Pr}}
\newcommand{\DdagrigE}{\bD^{\dagger}_{\mathrm{rig}
}}
\newcommand{\id}{\mathrm{id}}

\newcommand{\rk}{\mathrm{rk}}

\newcommand{\bdelta}{\boldsymbol{\delta}}
\newcommand{\f}{\mathbf f}
\newcommand{\g}{\mathbf{g}}
\newcommand{\bchi}{\boldsymbol{\chi}}
\newcommand{\boeta}{\boldsymbol{\eta}}
\newcommand{\Spm}{\mathrm{Spm}}

\newcommand{\CF}{\mathcal{F}}

\newcommand{\ba}{\mathbf{a}}
\newcommand{\bb}{\mathbf{b}}
\newcommand{\bxi}{\boldsymbol{\xi}}
\newcommand{\BF}{\mathbf{BF}}
\newcommand{\BFrm}{\mathrm{BF}}
\newcommand{\Zrm}{\mathrm{Z}}
\newcommand{\BFfrak}{\mathfrak{BF}}
\newcommand{\Tw}{\mathrm{Tw}}
\newcommand{\tr}{\mathrm{tr}}
\newcommand{\Zfrak}{\mathfrak{Z}}
\newcommand{\mm}{\mathfrak m}
\newcommand{\bN}{\mathbf{N}}

\newcommand{\Sym}{\mathrm{S}}
\newcommand{\TSym}{\mathrm{TS}}

\newcommand{\cF}{\mathscr{F}}
\newcommand{\Eis}{\mathrm{Eis}}
\newcommand{\EI}{\mathbf{Eis}}
\newcommand{\RI}{\mathbf{RI}}
\newcommand{\mom}{\mathrm{m}}
\newcommand{\et}{\mathrm{\acute et}}
\newcommand{\spec}{\mathrm{sp}}
\newcommand{\p}{\mathrm p}
\newcommand{\CGm}{\mathrm{CG}}
\newcommand{\Tor}{\mathrm{Tor}}

\newcommand{\Cp}{\mathbf{C}_p}

\newtheorem{mytheorem}[subsubsection]{Theorem}
\newtheorem{myproposition}[subsubsection]{Proposition}
\newtheorem{mydefinition}[subsubsection]{Definition}
\newtheorem{mylemma}[subsubsection]{Lemma}
\newtheorem{myremark}[subsubsection]{Remark}
\newtheorem*{mynonumberremark}{Remark}
\newtheorem{mycorollary}[subsubsection]{Corollary}
\newtheorem*{Extra-zero conjecture}{Extra-zero conjecture}

\newtheorem*{MTheorem}{Theorem I}
\newtheorem*{MT2Theorem}{Theorem II}
 
\makeindex             

\begin{document}
\title{On extra zeros of $p$-adic  Rankin--Selberg $L$-functions}
\author{Denis Benois}
\address{Institut de Math\'ematiques,
Universit\'e de  Bordeaux, 351, cours de la Lib\'eration  33405
Talence, France} 
\email{denis.benois@math.u-bordeaux.fr}
\author{St\'ephane Horte}
\address{28 rue des Platanes, 92500 Rueil-Malmaison, France}
\email{stephane.horte@polytechnique.edu }
\begin{abstract}
We prove a version of the Extra-zero conjecture, formulated by the first named author in \cite{Ben14},    for $p$-adic $L$-functions associated to Rankin--Selberg convolutions of modular forms of the same weight. The novelty of this result is to provide strong evidence in support of this conjecture in the {\it  non-critical} case,  which remained essentially unstudied.   
\end{abstract}
%
%
\maketitle
\tableofcontents 

\setcounter{section}{-1}

\section{Introduction}

\subsection{The extra-zero conjecture}
Let $ M/\Q$ be a pure motive of weight $\mathrm{wt} (M)\leqslant -2.$ 
Fix an odd prime number $p$ and denote by $V$ the $p$-adic realization of $M .$
So $V$ is a finite dimensional vector space over a finite extension $E$ of $\Qp,$ 
equipped with a continuous Galois action. 
We will always assume that $M$ has a good reduction at $p.$ Let $\Dc (V)$
denote the Dieudonn\'e module associated to the restriction of $V$ on the decomposition group at $p$  and $t_V(\Qp) =\Dc (V)/\F^0\Dc (V)$ be the corresponding tangent space. The Bloch--Kato logarithm is an isomorphism
\footnote{Since $\mathrm{wt} (M)\leqslant -2,$ one has $\Dc (V)^{\Ph=1}=0.$
}
\begin{equation}
\nonumber
\log_{V}\,:\,H^1_f(\Q_p, V)\rightarrow t_V(\Qp ) .
\end{equation}
Let $H^1_f(\Q, V)$ denote the Bloch--Kato Selmer group of $V.$ We have 
a commutative diagram 
\begin{equation}
\nonumber
\xymatrix
{
H^1_f(\Q, V) \ar[r]^{\res_p} \ar[dr]_{r_{V}} &H^1_f(\Q_p, V) \ar[d]^{\log_V}\\
 &t_V(\Qp ) ,
 }
\end{equation}
 where $\res_p$ is the restriction map, and $r_{V}$ denotes the resulting
 map. Note that $r_V$  is closely related to the  syntomic regulator. Assume that $\res_p$ is injective. One expects that this 
 always holds under our assumptions \cite{Ja89}.
A $\Ph$-submodule  $D\subset \Dc (V)$ is called {\it  regular}, if 
$D\cap \F^0\Dc (V)=\{0\}$ and 
\[
t_V(\Qp)=D\oplus r_{V} \left (H^1_f(\Q, V) \right ),
\]
where we identify $D$ with its image in $t_V(\Qp).$ If $D$ is regular, 
the composition of $r_V$ with the  projection $t_V(\Qp)
\rightarrow \Dc (V)/\left (\F^0\Dc (V)+D\right )$ is an isomorphism
\begin{equation}
\label{the map r_{VD}}
r_{V,D}\,:\,H^1_f(\Q, V) 
\rightarrow \Dc (V)/\left (\F^0\Dc (V)+D\right ).
\end{equation} 
We call $p$-adic regulator and denote by $R_p(V,D)$ the determinant of this
map. Of course, it depends on the choice of bases, but we omit them from notation
in this general discussion.

In \cite{PR95}, Perrin-Riou conjectured that to each regular $D$ one can associate a $p$-adic $L$-function $L_p(M,D,s)$
satisfying some precise interpolation property. At $s=0,$ the conjectural interpolation formula reads
\begin{equation}
\label{perrin-riou conjecture}
L_p( M,D,0)=\CE (V,D) R_p(V,D)\frac{L( M,0)}{R_\infty ( M)},
\end{equation}
where $L(M,s)$ is the complex $L$-function associated to $M,$ $R_\infty (M)$ and $R_p(V,D)$ are the archimedean and  $p$-adic regulators respectively,
computed in the compatible bases\footnote{See, for example \cite[Section~4.2.1]{Ben14}.} 
and $\CE (V,D)$ is the Euler-like factor given by
\begin{equation}
\label{definition of Euler-like factor}
\CE (V,D)=\det (1-p^{-1}\Ph^{-1} \vert D)\,
\det (1-\Ph \vert \Dc (V)/D).
\end{equation}
We say that $L_p(M,D,s)$ has an extra-zero at $s=0$ if $\CE (V,D)=0.$
By the weight argument, this can occur only if $\mathrm{wt}(M)=-2,$ and in this 
case we define 
\begin{equation}
\nonumber
\CE^+ (V,D)=\det \left (1-p^{-1}\Ph^{-1} \vert D/D^{\Ph=p^{-1}}\right )\cdot 
\det \bigl (1-\Ph \vert \Dc (V)/D \bigr
 ).
\end{equation}
If, in addition, we assume that the action of $\Ph$ on $\Dc (V)$ is semisimple
at $p^{-1},$ then $\CE^+ (V,D)\neq 0.$ In \cite{Ben14}, the first named author 
proposed the following conjecture. 

\begin{Extra-zero conjecture} The $p$-adic $L$-function $L_p(M,D,s)$
has a zero of order $e=\dim_{E}D^{\Ph=p^{-1}}$ at $s=0$ and 
\begin{equation}
\nonumber 
L_p(M,D,0)=\mathcal L(V,D)\, \CE^+(V,D)\, R_p(V,D)\frac{L(M,0)}{R_\infty (M)},
\end{equation}
where $\mathcal L(V,D)$ is the $\mathcal L$-invariant constructed 
in \cite{Ben14}. 
\end{Extra-zero conjecture}

\subsection{Examples}

1) Let $\Q (\eta)$ be the motive associated to an odd Dirichlet character 
$\eta \,:\,(\Z/N\Z)^* \rightarrow \overline\Q^{\,*}$ such that $(p,N)=1$
and let $\Q (\eta \chi)$ be its  twist by the cyclotomic character $\chi .$ In this case, the extra-zero conjecture follows from the explicit formula for the derivative of
Kubota--Leopoldt $p$-adic $L$-functions 
proved by Ferrero and Greenberg \cite{FG} and Gross--Koblitz \cite{GrosK} (see also \cite{Ben14a}). 

2) More generally, assume that  $F$ is either a totally real or a CM-field and $\Q (\rho)$ is the Artin motive over $F$ associated to an Artin representation $\rho$
of $G_F=\Gal (\overline F/F).$ In this case, the extra-zero  conjecture
for $M= \Q (\rho \chi)$ generalizes 
the Gross--Stark conjecture for abelian characters of totally real fields  proved by Dasgupta, Kakde and Ventullo \cite{DKV}.
However, our methods are also applicable beyond the totally real and CM cases,
and should provide some insight on a computation of Betina and Dimitrov \cite{BD19}. On the other hand, it seems interesting to 
compare the formalism of \cite{Ben14} with the approach of B\"uy\"ukboduk and Sakamoto \cite{BS19}.

3) Let  $M_f$ be the motive  associated to  a modular form $f$ of odd weight $k\geqslant 3$  and level $N_f.$ We assume  that $(p,N_f)=1.$ Then the Tate twist 
$M_f\left (\frac{k+1}{2}\right )$ of $M_f$ is a motive of weight $-2$ which has a good reduction at $p.$ In this case, the  extra-zero conjecture  was proved in \cite{Ben14a}. 

4) The $\mathcal L$-invariant of the adjoint weight one modular form introduced 
and  studied  in \cite{RRV20} is covered by the  general formalism of \cite{Ben14}.

We remark that in the cases 1) and 3) the motive $M$ is {\it critical} and 
$H^1_f(\Q, V)=0.$

\subsection{Rankin--Selberg $L$-functions}
In this paper, we prove some results toward the extra-zero conjecture for Rankin--Selberg convolutions of modular forms. This provides some evidences
for the Extra-zero conjecture in a {\it non-critical} setting.

 Let $f=\underset{n=1}{\overset{\infty}\sum } a_nq^n$ and 
$g=\underset{n=1}{\overset{\infty}\sum }b_n q^n$ be two newforms of weights $k_0$ and $l_0,$
levels $N_f$ and $N_g$ and nebentypus $\ep_f$ and $\ep_g$ respectively. 
Let $S$ denote the set of primes dividing $N_fN_g .$
The Rankin--Selberg $L$-function $L(f,g,s)$ is defined by
\begin{equation}
\nonumber
L(f,g,s)=L_{(N_fN_g)}(\ep_f\ep_g, 2s-k_0-l_0+2)\,\underset{n=1}{\overset{\infty}\sum}
\frac{a_nb_n}{n^s},
\end{equation}
where $L_{(N_fN_g)}(\ep_f\ep_g, 2s-k_0-l_0+2)$ is the Dirichlet $L$-function 
with removed Euler factors at the primes $q \in S .$ Note that, 
up to Euler factors at the bad primes, $ L(f,g,s)$ coincides with the 
$L$-function of the motive $M_{f,g}= M_f \otimes  M_g.$

Fix an odd  prime number $p$ such that $(p,N_fN_g)=1$ and denote by  $\alpha_p(f)$ and $\beta_p(f)$
(respectively by $\alpha_p(g)$ and $\beta_p(g)$) the roots of the Hecke polynomial
of $f$ (respectively $g$) at $p.$ We will always assume that 

\begin{itemize}
\item[]{\bf M1)}  $\alpha (f)\neq \beta (f)$
and $\alpha (g)\neq \beta (g).$

\item[]{\bf M2)}  $v_p(\alpha (f))<k_0-1$ and $v_p(\alpha (g)) <l_0-1,$
where $v_p$ denotes the $p$-adic valuation normalized by $v_p(p)=1.$

\end{itemize}

Denote by $f_{\alpha}$ and $g_{\alpha}$  the $p$-stabilizations of the forms $f$ and
$g$ with respect to $\alpha (f)$ and $\alpha (g).$ Let  $\f$ and $\g$ be 
Coleman families passing through $f_\alpha$ and $g_\alpha$ respectively.
We denote by $\f_x$ and  $\g_y$ the specializations of 
$\f$ and $\g$ at $x$ and $y$ respectively and by  
 $\f^0_x$ (respectively $\g^0_y$)
the primitive modular form of weight  $x$ (respectively $y$)
whose $p$-stabilization is $\f_x$ (respectively $\g_y$).

\begin{mytheorem} 
\label{theorem 3-variable p-adic L-function}
For each $0\leqslant a\leqslant p-2,$ there exists
a three variable $p$-adic analytic function $L_p(\f, \g, \omega^a ) (x,y,s)$ defined on $U_{f,g}\times \Zp,$ where $U_{f,g}$ is a sufficiently small neighborhood of $(k_0, l_0)$ in the weight space,   such that for each triple of integers $(x,y,j)\in U_{f,g}\times \Zp$ satisfying
\[
\begin{aligned}
&x\equiv k_0\mod{(p-1)},  &&y\equiv l_0\mod{(p-1)},\\ 
&j\equiv a\mod{(p-1)}, &&2\leqslant y\leqslant j< x, &&&&&
\end{aligned}
\]
one has 
\begin{equation}
\nonumber
L_p(\f, \g, \omega^a ) (x,y,j)=\frac{\CE (\f^0_x,\g^0_y,j)}{C(\f^0_x)}
\cdot 
\frac{\Gamma (j)\Gamma (j-y+1)}{(-i)^{x-y}2^{x-1}(2\pi)^{2j-y+1}
\left <\f^0_x,\g^0_y\right >}
\cdot 
L(\f^0_x,\g^0_y,j).
\end{equation}
In this formula, $\left <\f^0_x,\g^0_y\right >$ is the Petersson inner product,
$C(\f^0_x)$ is  defined in (\ref{definition of C(f)}),
and the Euler-like factor $\CE (\f^0_x,\g^0_y,j)$ is given by
\begin{align}
\nonumber
\CE (\f^0_x,\g^0_y,j)=
\left (1-\frac{p^{j-1}}{\alpha (\f_x^0)\alpha (\g_y^0)}\right )
\left (1-\frac{p^{j-1}}{\alpha (\f_x^0)\beta (\g_y^0)}\right )
\left (1-\frac{\beta (\f_x^0)\alpha (\g_l^0)}{p^j}\right )
\left (1-\frac{\beta (\f_x^0)\beta (\g_l^0)}{p^j}\right ).
\end{align}
\end{mytheorem}

This theorem was first proved in the ordinary case  by Hida \cite{Hi88}.
In \cite{Ur14}, Urban introduced the overconvergent projector and 
sketched a proof in the general non-ordinary case, but his arguments 
contained a gap which was filled  recently in \cite{Ur19}.
Meanwhile, Loeffler and Zerbes \cite{LZ} gave a complete proof of Theorem~\ref{theorem 3-variable p-adic L-function} based on the theory of Euler systems
and unconditional properties of the overconvergent projector proved in \cite{Ur14}.

\subsection{$L$-values at near central points}
Assume now that $f$ and $g$ are modular forms of the same weight $k_0=l_0\geqslant 2.$
Let  $M_{f,g}=M_f\otimes M_g$ be the tensor product of motives associated to 
$f$ and $g.$ Its $p$-adic realization is $W_{f,g}=W_{f}\otimes_E W_{g},$
where  $W_{f}$ and $W_{g}$ are the $p$-adic representations associated
to $f$ and $g$ by Deligne \cite{De71}, and $E$ denotes an
appropriate  finite extension of $\Qp.$ Since $(p, N_fN_g)=1,$ the representations
$W_f,$ $W_g$ and $W_{f,g}$ are crystalline at $p,$ and we have
\[
\Dc (W_{f,g})= \Dc (W_f)\otimes_E \Dc (W_g).
\] 
The motive $M_{f,g}(k_0)$  is non critical, of motivic  weight $-2,$ and
its $p$-adic realization is $V_{f,g}=W_{f,g}(k_0).$ 
Let $E\eta_f^\alpha$  be the one dimensional eigenspace
\footnote{Here $\eta_f^\alpha$ denotes the canonical eigenvector associated 
to $\alpha (f).$  See Section~\ref{subsection p-adic representations} below.}
 of $\Dc (W_f)$ 
associated with the eigenvalue $\alpha (f).$ Set
\begin{equation}
\nonumber
D =\eta_f^{\alpha}\otimes_E \Dc (W_g(k_0)).
\end{equation}
Then $D$ is  a $\Ph$-submodule of $\Dc (W_{f,g}),$ and an easy computation
shows that
\begin{equation}
\label{comparision of euler factors}
\CE (f,g,k_0)= \CE (V_{f,g}, D),
\end{equation} 
where the right hand side term is defined by (\ref{definition of Euler-like factor}). 

We define the $p$-adic $L$-function $L_{p,\alpha}(f,g,s)$  as the restriction
of the  three variable $p$-adic $L$ function from Theorem~\ref{theorem 3-variable p-adic L-function}:
\begin{equation}
\nonumber
L_{p,\alpha}(f,g,s)=L_p(\f,\g,\omega^{k_0}) (k_0,k_0,s).
\end{equation} 
A density argument shows that this function does not depend on the choice of the 
$p$-stabilization of $g$ (see Section~\ref{subsection Extra-zeros of $p$-adic Rankin--Selberg $L$-functions}).  

Assume that  $\ep_f \ep_g\neq \id.$ Let $\BFrm_{f^*,g^*}^{[k_0-2]}\in H^1(\Q, V_{f,g})$ denote the Beilinson--Flach element
\footnote{In the weight $2$ case, the motivic version of this element was first constructed  by Beilinson\cite{Bei84}. In \cite{Fl90}, Flach exploited  the $p$-adic 
realization of Beilinson's elements in the study of the Selmer group of the symmetric square of an elliptic curve.} constructed in \cite{KLZb}. 
From the results of Besser \cite{Bes00}, it  follows that  the restriction 
$\res_p\left (\BFrm_{f^*,g^*}^{[k_0-2]}\right )$ of this element on the decomposition group at $p$ lies in $H^1_f(\Qp, V_{f,g})$ (see \cite[Proposition~5.4.1]{KLZb} for  detailed arguments). The modular forms 
$f$ and $g^*$ define canonical bases $\omega_f$ of $\F^0\Dc (W_f)$ and
$\omega_{g^*}$ of  $\F^0\Dc (W_{g^*}).$  Consider the canonical pairing
\[
\left [\,\,,\,\,\right ]\,:\,\Dc (W_g)\times \Dc (W_{g^*}) \rightarrow \Dc (E(1-k_0))
\]
Let $\eta_g$ be any element of $\Dc (W_g)$ such that
\[
\left [\eta_g, \omega_{g^*}\right ]= e_{1}^{\otimes (1-k_0)},
\]
where $e_{1}$ is the canonical basis
\footnote{More explicitly, $e_1=\ep \otimes t^{-1}$, where $\ep$ is a compatible system of $p^n$th roots of unity, and $t=\log [\ep]$ the associated 
element of $\Bc$.
}
of $\Dc (E(1)).$ Set $b=\omega_f \otimes \eta_g\otimes e_1^{\otimes k_0}\in \Dc (V_{f,g}).$ Then the element
\[
\overline b_{\alpha} =b \mod{\left (\F^0\Dc (V_{f,g})+D \right )}
\]
is a basis of the one dimensional space $\Dc (V_{f,g})/\left (\F^0\Dc (V_{f,g})+D\right ).$ 
Therefore the image of $\res_p\left (\BFrm_{f^*,g^*}^{[k_0-2]}\right )$ under the composition
\[
H^1_f(\Qp, V_{f,g}) \xrightarrow{\log_{V_{f,g}}} 
t_{V_{f,g}}(\Qp) \rightarrow
\Dc (V_{f,g})/\left (\F^0\Dc (V_{f,g})+D \right 
)
\]
can be written in a unique way as $\widetilde R_p(V_{f,g},D) \cdot \overline b_\alpha $ with $\widetilde R_p(V_{f,g},D) \in E.$ We remark that 
$\widetilde R_p(V_{f,g},D)$ concides with the regulator $R_p(V_{f,g},D)$ if $H^1_f(\Q, V_{f,g})$ is the one dimensional vector space generated 
by $\BFrm_{f^*,g^*}^{[k_0-2]}.$ One expect that this always holds
\footnote{Beilinson conjectures in the formulation of Bloch and Kato 
predict that $H^1_f(\Q, V_{f,g})$ has dimension $1$.} . 
We have the following result toward Perrin-Riou's conjecture 
(\ref{perrin-riou conjecture}). 

\begin{mytheorem} Assume that $\ep_f \ep_g\neq \id$ (in particular, this condition
implies that $f\neq g^*$).  Then  
the following formula holds:
\[
L_{p,\alpha}(f,g,k_0)=\frac{\ep (f,g,k_0)\cdot
\CE(V_{f,g},D)}{  C(f) \cdot   G(\ep_f) \cdot   G(\ep_g)
\cdot  (k_0-2)!}  \cdot \widetilde{R}_p (V_{f,g},D).
\]
Here $G(\ep_f)$ and $G(\ep_g)$ are Gauss sums associated to $\ep_f$ and $\ep_g,$
$\ep (f,g,k_0)$ the epsilon constant of the functional equation of 
the complex $L$-function, and 
\[
C(f)=\left (1-\frac{\beta (f)}{p\alpha (f)}\right )\cdot 
\left (1-\frac{\beta (f)}{\alpha (f)}\right ).
\]
\end{mytheorem}
\begin{proof} This theorem was first proved by Bertolini--Darmon--Rotger \cite{BDR15a} for modular forms of weight $2.$  Kings, Loeffler and Zerbes  extended the proof to the higher weight case \cite[Theorem~7.2.6]{KLZb}, \cite[Theorem~7.1.5]{LZ}.  Note that the results  proved in \cite{KLZb} and \cite{LZ} are in fact more general and include also  the case of modular forms of different weights.
\end{proof}
We remark that, combining this formula with the computaiton of the special value
of the complex $L$-function in terms of the Beilinson regulator, one can
write this theorem in the form (\ref{perrin-riou conjecture})
(see \cite[Theorem~7.2.6]{KLZb}). Also, this result suggests that 
$L_{p,\alpha}(f,g,s+k_0)$ satisfies the conjectural interpolation properties 
of
$
L_p(M_{f,g}(k_0), D,s)
$
 up to "bad" Euler factors at primes dividing $N_fN_g.$

\subsection{The main result}
We keep previous notation and conventions. 
In this paper, we prove a result toward the extra-zero conjecture 
for the $p$-adic $L$-finction $L_{p,\alpha}(f,g,s)$ as $s=k_0.$
In addition to {\bf M1-2)} we assume that the following conditions hold:

\begin{itemize}
\item[]{\bf M3)}  $\ep_f,$ $\ep_g$ and $\ep_f\ep_g$ are primitives 
modulo $N_f,$ $N_g$ and $\text{lcm} (N_f,N_g)$ respectively. 

\item[]{\bf M4)} $\ep_f(p)\ep_g (p)\neq 1.$
\end{itemize}
Note that the condition {\bf M3)} can be relaxed. We introduce it mainly because 
in this case the functional equation for the Rankin--Selberg $L$-function has a
simpler form. However, the condition  $\ep_f\ep_g \neq \id$ can not be relaxed.
In particular, the case   $f=g^*$  should be excluded. 

Assume that the interpolation factor $\CE (f,g,k_0)$ vanishes. Without 
loss of generality, we can assume that $\alpha (f) \beta (g)=p^{k_0-1}.$
Then the $\Ph$-module $D$ has a four-step filtration 
\begin{equation}
\label{filtration introduction}
\{0\}\subset D_{-1}\subset D_0 \subset D_1\subset \Dc (V_{f,g})
\end{equation}
such that $D_0=D,$  $D_{-1}$ is the eigenline of $\Ph$ associated to 
the eigenvalue $\alpha (f) \alpha (g)p^{-k_0},$ and $D_1$ is the unique subspace
such that $\Ph$ acts on $D_1/D_0$ as the multiplication by 
$\beta (f)\alpha (g)p^{-k_0}.$
Note that $\Ph$ acts
on $D_0/D_{-1}$ as the multiplication by $p^{-1}.$ Taking duals, we have 
a  filtation on $\Dc (V_{f,g}^*(1))$
\[
\{0\} \subset D_{-1}^\perp \subset D_{0}^\perp\subset D_1^\perp \subset \Dc (V_{f,g}^*(1)),
\]
such  that $\Ph$ acts trivially on $D_1^\perp/D_0^\perp.$ This filtration 
induces a filtration of the associated $(\Ph,\Gamma)$-modules
\[
\{0\}\subset F_0\Ddagrig (V_{f,g}^*(1))\subset F_1\Ddagrig (V_{f,g}^*(1))\subset 
\Ddagrig (V_{f,g}^*(1)).
\]
By \cite[Proposition~1.5.9]{Ben11}, the cohomology of the quotient 
\[
\gr_1\Ddagrig (V_{f,g}^*(1))= F_1\Ddagrig (V_{f,g}^*(1))/F_0\Ddagrig (V_{f,g}^*(1))
\]
has a canonical  decomposition 
\begin{equation}
\label{decomposition of cohomology introduction}
H^1\left (\gr_1\Ddagrig (V_{f,g}^*(1))\right )=
H^1_f\left (\gr_1\Ddagrig (V_{f,g}^*(1))\right ) \oplus
H^1_c\left (\gr_1\Ddagrig (V_{f,g}^*(1))\right )
\end{equation}
into two subspaces of dimension $1,$ which are both canonically isomorphic to
$D_1^\perp/D_0^\perp.$ The   interpolation of Beilinson--Flach elements (see \cite{KLZ}) provides us with an 
element 
$\widetilde Z_{f,g}^{[k_0-1]}\in H^1\left (\gr_1\Ddagrig (V_{f,g}^*(1))\right ).$
Since Beilinson's conjecture and the injectivity of the restriction map
$H^1_f(\Q,V_{f,g}) \rightarrow H^1_f(\Qp,V_{f,g})$ are not known in our case, 
we can not work with the general definition of the $\mathcal L$-invariant 
proposed in \cite{Ben14}. To remeday this problem, we introduce 
the   ad hoc invariant $\widetilde{\mathcal L}(V_{f,g},D)$
as the slope of the line generated by $\widetilde Z_{f,g}^{[k_0-1]}$
under the decomposition (\ref{decomposition of cohomology introduction}).
We show that $\widetilde{\mathcal L}(V_{f,g},D)$ coincides 
with the invariant $\mathcal L (V_{f,g},D)$ defined in \cite{Ben14}
if the above mentioned conjectures hold and the regulator $\widetilde{R}_p (V_{f,g},D)$ does not vanish. The main result of this paper is the following theorem (see Theorem~\ref{main theorem}).

 \begin{MTheorem} 
\label{main theorem}
Assume that $\alpha (f) \beta (g)=p^{k_0-1}.$ Then

1) $L_{p,\alpha}(f,g, k_0)=0.$

2) The following conditions are equivalent:
\begin{itemize}
\item[i)]{} $\mathrm{ord}_{s=k_0}L_{p,\alpha}(f,g,s)=1$.
\item[ii)]{} $\,_b\widetilde\Zrm_{f,g}^{[k_0-1]}\notin 
 H^1_c\left (\gr_1 \Ddagrig (V_{f,g}^*(1))
\right ).$
\end{itemize}

3) In addition to the assumption that $\alpha (f) \beta (g)=p^{k_0-1},$ suppose that  
\[
\,_b\widetilde\Zrm_{f,g}^{[k_0-1]}\notin 
 H^1_f\left (\gr_1 \Ddagrig (V_{f,g}^*(1))
\right ).
\]
 Then 
\[
L_{p,\alpha}'(f,g,k_0)=\frac{\ep (f,g,k_0)\cdot \widetilde{\mathcal L}(V_{f,g},D) \cdot
\CE^+(V_{f,g},D)}{  C(f) \cdot   G(\ep_f) \cdot   G(\ep_g)
\cdot  (k_0-2)!}  \cdot \widetilde{R}_p (V_{f,g},D), 
\]
where 
\begin{equation}
\nonumber
\CE^+(V_{f,g},D)=\left (1-\frac {p^{k_0-1}}{\alpha (f) \alpha (g)} \right ) 
\left (1-\frac {\beta (f) \alpha (g)}{p^{k_0}} \right )
\left (1-\frac {\beta (f) \beta (g)}{p^{k_0}} \right ).
\end{equation}
\end{MTheorem}

\begin{mynonumberremark}
{\rm We expect that $\,_b\widetilde\Zrm_{f,g}^{[k_0-1]}$ is in general position 
with respect to the subspaces $H^1_c\left (\gr_1 \Ddagrig (V_{f,g}^*(1))\right )$
and $H^1_f\left (\gr_1 \Ddagrig (V_{f,g}^*(1))\right ).$  In this case, 
$\mathrm{ord}_{s=k_0}L_{p,\alpha}(f,g,s)=1$ and both  the $p$-adic regulator   $\widetilde{R}_p (V_{f,g},D)$ and the $\mathcal L$-invariant 
$\widetilde{\mathcal L}(V_{f,g},D)$ does not vanish.}
\end{mynonumberremark}

It would be interesting to understand the relationship between our approach and the 
methods of Rivero and Rotger \cite{RR18}, where the case $g=f^*$ is studied. 

\subsection{Outline of the proof}
The proof of Theorem~I relies heavily on the theory of Beilinsion--Flach elements
 initiated by  Bertolini, Darmon and Rotger \cite{BDR15a, BDR15b} and  extensively developed by Lei, Kings,  Loeffler and  Zerbes \cite{LLZ14, KLZ, KLZb, LZ}.  Note  
that in the non ordinary case, the overconvergent Shimura isomorphism  of Andreatta, Iovita and Stevens \cite{AIS} plays a crucial role in the theory.

Let $\f=\underset{n=1}{\overset{\infty}\sum } \ba_nq^n$ and 
$\g=\underset{n=1}{\overset{\infty}\sum} \bb_nq^n$ denote Coleman families passing through the stabilizations $f_\alpha$ and $g_\alpha$ of two forms of weight $k_0.$
Kings, Loeffler and Zerbes \cite{KLZb} (in the ordinary case)  and
Loeffler and Zerbes \cite{LZ} (in the general case)  expressed the three variable $p$-adic $L$-function as the image of the stabilized three variable
Beilinson--Flach element under the large exponential map. 
Using the semistabilized versions of Beilinson--Flach elements we define, 
in a neighborhood of $k_0,$ two anaytic $p$-adic $L$-functions 
$L_p^{\textrm{wc}}(\f,\g,s)$ and $L_p^{\textrm{wt}}(\f,\g,s)$ 
which can be viewed as "improved" versions of the three variable $p$-adic $L$-function. Namely
\begin{equation}
\nonumber
\begin{aligned}
&
 L_p(\f, \g, \omega^{k_0} ) (k_0,s,s)=
(-1)^{k_0}\left (1-\frac{\bb_{p}(s)}{\ep_{g}(p)\ba_{p}(k_0)}\right )
 \left (1-\frac{\ep_g(p)\ba_{p}(k_0)}{p\bb_{p}(s)}\right )^{-1}
 L_p^{\mathrm{wc}}(\f,\g,s),\\
&L_p(\f, \g, \omega^{k_0-1} ) (k_0,s,k_0-1)
=-\left (1-\frac{p^{k_0-2}}{\ba_p(k_0)\bb_p(s)}\right )
\left (1-\frac{\ep_f(p)\bb_p(s)}{\ba_p(k_0)}\right )
L_p^{\mathrm{wt}}(\f,\g,s)
\end{aligned}
\end{equation}
(see Propositions~\ref{proposition first improved L-function} and \ref{proposition second improved L-function}).

The crystalline $(\Ph,\Gamma)$-modules $\gr_1   \Ddagrig (V_{f,g})$ and 
 $\gr_0   \Ddagrig (V_{f,g}^*(1))$ are Tate dual to each other, and we denote by
\[
\left [\,\,,\,\,\right ]\,:\,
\CDcris \left (\gr_1   \Ddagrig (V_{f,g})\right ) \times
\CDcris \left (\gr_0   \Ddagrig (V_{f,g}^*(1))\right )
\rightarrow E
\]
the resulting duality of Dieudonné modules. Analogously, the  $(\Ph,\Gamma)$-modules $\gr_1   \Ddagrig (V_{f,g}^*(1))$ and 
 $\gr_0   \Ddagrig (V_{f,g})$ are Tate dual to each other and 
we  denote by 
\[
\left < \,\,,\,\,\right >\,:\, H^1\left ( \gr_1   \Ddagrig (V_{f,g}^*(1))\right )
\times H^1\left ( \gr_0  \Ddagrig (V_{f,g})\right ) \rightarrow E
\]
the induced local duality on cohomology. Let 
\[
\exp \,:\, \CDcris \left ( \gr_0  \Ddagrig (V_{f,g})\right )\rightarrow H^1 \left ( \gr_0  \Ddagrig (V_{f,g})\right )
\]
and 
\[
\log \,:\,H^1\left (\gr_1   \Ddagrig (V_{f,g})\right )\rightarrow
\CDcris \left (\gr_1   \Ddagrig (V_{f,g})\right )
\]
denote the Bloch--Kato exponential and logarithm maps for the corresponding 
$(\Ph,\Gamma)$-modules (see \cite[Section~2.1.4]{Ben14}, \cite{Nak14}).
The filtered Dieudonné modules $\CDcris \left ( \gr_0  \Ddagrig (V_{f,g})\right )$
and $\CDcris \left ( \gr_0  \Ddagrig (V_{f,g}^*(1))\right )$ have canonical bases
which we denote by $d_{\alpha\beta}$ and $n_{\alpha\beta}$ respectively
(see Section~\ref{subsection zeta elements}). 
 The Beilinson--Flach element 
$\BFrm^{[k_0-2]}_{f^*g^*}$ can be "projected" on the subquotient 
$H^1 \left (\gr_1 \Ddagrig (V_{f,g})\right )$ of $H^1(\Qp, V_{f,g})$
and we denote by $\Zrm_{f^*,g^*}^{[k_0-2]}$ its image in $H^1 \left (\gr_1 \Ddagrig (V_{f,g})\right )$ (see Definition~\ref{definition of Zrm}
and Corollary~\ref{corollary about Zrm}).
The functional equation for the improved $L$-functions has the following interpretation
 in terms of Beilinson--Flach elements (see Theorem~\ref{theorem functional equation for zeta elements}).

\begin{MT2Theorem} Assume that $\beta (f)\alpha (g)\neq p^{k_0-1}.$ Then 
the elements $\Zrm_{f^*,g^*}^{[k_0-2]}$ and
$\widetilde\Zrm_{f,g}^{[k_0-1]}$ are related by the equation
\begin{equation}
\label{functional equation for zeta introduction}
\frac{\left <\widetilde\Zrm_{f,g}^{[k_0-1]}, \exp (d_{\alpha\beta}) \right >}
{
G(\ep_f^{-1}) G(\ep_g^{-1})
}
=(-1)^{k_0-1}\ep (f,g,k_0)\cdot 
\CE (V_{f,g}, D_{-1}) \cdot
\frac{\left [ \log \left (\Zrm_{f^*,g^*}^{[k_0-2]}\right ), n_{\alpha\beta}
\right ]}{(k_0-2)! G(\ep_f) G(\ep_g)},
\end{equation}
where 
\[
\CE (V_{f,g}, D_{-1})= \det \left (1-p^{-1}\Ph^{-1} \mid D_{-1} \right )
\det \left (1-\Ph \mid \Dc (V_{f,g})/D_{-1}\right ).
\]
\end{MT2Theorem}

We deduce Theorem I from this theorem. Namely, the machinery developed in 
\cite{Ben14} gives a formula for the derivative of $L_{p,\alpha}(f,g,s)$ at $s=k_0$
in terms of the $\mathcal L$-invariant $\widetilde{\mathcal L}(V_{f,g},D)$ 
and   the left hand side of equation (\ref{functional equation for zeta introduction}). Using Theorem~II, we express it in terms of the 
right hand side of (\ref{functional equation for zeta introduction}), which 
is essentially the regulator $\widetilde R_p(V_{f,g},D).$

We hope  that our approach could be useful to study some other cases 
of extra-zeros of   non-critical motives.

\subsection{The plan of the paper} The organization of the paper is as follows. 
In Section 1, we review basic results about the cohomology of $(\Ph,\Gamma)$-modules
and the large exponential map. In Sections 2.1-2.2, we review the definition of the 
$\mathcal L$-invariant in the non critical case. Note that in \cite{Ben14},  
the first named author considered only the representations arising from motives
of weight $-2$ because the dual case  can be treated
using the functional equation. However, to compare this general  definition with our 
ad hoc invariant $\widetilde {\mathcal L}(V_{f,g},D),$  it is important to have an intrinsic definition of the $\mathcal L$-invariant  in the weight $0$ case. This is the subject of Section 2.3. In Section 3, for the convenience of the reader, we review the overconvergent \'etale cohomology of modular curves and its application to Coleman families following \cite{AIS} and \cite{LZ}. 
In Section 4, we review the construction of Beilinsion--Flach  elements following 
\cite{KLZb, LLZ14, KLZ} and introduce  semistabilized Beilinsion--Flach  elements,
which play a key role in this paper. Local properties of these elements are studied in Section 5. In Section 6, using semistabilized  Beilinsion--Flach  elements,
we define the improved $p$-adic $L$-functions $L_p^{\textrm{wc}}(\f,\g,s)$ and
$L_p^{\textrm{wt}}(\f,\g,s)$ and prove Theorem~II. In Section~7,
we prove Theorem~I. 

\subsection{Acknowledgements} We would like to thank Mladen Dimitrov, David Loeffler
and Sarah Zerbes for their remarks on the first draft of this paper. 
This work was partially supported by the Agence National de Recherche
(grant  ANR-18-CE40-0029) in the framework 
of the ANR-FNR project
"Galois representations, automorphic forms and their $L$-functions".

\newpage

\section{The exponential map}

\subsection[]{Notation and conventions.}

\subsubsection{} Le $p$ be an odd prime. In this subsection, $\overline{\Q}_p$ denotes  a fixed algebraic closure of $\Qp$ and $\Cp$ the $p$-adic completion
of $\overline{\Q}_p.$ For any extension $L/\Qp,$ we set $G_L=\Gal (\overline{\Q}_p/L).$
Fix  a  system $\ep=(\zn)_{n\geq 0}$ of primitive 
$p^n$th roots of unity such that $\zeta_{p^{n+1}}^p=\zn$ for all $n\geq 0.$
Set $K_n=\Qp (\zn),$ $K_\infty=\underset{n\geq0}{\cup}K_n$ and 
$\Gamma=\Gal (K_\infty/\Qp).$
There is  a canonical decomposition 
\begin{equation}
\nonumber
\Gamma \simeq \Delta \times \Gamma_1, \qquad \Gamma_1=\Gal (K_\infty/K_1).
\end{equation}
 We denote by $\chi \,:\,\Gamma \rightarrow \Zp^*$
the cyclotomic character and by $\omega$ its restriction on $\Delta=\Gal (K_1/\Qp).$
We also denote by $\left <\chi \right >$ the composition of $\chi$ with 
the projection $\Z_p^*\rightarrow (1+p\Zp)^*$ induced by the canonical decomposition
$\Z_p^*\simeq (\Z/p\Z)^*\times  (1+p\Zp)^*.$

\subsubsection{} Let $(w_1,\ldots ,w_d)$ be a finite
set of variables. 
If $E/\Qp$ is a finite extension,   we denote by 
\begin{equation}
\label{definition of Tate algebra}
A=
E\left < {w_1}/{p^r}, \ldots , {w_d}/{p^r} \right > 
\end{equation}
the Tate algebra of formal power series 
\[
F(w_1, \ldots ,w_d)=
\underset{(m_1, \ldots ,m_d)\in \mathbf{N}^d}
\sum c_{m_1, \ldots ,m_d}(w_1/p^r)^{m_1} \cdots (w_d/p^r)^{m_d}
\]
such that $c_{m_1, \ldots ,m_d}\to 0$ when $m_1+\cdots +m_d\to +\infty .$

\subsubsection{}
\label{subsection A^{wt}}
Fix $k=(k_1, \ldots, k_d)\in \Zp^d$ and consider 
the closed disk with center $k$ and radius $1/p^{r-1}$ in $\Zp^d$:
\begin{equation}
\nonumber
D(k,1/p^{r-1})=k+p^{r-1}\Zp^d.
\end{equation}
For each $F\in A,$ we define the   $p$-adic  analytic function $\mathcal A^{\mathrm{wt}}(F)$
on $D(k,1/p^{r-1})$ with values in $E$ by
\[
\mathcal A^{\mathrm{wt}}(F)(\kappa_1,\ldots ,\kappa_d)=
F \left ((1+p)^{\kappa_1-k_1}-1, \ldots ,
(1+p)^{\kappa_d-k_d}-1 \right ).
\]
If $M$ is an $A$-module, and $x\in \Spm (A),$ we denote by $\mathfrak m_x$ the corresponding maximal ideal of $A$ and set $k(x)=A/\mathfrak m_x$ and $M_x=M\otimes_{A}k(x).$
Let
\begin{equation}
\label{definition of general specialization map}
\spec_x\,:\,M\rightarrow M_x
\end{equation}
denote the specialization map. If 
\[\mathfrak m_x=\left ((1+w_1)-(1+p)^{\kappa_1-k_1}, \ldots ,
(1+w_d)-(1+p)^{\kappa_d-k_d}\right )
\]
with $\kappa=(\kappa_1,\ldots ,\kappa_d)\in D(k,1/p^{r-1}),$ we will
often  write $M_\kappa$ and $\spec_\kappa$ instead $M_x$ and $\spec_x$
respectively.

\subsubsection{}
Let $\CH_E$ denote the ring of power series $f(T)\in E[[T]]$ which converge
on the open unit disk. If $\gamma_1\in \Gamma_1$ is a fixed generator 
of the $p$-procyclic group $\Gamma_1,$ then the map $\gamma_1 \mapsto T-1$
identifies $\CH_E$ with the large Iwasawa algebra $\CH_E (\Gamma_1).$
We set $\CH_E(\Gamma) =E [\Delta ]\otimes_E \CH (\Gamma_1).$
Each $h\in  \CH_E(\Gamma)$ can be written in the form
\begin{equation}
\nonumber
h=\underset{i=1}{\overset{p-1}\sum}\delta_i h_i(\gamma_1-1), 
\qquad \text{\rm where $\delta_i=\frac{1}{\vert \Delta \vert}\underset{g\in\Delta}\sum \omega (g)^{-i}g.$}
\end{equation}
Define
\begin{equation}
\nonumber
\mathcal A^{\mathrm c}_{\omega^i}(h)(s)=h_i(\chi (\gamma_1)^s-1),
\qquad 1\leqslant i\leqslant p-1.
\end{equation}
Note that the series $\mathcal A^{\mathrm c}_{\omega^i}(h)(s)$ converge
on the open unit disk.

For each $i\in \Z,$ we have a $\Gamma$-equivariant map
\begin{equation}
\label{definition of cyclotomic spec on power series}
\begin{aligned}
&\spec^c_{m}\,:\, \CH_E(\Gamma) \rightarrow E(\chi^m),\\
&\spec^c_m(f)=\mathcal A_{\omega^m}^c(f)(m)\otimes \chi^m.
\end{aligned}
\end{equation}

\subsubsection{}
If $A$ is a Tate algebra of the form (\ref{definition of Tate algebra}), we 
set $\CH_A(\Gamma)=A\widehat\otimes_E\CH_E(\Gamma).$ For each $F\in \CH_A(\Gamma)$
define 
\begin{equation}
\label{transform A}
\mathcal A_{\omega^i}(F)(\kappa_1, \ldots ,\kappa_d,s)=
\left (\mathcal A^{\mathrm{wt}}\otimes \mathcal A^{\mathrm c}_{\omega^i}\right ) (F),
\qquad 1\leqslant i\leqslant p-1.
\end{equation}
Let $\eta \,:\,\Gamma \rightarrow A^*$ be a continuous character. 
When $\left. \eta  \right \vert_{\Delta}=\omega^m$ for some $0\leqslant m\leqslant p-2.$ The algebra $\CH_A(\Gamma)$ is equipped with the twist operator
\begin{equation}
\label{definition of Tw_eta}
\Tw_{\eta} \,:\, \CH_A(\Gamma)\rightarrow \CH_A(\Gamma), \qquad 
\Tw_{\eta} \left (F(\gamma_1-1)\delta_i \right )= 
F(\chi (\gamma_1)^m\gamma_1-1)\delta_{i-m}.
\end{equation}
If $\eta=\chi^m$ with $m\in \Z,$ we write $\Tw_{m}$ instead $\Tw_{\chi^m}.$
We have
\begin{equation}
\label{definition of Tw_m}
\Tw_m \left (F(\gamma_1-1)\delta_i \right )= 
F(\chi (\gamma_1)^m\gamma_1-1)\delta_{i-m}.
\end{equation}
The map $\spec^c_m$ can be extended by linearity to a map $\CH_A(\Gamma)\rightarrow A (m).$ Directly from definitions, one has 
\begin{equation}
\nonumber
\spec_m^c= \spec_{0}^c \circ \Tw_m.
\end{equation}


\subsubsection{} Let $A=E\left <  {w}/{p^r} \right >$
be the one variable Tate algebra over $E.$
We denote by  $\bchi \,:\,\Gamma \rightarrow A^*$  the character defined by
\begin{equation}
\label{character chi bold}
\bchi (\gam )=
\exp \left (\log_p(1+w)\frac{\log (\left <  (\chi (\gam)) \right >}
{\log (1+p)} \right )
\end{equation}
Note that  
$\mathcal A^{\mathrm{wt}}(\bchi (\gam ))(\kappa)=
\left <  (\chi (\gam)) \right >^{\kappa-k}. $ The map $\Tw_{\bchi} \,:\, \CH_A (\Gamma )\rightarrow \CH_A (\Gamma )$ is explicitly give by  
\begin{equation}
\begin{aligned}
\label{definition of the bold Tw}
&\Tw_{\bchi}\left (F(\gamma_1-1)\delta_i \right )= 
F(\bchi (\gamma_1) \gamma_1-1)\delta_{i}.
\end{aligned}
\end{equation}
For any $h\in \CH_A(\Gamma)$ one has 
\begin{equation}
\nonumber
\left.\left (\mathcal A^{\mathrm{wt}} \circ \Tw_{\bchi} h\right )\right \vert_{\kappa=m} =\Tw_{m-k} \circ ((\mathcal A^{\mathrm{wt}} h )\vert_{\kappa=m}) ,
\qquad m\equiv k\pmod{(p-1)p^{r-1}}
\end{equation}
and 
\begin{equation}
\label{property of Tw_bchi}
\mathcal A_{\omega^i} ( \Tw_{\bchi} h) (\kappa,s) =\mathcal A_{\omega^i} (h)(\kappa, s+\kappa-k).
\end{equation}

\subsubsection{}  For each $r\in [0,1),$ we denote by $\CR_E^{(r)}$
the ring of power series 
\[
f(X)=\underset{n\in \Z}{\sum}a_nX^n, \qquad a_n\in E
\]
converging  on  the open annulus $\mathrm{ann} (r,1)=\{X\in \Cp \vert r\leqslant \vert X\vert_p <1\}.$ These rings are equipped with a canonical Fr\'echet topology \cite{Ber02}. 
For each affinoid algebra $A$  over $E$ we define 
\begin{equation}
\nonumber
\CR^{(r)}_A=A\widehat\otimes_E \CR_E^{(r)}.
\end{equation}
We define the Robba ring over $A$ the ring $\CR_A=\underset{0\leqslant r<1}\cup \CR^{(r)}_A.$ Equip  $\CR_A$
 with a continuous action of $\Gamma$ and a Frobenius operator
$\Ph$ given by
\begin{equation}
\begin{aligned}
\nonumber
& \gamma (f(X))=f((1+X)^{\chi (\gamma )}-1),&& \gamma \in \Gamma,\\
& \Ph (f(X))=f((1+X)^{p}-1). &&
\end{aligned}
\end{equation}
In particular, 
\[
t=\underset{n=1}{\overset{\infty}\sum} (-1)^{n-1}\frac{X^n}{n} \in \CR_A,
\]
and we have $\Ph (t)=pt$ and $\gamma (t)=\chi (\gamma)t,$ $\gamma\in \Gamma.$

\subsubsection{} The operator $\Ph \,:\,\CR_A\rightarrow \CR_A$ has a left inverse  $\psi$ given by
\begin{equation}
\nonumber
\psi (f)=
\frac{1}{p}\Ph^{-1} \left (\underset{\zeta^p=1}\sum f(\zeta(1+X)-1)\right ),
\qquad f\in \CR_A.
\end{equation}
Set $\mathcal E_A=\CR_A\cap A[[X]].$ Then $\mathcal E_A^{\psi=0}$ is the free
$\mathcal H_A(\Gamma)$-submodule of $\mathcal E_A$ generated by 
$X+1.$ 
\newpage

\subsection{Cohomology of $(\Ph,\Gamma)$-modules}

\subsubsection{} In this section,  we use freely the theory of $(\Ph,\Gamma)$-modules over relative Robba rings  $\CR_A$ \cite{KPX}. 
If $\bD$ is a $(\Ph,\Gamma)$-module over $\CR_A,$ we set
$\CDcris (\bD)=\left (\bD [1/t] \right )^{\Gamma}.$ Then 
$\CDcris (\bD)$ is an $A$-module equipped with the induced action 
of $\Ph$  and a decreasing filtration $(\F^i\CDcris (\bD))_{i\in\Z}.$
For each $p$-adic representation $V$ of $G_{\Qp}$  with coefficients in $A$
we denote by $\DdagrigA (V)$ the associated $(\Ph, \Gamma)$-module 
 For any $(\Ph,\Gamma)$-module $\bD ,$ we denote by $H^i(\bD)$ the cohomology of the Fontaine--Herr
complex 
\begin{equation}
\nonumber 
\xymatrix{
C_{\Ph,\gamma_1}(\bD)\,\,:\,\,\bD^{\Delta} \ar[r]^(.6){d_0}&\bD^{\Delta}  \oplus \bD^{\Delta}  \ar[r]^(.6){d_1} &\bD^{\Delta},
}
\end{equation}
where $\gamma_1$ is a fixed generator of $\Gamma_1,$  $d_0(x)=((\Ph-1)x,  (\gamma_1-1)x)$ and $d_1(y,z)=(\gamma_1-1)y-  (\Ph-1)z.$ 
Let $\bD^*(\chi)=\Hom_{\CR_A}(\bD,\CR_A(\chi)) $ be the Tate dual of $\bD.$
We have a canonical pairing
\begin{equation}
\nonumber
\left <\,\,,\,\,\right >_{\bD}\,\,:\,\, H^1(\bD^*(\chi))\times H^1(\bD )\rightarrow 
H^2(\CR_A(\chi))\simeq A,
\end{equation}
which generalizes the classical local  duality. 

 The Iwasawa cohomology $H^*_{\Iw}(\bD )$ of $\bD$
is defined as the cohomology of the complex 
\begin{equation}
\nonumber
\bD \xrightarrow{\psi-1} \bD
\end{equation}
concentrated in degrees $1$ and $2.$ Let $\bD \widehat{\otimes}_A \CH_A (\Gamma)^{\iota}$ denote the tensor product $\bD \widehat{\otimes}_A \CH_A (\Gamma)$
which we consider  as a $(\Ph,\Gamma)$-module with the diagonal action of $\Gamma$
and the additional sructure of $\CH (\Gamma)$-module given by
\[
\gamma (d\otimes h)=d\otimes h\gamma^{-1}, \qquad d\in \bD, \,\, 
h\in \CH (\Gamma), \,\,\gamma\in\Gamma.
\]
There exists a canonical isomorphism of $\CH (\Gamma)$-modules
\begin{equation}
\label{isomorphism for iwasawa cohomology}
H^1_{\Iw}(\bD) \simeq 
H^1 \left (\bD \widehat{\otimes}_A \CH_A (\Gamma)^{\iota} \right )
\end{equation}
(see \cite[Theorem~4.4.8]{KPX}).
We remark 
that if $\bD=\DdagrigA (V)$ for a $p$-adic representation $V,$ then $H^i(\Qp, V)\simeq H^i(\DdagrigA (V))$ and $H^i_{\Iw}(\DdagrigA (V))\simeq H^1_{\Iw}(\Qp,V),$
where $H^1_{\Iw}(\Qp,V)$ denotes the usual Iwasawa cohomology of $V$
(see \cite{CC99} and \cite[Corollary~4.4.11]{KPX}).

For each $m\in \Z,$ the map $\spec^c_{-m}$ induces  a morphism
of $\Ph,\Gamma)$-modules 
\begin{equation}
\label{definition of cyclotomic specialization for modules}
\spec^{c}_{\bD, m}\,:\,\bD\,\widehat{\otimes}_A\CH_A(\Gamma)^{\iota}
\rightarrow 
\left (\bD\,\widehat{\otimes}_A\CH_A(\Gamma)^{\iota}\right )
\otimes_{\CH_A(\Gamma), \spec^c_{-m}} A \simeq \bD (\chi^{m}).
\end{equation}
Together with the isomorphism (\ref{isomorphism for iwasawa cohomology}), it induces homomorphisms on cohomology
\begin{equation}
\label{specialization of Iwasawa cohomology}
\spec^c_{\bD, m}\,:\, H^i_{\Iw}(\bD) \rightarrow H^i(\bD (\chi^{m})).
\end{equation}
In the remainder of this paper, we will often omit $\bD$ in notation.

\subsubsection{} 
We have a canonical 
$\CH_A(\Gamma)$-linear pairing
\begin{equation}
\label{definition of Iwasawa pairing}
\left \{\,\,,\,\,\right \}_{\bD}\,\,:\,\,H^1_{\Iw}(\bD^*(\chi))\times H^1_{\Iw}(\bD )^{\iota}
\rightarrow \CH_A(\Gamma)
\end{equation}
(see \cite[Definition~4.2.8]{KPX}).  It generalizes
the pairing in Iwasawa cohomology of $p$-adic representations \cite[Section~3.6]{PR94}. 

\begin{mylemma} i) The pairings $\left <\,\,,\,\,\right >_{\bD}$ and 
$\left \{\,\,,\,\,\right \}_{\bD}$ commute with the base change.


ii) The following diagram commutes
\begin{equation}
\nonumber
\xymatrix{
&H^1_{\Iw}(\bD^*(\chi))\times H^1_{\Iw}(\bD )^{\iota}
\ar[rr]^(.6){\left \{\,\,\,,\,\,\,\right \}_{\bD}} 
\ar[d]^{(\spec^c_{0}\,,\,\spec^c_{0})}
&&
\CH_A(\Gamma)
\ar[d]^{\spec^c_{0}}
\\
&H^1 (\bD^*(\chi))\times H^1 (\bD )
\ar[rr]^(.6){\left <\,\,\,,\,\,\,\right >_{\bD}}
&&
A.
}
\end{equation}
\end{mylemma}
\begin{proof} i)  follows immediately  from the definition of the pairings and 
ii)  is a particular case of \cite[Proposition~4.2.9]{KPX}.
\end{proof}

\subsubsection{} 
Let $\eta \,:\,\Gamma \rightarrow A^*$ be a continuous character.  We denote by 
\begin{equation}
\label{definition of Tw for Iwasawa cohomology}
\Tw_{\bD,\eta} \,:\, H^i_{\Iw}(\bD) \rightarrow H^i_{\Iw}(\bD (\eta))
\end{equation}
the isomorpism given by $d\mapsto d\otimes \eta.$ Note that it is 
not $\Gamma$-equivariant. If $\eta=\chi^m,$ where $\chi$ is the cyclotomic character, we write $\Tw_{\bD,m}$ instead $\Tw_{\bD,\chi^m}.$ Note that
\begin{equation}
\label{formula for spec}
\spec^c_{0}\circ \Tw_m=\spec^c_{m}.
\end{equation}

\begin{mylemma} 
\label{Lemma twisted Iwasawa pairing}
Let $\eta \,:\,\Gamma \rightarrow A^*$ be a continuous character.
Then 
\begin{equation}
\nonumber
\left \{\Tw_{\bD^*(\chi),\eta }(x),\Tw_{\bD ,\eta^{-1}}(y^{\iota}) \right \}_{\bD (\eta)}=
\Tw_{\eta^{-1}} \left \{ x,y^{\iota}\right \}_{\bD},
\end{equation}
where the twisting map in the right hand side is defined by  (\ref{definition of Tw_eta}). 
\end{mylemma}
\begin{proof} By \cite[Definition~4.2.8 and Lemma~4.2.5]{KPX}
\[
 \left \{ x,y^{\iota}\right \}_{\bD}=\left \{(\Ph-1)(x),(\Ph-1)(y^{\iota})
 \right \}_{\bD}^0,
 \]
where $\{\,\,,\,\,\}_{\bD}^0\,\,:\,\,\bD^*(\chi)^{\psi=0}\times \bD^{\psi=0, \iota}\rightarrow  \CR_A (\Gamma)$ is the unique $\CR (\Gamma)$-linear pairing satisfying the following condition: 
for all $x\in \bD^*(\chi)^{\psi=0}$ and $y\in \bD^{\psi=0}$ one has
\begin{equation}
\label{caracterization of auxiliary Iwasawa pairing}
\res \left (\{x,y^{\iota}\}^0_{\bD} \cdot \frac{d \gamma_1}{\gamma_1} \right )=
\log (\chi (\gamma_1))\cdot \res \left ( [x,y]_{\bD} \cdot \frac{dX}{1+X} \right ).
\end{equation}
Here $[\,\,,\,\,]_{\bD}\,:\,\bD^* (\chi)\times \bD \rightarrow \CR_A$
denotes the canonical pairing, and we refer the reader to \cite[Section~2.1]{KPX} for any unexplained notation. Let 
\[
\{ \,\,,\,\,\}'\,:\,\bD^*(\chi\eta)^{\psi=0}\times \bD (\eta^{-1})^{\psi=0, \iota}\rightarrow  \CR_A (\Gamma)
\]
denote the map $\{x,y\}'=\Tw_{\eta^{-1}}\left \{\Tw_{\bD^*(\chi\eta),\eta^{-1}}(x),\Tw_{\bD (\eta^{-1}),\eta}(y^{\iota}) \right \}^0_{\bD}.$ An easy computation shows that 
it is a $\CR_A(\Gamma)$-linear pairing which satisfies (\ref{caracterization of auxiliary Iwasawa pairing}) for the $(\Ph,\Gamma)$-module $\bD (\eta^{-1}).$
Therefore it coincides with $\{\,\,,\,\,\}_{\bD (\eta^{-1})}^0.$
This imples the lemma. 
\end{proof}

\subsection{The large exponential map}
\subsubsection{} Assume that $A=E.$ Let $\bD$ be a crystalline $(\Ph,\Gamma)$-module
over $\CR_E,$ {\it i.e.} $\dim_E \CDcris (\bD)=\rk_{\CR_E} \bD.$
We denote by $H^1_f(\bD)$ the subgroup of $H^1(\bD)$ that classifies 
crystalline extensions of the form $0\rightarrow \bD\rightarrow \mathbf{X}
\rightarrow \CR_E \rightarrow 0$ \cite[Section~1.4]{Ben11}.
The equivalence between the category of crystalline $(\Ph,\Gamma)$-modules
and that of filtered Dieudonn\'e modules \cite{Ber08} induces a canonical 
homomorphism
\begin{equation}
\nonumber
\exp_{\bD}\,\,:\,\,\CDcris (\bD)/\F^0 \CDcris (\bD) \rightarrow H^1_f(\bD),
\end{equation}
which is a direct generalization of the Bloch--Kato exponential map 
\cite[Proposition~1.4.4]{Ben11}, \cite{Nak14}. Note that $\exp_{\bD}$ is an isomorphism 
if $\CDcris (\bD)^{\Ph=1}=0.$

If $V$ is a crystalline representation of $G_{\Qp},$ then 
$\CDcris (\bD^{\dagger}_{\mathrm{rig},E}(V))\simeq \Dc (V),$ where $\Dc$ is  classical Fontaine's functor \cite{Fo94a, Fo94b}. We have a commutative diagram
\begin{equation}
\nonumber
\xymatrix{
\Dc (V)/\F^0\Dc (V) \ar[r]^(.6){\exp_V} \ar[d]^{\simeq} & H^1_f(\Qp,V) \ar[d]^{\simeq}
\\
\CDcris (\bD)/\F^0 \CDcris (\bD) \ar[r]^(.6){\exp_{\bD}} &H^1_f(\bD),
}
\end{equation}
where $\exp_V$ is the Bloch--Kato exponential map \cite{BK90}. 

\subsubsection{} Let $\bD$ be a $(\Ph,\Gamma)$-module over a Tate algebra $A.$
Assume that for all $x\in \mathrm{Spm}(A)$ the specialization $\bD_x=\bD\otimes_A A/\mathfrak{m}_x$ of $\bD$ at $x$ is a crystalline module. Then $\CDcris (\bD)$ is a projective $A$-module of rank $\rk_{\CR_A}(\bD)$ and $\CDcris (\bD_x)\simeq \CDcris (\bD)\otimes_A A/\mathfrak{m}_x$ \cite[Th\'eor\`eme C]{BCz}. Moreover, 
 Nakamura \cite[Section~2]{Nak17} constructed the relative version  $\exp_{\bD}\,:\,
\CDcris (\bD)/\F^0\CDcris (\bD) \rightarrow H^1(\bD)$ of the exponential map. 
For any $x\in \mathrm{Spm}(A)$ we have a commtative diagram  
\begin{equation}
\nonumber
\xymatrix{
\CDcris (\bD)/\F^0\CDcris (\bD) \ar[r]^(.6){\exp_{\bD}} \ar[d] & H^1(\bD) \ar[d]
\\
\CDcris (\bD_x)/\F^0 \CDcris (\bD_x) \ar[r]^(.6){\exp_{\bD_x}} &H^1(\bD_x).
}
\end{equation}


\subsubsection{}
\label{subsection modules rank 1}
Let $\bdelta \,:\,\Qp^*\rightarrow A^*$ be a continuous character with values 
in a Tate algebra $A.$ We denote by $\CR_A(\bdelta )$ the  $(\Ph,\Gamma)$-module $\CR_A\cdot e_{\boldsymbol\delta}$  of rank $1$ over $\CR_A$ defined  by  
\[
\varphi (e_{\bdelta})=\delta (p)\cdot e_{\bdelta},\qquad  
\gamma (e_{\bdelta})=\bdelta (\chi (\gamma))\cdot e_{\bdelta},\quad \gamma \in \Gamma.
\]
In Sections \ref{subsection modules rank 1}-\ref{them:propertiestwovarPRlog}, we assume that 
\[
\bdelta \vert_{\mathbb Z_p^*} (u)=u^m \quad \text{ for some integer $ m\geqslant 1.$}
\]
Then the crystalline module $\CDcris (\CR_A(\bdelta ))$
associated to $\CR_A(\boldsymbol\delta )$ is the free $A$-module of rank $1$ generated by $d_\delta=t^{-m}e_{\bdelta}.$ The action of $\Ph$ 
on $d_{\bdelta}$ is given by 
\[
\varphi (d_{\boldsymbol\delta})=p^{-m}\boldsymbol\delta (p)d_{\boldsymbol\delta}.
\]
Moreover, $\F^0\CDcris (\CR_A(\bdelta ))=0,$ and the exponential map takes the form 
\begin{equation}
\nonumber
\exp_{\CR_A(\bdelta )}\,:\,\CDcris (\CR_A(\bdelta )) \rightarrow 
H^1(\CR_A(\bdelta )).
\end{equation}


\subsubsection{} 
\label{subsubsection large logarithm}
We review the construction of the large exponential map for $(\Ph, \Gamma)$-modules of rank one.
We refer the reader to  \cite{Nak17} for general constructions and more detail.
Equip the ring $\mathcal E_A=\CR_A\cap A[[X]]$ with the operator $\partial =(1+X)\,\displaystyle\frac{d}{dX}.$
Let $z\in \CDcris (\CR_A(\boldsymbol\delta ))\otimes_A \mathcal E_A^{,\psi=0}.$
It may be shown that the equation
\[
(\varphi-1) F=z-\frac{\partial^mz(0)}{m!}t^m
\]
has a solution in $\CDcris (\CR_A(\bdelta ))\otimes_A \CR_A$ and we define
\begin{equation}
\nonumber
\textup{Exp}_{\CR_A (\boldsymbol\delta)}(z)=
(-1)^m\frac{\log \chi (\gamma_1)}{p}\,t^m\partial^m (F).
\end{equation}
Exactly as in the classical case $A=E$ (see \cite{Ber03}), it is not hard to check that  $\textup{Exp}_{\CR_A (\boldsymbol\delta)}(z) \in \CR_A (\boldsymbol\delta )^{\psi=1}\simeq H^1_{\Iw}(\CR_A(\delta))$ and we denote by 
\begin{equation*}
\textup{Exp}_{\CR (\bdelta)}\,:\,\CDcris (\CR_A(\bdelta ))\otimes_A\mathcal E_A^{\psi=0} \rightarrow H^1_{\Iw}(\CR_A(\bdelta ))
\end{equation*}
the resulting map \cite{Ber03, Nak14, Nak17}.
Let $c\in \Gamma$ denote the unique element such that $\chi (c)=-1.$ 
Set $\textup{Exp}^c_{\CR (\bdelta)}=c \circ \textup{Exp}_{\CR (\bdelta)}.$
For any generator $d \in  \CDcris (\CR_A(\bdelta ))$  define 
\begin{equation}
\nonumber
\begin{aligned}
&\frak{Log}_{\CR_A(\bdelta^{-1}\chi ),d}\,:\, H^1_{\Iw}(\CR_A(\bdelta^{-1}\chi ))\rightarrow \CH_A(\Gamma),\\
&\frak{Log}_{\CR_A(\bdelta^{-1}\chi ),d}(x)=\left \{x, \Exp^c_{\CR_A (\boldsymbol\delta)}(d\otimes (1+X)^{\iota})\right \}_{\CR_A(\bdelta)}.
\end{aligned}
\end{equation}

\begin{myproposition}
\label{them:propertiestwovarPRlog}
\textup{1)} The maps $\textup{Exp}_{\CR_A(\bdelta ) }$ and $\frak{Log}_{\CR_A(\bdelta^{-1}\chi ),d}$
commute with the base change.

\textup{2)} Let  $A=E$ and let $V$ be a crystalline representation of $G_{\Qp}.$
The choice of a compatible system $\ep=(\zeta_{p^n})_{n\geqslant 0}$ of $p^n$th roots of unity fixes 
an isomorphism between $\CR_{\Qp}$ and the ring $\mathbf B^{\dagger}_{\mathrm{rig}, \Qp}$ from \cite{Ber02}.
Assume that $\CR_E(\delta)$ is a submodule of $\bD^{\dagger}_{\mathrm{rig},E}(V).$ Then $\textup{Exp}_{\CR_E(\delta ) }$ 
coincides with the restriction of Perrin-Riou's large exponential map \cite{PR94}
\[
\textup{Exp}_{V,m}^{\varepsilon}\,:\,\Dc (V)\otimes_E\mathcal E_E^{\psi=0} \rightarrow H^1_{\Iw}(\Qp,V)
\]
on $\CDcris (\CR_E(\delta) )\otimes_E \mathcal E_E^{\psi=0}.$

\textup{3)} Let  $k\in \Z$ be an integer   such that $k+m\geqslant 1$
and $p^{-k-m}\bdelta (p) -1$ does not vanish on $A.$ Then 
\[
\spec^c_{k}\circ  \frak{Log}_{\CR_A(\bdelta^{-1}\chi )  ,d}(x)=
(m+k-1)! \cdot \frac{1-p^{m+k-1}\bdelta(p)^{-1}}{1-p^{-m-k}\bdelta (p)}\cdot 
\left <\spec^c_{-k}(x), \exp_{\CR_A(\bdelta \chi^k )} (d [k])\right >_{\CR_A(\bdelta \chi^k)}, 
\]  
where we denote by  $d[k]$ the image of $d$ under the canonical shift
$\CDcris (\CR_A(\bdelta)) \rightarrow  \CDcris (\CR_A(\bdelta \chi^k)).$

\textup{4)} Let $A=E.$ Then for any $k\in \Z$ such that $k+m\leqslant 0$ 
and $p^{-k-m}\bdelta (p) \neq 1$     one has 
\[
\spec^c_{k}\circ  \frak{Log}_{\CR_E(\bdelta^{-1}\chi )  ,d}(x)=
\frac{(-1)^{m+k}}{(-m-k)!} \cdot \frac{1-p^{m+k-1}\bdelta(p)^{-1}}{1-p^{-m-k}\bdelta (p)}\cdot 
\left [ \log_{\CR_E(\bdelta^{-1}\chi^{1-k})}\left (\spec^c_{-k}(x)\right ), d [k]\right ]_{\CR_E(\bdelta \chi^k)}, 
\]  
where 
\[
[\,\,,\,\,]_{\CR_E(\bdelta\chi^k)}\,:\, \CDcris \left (\CR_E(\bdelta^{-1}\chi^{1-k} )\right )\times 
\CDcris \left (\CR_E (\bdelta\chi^k)\right ) \rightarrow E
\]
denotes the canonical pairing.

\end{myproposition}

\begin{proof} Part  1) is clear. Part 2)  follows from Berger's construction of the large exponential map \cite{Ber03}. Part  3) is essentially the interpolation property  of the large exponential map (see \cite{PR94} and \cite[Corollaire 4.10]{BB08}. Part 4) is equivalent to Perrin-Riou's explicit reciprocity law. See  \cite{Ben00, Ber03,  Cz98} for the proofs  in the cassical case of absolutely crystalline representations.  The case  of $(\Ph,\Gamma)$-modules of rank one over an unramified field is particularly  simple and can be treated by the method of Berger \cite{Ber03} without  additional difficulties.   It also  can be deduced
from the  results of Nakamura \cite{Nak14}, where the approach of Berger 
 was  extendend to  general de Rham $(\Ph,\Gamma)$-modules. 
\end{proof}

\begin{mycorollary}
\label{corollary large exponential}
 We record the particular cases that will be used in this paper:
\[
\nonumber
\begin{aligned}
\nonumber
&\spec^c_{0}\circ  \frak{Log}_{\CR_A(\bdelta^{-1}\chi )  ,d}(x)=
(m-1)! \cdot \frac{1-p^{m-1}\bdelta(p)^{-1}}{1-p^{-m}\bdelta (p)}\cdot 
\left <\spec^c_{0}(x), \exp_{\CR_A(\bdelta)} (d)\right >_{\CR_A(\bdelta )}, \\
&
\spec^c_{-m}\circ  \frak{Log}_{\CR_E(\bdelta^{-1}\chi )  ,d}(x)=
 \frac{1-p^{-1}\bdelta(p)^{-1}}{1-\bdelta (p)}\cdot 
\left [ \log_{\CR_E(\bdelta^{-1}\chi^{m+1})}\left (\spec^c_{m}(x)\right ), d [-m]\right ]_{\CR_E(\bdelta \chi^{-m})}. 
\end{aligned}
\]  
\end{mycorollary}

We also need the following technical result. 

\begin{myproposition}
\label{proposition cohomology of rank 1 modules in families}
Let $m\in \Z$ and let $\bdelta \,:\,\Qp^*\rightarrow A$ be a continuous character 
such that 
\[
\bdelta (u)=u^m, \qquad u\in \Zp^*.
\]
Then the following statements hold: 

1) If $m\geqslant 1$ or $\bdelta (p)p^{-m}\neq 1,$ then $H^0(\CR_A(\bdelta))=0.$

2a) If $A$ is a principal ideal domain and $m\leqslant 0,$ then 
$H^2(\CR_A(\bdelta))=0.$

2b) If, in addition, $p^{-m}\bdelta (p)-1$ is invertible in $A$ then $H^1(\CR_A(\bdelta))$
is a free $A$-module of rank one.
\end{myproposition}
\begin{proof} 1) If $m\geqslant 1,$ then $\CR_A(\bdelta)^{\Gamma}=0$
and therefore  $H^0(\CR_A(\bdelta))=0.$ Assume that $m\leqslant 0.$
Then  $\CDcris (\CR_A(\bdelta))=\CR_A(\bdelta)^{\Gamma}$ is generated by $d_{\bdelta}=t^{-m}e_{\bdelta},$ where 
$
\Ph (d_{\bdelta})=p^{-m}\bdelta (p) d_{\bdelta}.
$
Therefore
\[
H^0(\CR_A(\bdelta))= \CDcris (\CR_A(\bdelta))^{\Ph=1}=0.
\]

2a) For any bounded complex $C^\bullet$ of $A$-modules and any 
$A$-module $M$ one has a spectral sequence
\[
E_2^{i,j}:=\Tor^A_{-i}(H^j(C^\bullet),M) \Rightarrow H^{i+j}(C^\bullet \otimes_A M).
\]
In paticular, for any maximal ideal $\mathfrak m_x$ of $A$ one has
\begin{equation}
\label{spectral sequence (phi-Gamma)-modules in families}
E_2^{i,j}:=\Tor^A_{-i}\left (H^j(\CR_A(\bdelta), k(x)\right ) \Rightarrow H^{i+j}\left (\CR_{k(x)}(\bdelta_x)\right ),
\end{equation}
where $k(x)=A/{\mathfrak m_x}.$ 
In degree 2, this gives  isomorphisms
\[
H^2(\CR_A(\bdelta))\otimes_A k(x)\simeq 
H^{2}\left (\CR_{k(x)}(\bdelta_x)\right ), \qquad \forall x\in \Spm (A).
\]
By local duality and part 1) of the proposition, 
\[
H^{2}\left (\CR_{k(x)}(\bdelta_x)\right )=
H^0\left (\CR_{k(x)}(\bdelta^{-1}_x\chi )\right )=0,
\]
and therefore $H^2(\CR_A(\bdelta))\otimes_A k(x)=0$ for all $x\in \Spm (A).$
Since $A$ is a principal ring, this implies that $H^2(\CR_A(\bdelta))=0.$

2b) The spectral sequence (\ref{spectral sequence (phi-Gamma)-modules in families})
together with 2a) gives 
\[
H^1(\CR_A(\bdelta))\otimes_A k(x)\simeq 
H^{1}\left (\CR_{k(x)}(\bdelta_x)\right ), \qquad \forall x\in \Spm (A).
\]
From our assumptions, it follows that $H^{0}\left (\CR_{k(x)}(\bdelta_x)\right )=
H^{2}\left (\CR_{k(x)}(\bdelta_x)\right )=0$ and by the Euler characteristic formula
$\dim_{k(x)}H^{1}\left (\CR_{k(x)}(\bdelta_x)\right )=1$ for all $x\in \Spm (A).$
Thus, $H^1(\CR_A(\bdelta))\otimes_A k(x)$ is a $k(x)$-vector space of dimension
$1$ for all $x\in \Spm (A).$ Since  $H^1(\CR_A(\bdelta))$ is a finitely generated 
module over the principal ideal domain $A,$ this implies that 
$H^1(\CR_A(\bdelta))$ is free of rank one over $A.$ The proposition is proved. 
\end{proof}

\section{Complements on  the $\mathcal L$-invariant}

\subsection{Regular submodules}
\label{subsection regilar submodules}
\subsubsection{} 
In this section,  we first review the definition of the $\mathcal L$-invariant 
of $p$-adic representations of motivic weight $-2$ proposed in \cite{Ben14}.
For the purposes of this paper, it is more convenient to use another 
construction, which we describe in Section~\ref{subsection second definition of L-invariant}. We next show that the two definitions are equivalent.  

Fix a prime number $p$ and  a finite set $S$ of primes of $\Q$ 
containing $p.$ We denote by $G_{\Q,S}$ the Galois group
of the maximal algebraic extension of $\Q$ unramified outside $S\cup \{\infty\}.$
Let $V$ be a $p$-adic representation of 
$G_{\Q,S}$ with coefficients in a finite extension $E$ of $\Qp.$
We write $H^*_S(\Q,V)$ for the continuous cohomology of $G_{\Q,S}$ with coefficients in $V.$  Recall that the Bloch--Kato Selmer group $H^1_f(\Q,V)$ is defined as
\begin{equation}
\label{definition of Bloch-Kato Selmer}
H^1_f(\BQ,V)=\ker \left (H^1_S(\BQ,V)\rightarrow \underset{l\in S}{\bigoplus}
\frac{H^1(\BQ_l,V)}{H^1_f(\BQ_l,V)}\right ),
\end{equation}
where the "local conditions" $H^1_f(\BQ_l,V)$ are given by 
\begin{equation}
\label{Bloch-Kato local conditions}
H^1_f(\BQ_l,V)=\begin{cases} 
\ker (H^1(\BQ_l,V)\rightarrow H^1(I_l,V))  &\text{\rm if $l\neq  p$},
\\
\ker (H^1(\BQ_p,V) \rightarrow H^1(\BQ_p,V\otimes\Bc))  &\text{\rm if $l= p$} 
\end{cases}
\end{equation}
(see \cite{BK90}).
Here  $\Bc$ is  Fontaine's ring of crystalline periods  \cite{Fo94a}.
We also consider the relaxed Selmer group 
\begin{equation*}
\label{definition of relaxed Bloch-Kato Selmer}
H^1_{f, \{p\}}(\BQ,V)=\ker \left (H^1_S(\BQ,V)\rightarrow \underset{l\in S\setminus\{p\}}{\bigoplus}
\frac{H^1(\BQ_l,V)}{H^1_f(\BQ_l,V)}\right ).
\end{equation*}
The Poitou--Tate exact sequence induces an exact sequence
\begin{equation}
\label{modified Poitou-Tate sequence}
0\rightarrow H^1_f(\Q,V)\rightarrow H^1_{f, \{p\}}(\Q,V)
\rightarrow \frac{H^1(\Qp,V)}{H^1_f(\Qp,V)}
\rightarrow H^1_f(\Q, V^*(1))^*
\end{equation}
(see \cite[Proposition~2.2.1]{FPR94} and \cite[Lemme~3.3.6]{PR95}).

\subsubsection{}In the rest of this section,  we assume that $V$ 
satisfies the following conditions:
\begin{itemize}
\item[]{\bf C1)} $H^0_S(\Q,V)=H^0_S(\Q, V^*(1))=0.$ 

\item[]{\bf C2)} $V$ is crystalline at $p$ and  $\Dc (V)^{\Ph=1}=0.$

\item[]{\bf C3)} $\Ph \,:\,\Dc (V) \rightarrow \Dc (V)$
is semisimple at $p^{-1}$.

\item[]{\bf C4)}   $H^1_f(\BQ, V^*(1))=0.$

\item[]{\bf C5)}  The localization map  $\res_p\,:\,H^1_f(\Q, V)\rightarrow H^1_f(\Qp,V)$
is injective. 
\end{itemize}
We refer the reader to \cite{Ben15} for a discussion of these conditions. 
Here we only remark that if $V$ is the $p$-adic realization of a pure motive
of weight $\leqslant -2$ having  good reduction at $p,$ then 
{\bf C3--5)} are deep conjectures which are known only in some special cases.

From {\bf C2)} it follows that the exponential map 
$\Dc (V)/\F^0\Dc (V)\rightarrow H^1_f(\Qp,V)$
is an isomorphism, and we denote by $\log_V$ its inverse. Compositing $\log_V$ with
the localization map $H^1_f(\Q,V)\rightarrow H^1_f(\Qp,V),$  we obtain a map
\[
r_V\,:\,H^1_f(\Q,V) \rightarrow \Dc (V)/\F^0\Dc (V)
\]
which is closely related to the syntomic regulator.

\subsubsection{}We introduce the notion of  regular submodule, which first appeared 
in  Perrin-Riou's book  \cite{PR95} in the context of crystalline representations 
(see also \cite{Ben11}). 

\begin{mydefinition}[\sc Perrin-Riou] Assume that $V$ is a $p$-adic representation
which satisfies the conditions {\bf C1--5)}. 

i) A $\Ph$-submodule $D$ of $\Dc (V)$ is regular if $D\cap \F^0\Dc (V)=0$
and the map
\[
r_{V,D}\,:\, H^1_f(\Q,V)\rightarrow \Dc (V)/(\F^0\Dc (V)+D),
\]
induced by $r_V,$ is an isomorphism.

ii) A $\Ph$-submodule $D$ of $\Dc (V^*(1))$ is regular if 

\[D+\F^0\Dc (V^*(1))=\Dc (V^*(1))
\]
 and the map
\[
D\cap \F^0\Dc (V^*(1))\rightarrow H^1_f(\Q,V)^*,
\]
induced by the dual map  $r_V^*\,:\,\F^0\Dc (V^*(1))\rightarrow H^1_f(V)^*,$ is an 
isomorphism. 
\end{mydefinition}
\noindent
\begin{myremark}   1) Assume that $H^1_f(\Q,V)=H^1_f(\Q,V^*(1))=0.$ Then $D$ is regular 
if the canonical projection $D\rightarrow t_V(\Qp)$ is an isomorphism of vector
spaces, and our definition agrees with the definition given in \cite{Ben11}.

2)  A $\Ph$-submodule $D$ of $\Dc (V)$ is regular if and only if 
\begin{equation}
\label{decomposition of H^1_f using H^1_f(D)}
H^1_f(\Qp,V)=\res_p \left (H^1_f(\Q,V)\right )\oplus H^1_f(\bD),
\end{equation}
where $\bD$ is the $(\Ph,\Gamma)$-submodule of $\DdagrigE (V)$
associated to $D$ by the theory of Berger \cite{Ber08} (see \cite[Section~3.1.3]{Ben14}). 
\end{myremark}

\subsection[]{The $\mathcal L$-invariant}
\label{The L-invariant}
\subsubsection{}
Let $D\subset \Dc (V)$ be a regular submodule. From the semisimplicity 
of $\Ph$ it follows that, as a $\Ph$-module,  $D$ decomposes into the direct sum
\[
D=D_{-1}\oplus D^{\Ph=p^{-1}},\qquad D_{-1}^{\Ph=p^{-1}}=0,
\]
which gives a four step filtration
\begin{equation}
\label{definition of L-inv filtration}
\{0\}\subset D_{-1}\subset D \subset \Dc (V).
\end{equation}
Let $F_0\Ddagrig (V)$ and $F_{-1}\Ddagrig (V)$  denote  the $(\Ph,\Gamma)$-submodules 
of $\DdagrigE (V)$ associated  
to $D$ and $D_{-1}$ by Berger \cite{Ber08}, thus 
\[
D=\CDcris \left ( F_0\Ddagrig (V)\right ),\qquad 
D_{-1}=\CDcris \left (F_{-1}\Ddagrig (V)\right ).
\]
Set $\bM=\mathrm{gr}_0\Ddagrig (V)$ and $W=D/D_{-1}\simeq \CDcris (\bM).$  
The $(\Ph,\Gamma)$-module $\bM$ satisfies 
\[
\F^0\CDcris (\bM)=0, \qquad \CDcris (\bM)^{\Ph=p^{-1}}=\CDcris (\bM).
\]
The cohomology of such  modules was studied
in detail in \cite[Proposition 1.5.9 and Section 1.5.10]{Ben11}.
In particular, there exists a canonical decomposition of $H^1(\bM)$ into the direct sum of $H^1_f(\bM)$ and some canonical space $H^1_c(\bM)$
\begin{equation}
\label{decomposition of H^1(W)}
H^1(\bM)=H^1_f(\bM)\oplus H^1_c(\bM).
\end{equation}
Moreover,  there are canonical isomorphisms
\begin{equation}
\label{isomorphisms for H^1_f(W) and H^1c(W)}
i_{\bM,f}\,\,:\,\,\CDcris (\bM)\simeq H^1_f(\bM), \qquad i_{\bM,c}\,\,:\,\,\CDcris (\bM)\simeq H^1_c(\bM)
\end{equation}
(see \cite[Proposition 1.5.9]{Ben11}).
 
\subsubsection{} We have a tautological exact sequence

\begin{equation*}
0\rightarrow  F_{-1}\Ddagrig (V)\rightarrow  F_{0}\Ddagrig (V)\rightarrow \bM
\rightarrow 0
\end{equation*}
which induces exact sequences \cite[Section~3.1.5]{Ben14}
\begin{align}
\nonumber
&0\rightarrow H^1(F_{-1}\Ddagrig (V))\rightarrow H^1(F_{0}\Ddagrig (V))\rightarrow H^1(\bM)
\rightarrow 0,\\
\nonumber
&0\rightarrow H^1_f(F_{-1}\Ddagrig (V))\rightarrow H^1_f(F_{0}\Ddagrig (V))\rightarrow H^1_f(\bM)\rightarrow 0.
\end{align}
Moreover, $H^1_f(F_{-1}\Ddagrig (V))=H^1(F_{-1}\Ddagrig (V)),$ and we have 
\begin{equation}
\label{isomorphism for H^1/H^1_f}
\frac{H^1(F_{0}\Ddagrig (V))}{H^1_f(F_{0}\Ddagrig (V)} \simeq
\frac{H^1(\bM)}{H^1_f(\bM)}.
\end{equation}

\subsubsection{} From  {\bf C5)} it follows that the localisation map
$
H^1_{f,\{p\}}(\Q, V) \rightarrow H^1(\Qp,V)
$
is  injective. Let 
\begin{equation*}
\kappa_D\,\,:\,\,H^1_{f,\{p\}}(\Q, V) \rightarrow \frac{H^1(\Qp,V)}{H^1_f(F_{0}\Ddagrig (V))}
\end{equation*}
denote the composition of this map with the canonical projection.
Then $\kappa_D$ is an isomorphism \cite[Lemma~3.1.4]{Ben14}

Let $ H^1(D,V)$ denote the inverse image of
$H^1(F_{0}\Ddagrig (V))/H^1_f(F_{0}\Ddagrig (V))$ under $ \kappa_D.$ Then
$\kappa_D$ induces an isomorphism
\[
H^1(D,V) \simeq \frac{H^1(F_{0}\Ddagrig (V))}{H^1_f(F_{0}\Ddagrig (V))}.
\]
Consider the composition map
$
H^1(D,V)\rightarrow H^1(F_{0}\Ddagrig (V)) \rightarrow H^1(\bM).
$
From (\ref{isomorphism for H^1/H^1_f}) it follows that 
\begin{equation}
\label{isomorphism for H^1(D,V)}
H^1(D,V) \simeq \frac{H^1(\bM)}{H^1_f(\bM)}
\end{equation}
is an isomorphism. Taking into account (\ref{decomposition of H^1(W)})
and (\ref{isomorphisms for H^1_f(W) and H^1c(W)}) we obtain that 
$\dim_E H^1(D,V)=\dim \CDcris (\bM).$ 
Hence we have a diagram
\[
\xymatrix{
\CDcris (\bM) \ar[r]^{\overset{i_{\bM,f}}\sim} &H^1_f(\bM)\\
H^1(D,V) \ar[u]^{\rho_{D,f}} \ar[r] \ar[d]_{\rho_{D,c}} 
&H^1(\bM) \ar[u]_{p_{\bM,f}}
\ar[d]^{p_{\bM,c}}
\\
\CDcris (\bM) \ar[r]^{\overset{i_{\bM,c}}\sim} &H^1_c(\bM),}
\]
where $\rho_{D,f}$ and $\rho_{D,c}$ are defined as the unique maps making
this diagram commute.
From (\ref{isomorphism for H^1(D,V)})  it follows that  $\rho_{D,c}$ is an isomorphism.

\begin{mydefinition} The determinant
\[
\mathscr L (V,D)= \det \left ( \rho_{D,f} \circ \rho^{-1}_{D,c}\,\mid \,\CDcris (\bM) \right )
\]
will be called  the $\mathscr L$-invariant associated to $V$ and $D$.
\end{mydefinition}

\subsection{The dual construction}
\label{subsection second definition of L-invariant}

\subsubsection{} Let $D^{\perp}=\Hom_E \left (\Dc (V)/D, \Dc (E(1))\right ).$
Then $D^\perp$ is a regular submodule of $\Dc (V^*(1)).$
In this section, we define an $\mathcal L$-invariant $\mathcal L(V^*(1),D^{\perp})$
associated to the data $(V^*(1),D^\perp).$

Let $D^{\perp}_1$ denote the biggest $\Ph$-submodule of $\Dc (V^*(1))$
such that $(D^{\perp}_1/D^{\perp})^{\Ph=1}=D^{\perp}_1/D^{\perp}.$
This gives a four step filtration 
\begin{equation}
\label{definition L-inv dual filtration}
\{0\}\subset D^\perp \subset D_1^\perp \subset \Dc (V^*(1)).
\end{equation}
It follows from definition that $D^{\perp}_1= \left (D_{-1}\right )^{\perp}=
\Hom_E (\Dc (V)/D_{-1}, \Dc (E(1))),$ and therefore  
(\ref{definition L-inv dual filtration}) is the  dual filtration of (\ref{definition of L-inv filtration}).

We denote by $F_0\Ddagrig (V^*(1))$ and $F_1\Ddagrig (V^*(1))$ the $(\Ph, \Gamma)$-submodules
of $\DdagrigE (V^*(1))$ associated to $D^{\perp}$
and $D^{\perp}_1,$ thus 
\begin{equation*}
\CDcris \left (F_0\Ddagrig (V^*(1))\right ) \simeq D^{\perp}, \qquad 
\CDcris \left (F_1\Ddagrig (V^*(1))\right ) \simeq D^{\perp}_1.
\end{equation*}
Then 
\begin{align*}
&F_0\Ddagrig (V^*(1))=\Hom_{\CR_E} \left (\DdagrigE (V)/F_0\DdagrigE (V), \CR_E(\chi)\right ),\\
&F_1\Ddagrig (V^*(1))=\Hom_{\CR_E} \left (\DdagrigE (V)/F_{-1}\DdagrigE (V) , \CR_E(\chi)\right ),
\end{align*}
and we have an exact sequence 
\begin{equation}
\label{second definition of L-inv tautological sequence}
0\rightarrow F_0\Ddagrig (V^*(1))\rightarrow F_1\Ddagrig (V^*(1)) \rightarrow \bM^*(\chi)
\rightarrow 0,
\end{equation}
where $\bM$ is the $(\Ph,\Gamma)$-module defined in Section~\ref{The L-invariant}.
Note that 
\begin{equation}
\label{conditions for exceptional (phi,Gamma)-module}
\F^0\CDcris (\bM^*(\chi))=\CDcris (\bM^*(\chi)), \qquad
\CDcris (\bM^*(\chi))^{\Ph=1}=\CDcris (\bM^*(\chi)).
\end{equation}
We refer the reader to \cite[Proposition~1.5.9 and Section~1.5.10]{Ben11}
for the proofs of the following facts. 

\begin{itemize}
\begin{item}[a)]
The map
\begin{equation}
\label{construction of i_(W^*(chi))}
\begin{aligned}
&i_{\bM^*(\chi)}\,\,:\,\,\CDcris (\bM^*(\chi))\oplus \CDcris (\bM^*(\chi))
\rightarrow H^1(\bM^*(\chi)),\\
&i_{\bM^*(\chi)}(x,y)=\cl (-x,\, \frac{\log \chi (\gamma_1)}{p}\, y)
\end{aligned}
\end{equation}
is an isomorphism.
\end{item}

\begin{item}[b)]
Denote by $i_{\bM^*(\chi),f}$ (respectively $i_{\bM^*(\chi),c}$)
the  restriction of $i_{\bM^*(\chi)}$ on the first (respectively second) summand. 
Then $\im \left (i_{\bM^*(\chi),f}\right )=H^1_f(\bM^*(\chi))$ and 
\begin{equation}
\label{decomposition of W^*(chi)}
H^1(\bM^*(\chi)) \simeq H^1_f(\bM^*(\chi))\oplus H^1_c(\bM^*(\chi)),
\end{equation}
where $H^1_c(\bM^*(\chi))=\im \left (i_{\bM^*(\chi),c}\right ).$
\end{item}

\begin{item}[c)]
The local duality
\begin{equation}
\nonumber
\left < \,\,,\,\,\right >_{\bM}\,:\, H^1(\bM^*(\chi))\times H^1(\bM )
\rightarrow E
\end{equation}
has the following explicit description:
\begin{equation}
\label{local duality for exceptional modules}
\left < i_{\bM^*(\chi)}(\alpha,\beta), i_{\bM}(x,y) \right >_{\bM}=
\left [\beta,x\right ]_{\bM}-\left [\alpha,y\right ]_{\bM}, \qquad
\end{equation}
for all $\alpha,\beta \in \CDcris (\bM^*(\chi))$ and 
$x,y \in \CDcris (\bM).$
\end{item}
\end{itemize}

\begin{mylemma} i) The following sequences are exact 
\begin{equation*}
\begin{aligned}
\nonumber
& 0\rightarrow H^0\left (F_0\Ddagrig (V^*(1))\right ) \rightarrow 
H^0\left (F_1\Ddagrig (V^*(1))\right )
\rightarrow H^0\left (\bM^*(\chi)\right ) \rightarrow 0,\\
\nonumber
& 0\rightarrow H^0\left (F_1\Ddagrig (V^*(1))\right ) \rightarrow 
H^0\left (\DdagrigE (V^*(1))\right ) 
\rightarrow 
H^0\left (\mathrm{gr}_2\DdagrigE (V^*(1))\right ) \rightarrow 0.
\end{aligned}
\end{equation*}

ii) We have a commutative diagram with exact rows
\begin{equation}
\label{dual L-invariant diagram from lemma}
\xymatrix
{0\ar[r]  &H^1_f\left (F_0\Ddagrig (V^*(1))\right )\ar[r] \ar@{^{(}->}[d] &H^1_f
\left (F_1\Ddagrig (V^*(1))\right )
\ar[r] \ar@{^{(}->}[d] &H^1_f\left (\bM^*(\chi)\right ) \ar[r] \ar@{^{(}->}[d]&0,\\
0\ar[r] &H^1 \left (F_0\Ddagrig (V^*(1))\right ) \ar[r] 
&H^1 \left (F_1\Ddagrig (V^*(1))\right )
\ar[r] &H^1 \left (\bM^*(\chi)\right ) \ar[r] &0.
}
\end{equation}

iii) The natural map $H^1\left (F_1\Ddagrig (V^*(1))\right )\rightarrow 
H^1\left (\DdagrigE (V^*(1))\right )$ is injective and 
\[
H^1_f\left (F_1\Ddagrig (V^*(1))\right )= H^1_f\left (\DdagrigE (V^*(1))\right ).
\]

iv) There is a canonical isomorphism
\begin{equation}
\label{dual L-invariant first isomorphism}
\frac{H^1 \left (F_1\Ddagrig (V^*(1))\right )}{H^1_f \left (\DdagrigE (V^*(1) \right )+H^1 \left (F_0\Ddagrig (V^*(1)) \right )}
\simeq \frac{H^1 (\bM^*(\chi))}{H^1_f (\bM^*(\chi))}.
\end{equation}
\end{mylemma}
\begin{proof} i) Since the category of crystalline $(\Ph,\Gamma )$-modules is equivalent to the category of filtered Dieudonn\'e modules \cite{Ber08}, the exact sequence (\ref{second definition of L-inv tautological sequence})
induces an exact  sequence 
\[
0\rightarrow \F^0\CDcris \left (F_0\Ddagrig (V^*(1))\right )\rightarrow \F^0\CDcris \left (F_1\Ddagrig (V^*(1))\right )\rightarrow
\F^0\CDcris \left ( \bM^*(\chi)\right )
\rightarrow 0 .\]
The semisimplicity of $\Ph$ implies that the sequence 
\[
0\rightarrow \F^0\CDcris \left (F_0\Ddagrig (V^*(1))\right )^{\Ph=1}\rightarrow \F^0\CDcris \left (F_1\Ddagrig (V^*(1))\right )^{\Ph=1}\rightarrow
\F^0\CDcris \left ( \bM^*(\chi)\right )^{\Ph=1}
\rightarrow 0 \]
is exact. Since $H^0(\mathbf X)=\F^0\CDcris (\mathbf X)^{\Ph=1}$ for any crystalline $(\Ph,\Gamma)$-module $\mathbf X$ \cite[Proposition~1.4.4]{Ben11}, 
the first sequence is exact. The exactness of the second sequence 
can be proved by the same argument.

ii) We only need to prove that the rows of the diagram 
(\ref{dual L-invariant diagram from lemma})  are exact. 
From i) and the long exact cohomology sequence 
associated to (\ref{second definition of L-inv tautological sequence}) we 
obtain an exact sequence 
\[
0\rightarrow  H^1 \left (F_0\Ddagrig (V^*(1))\right ) \rightarrow 
H^1 \left (F_1\Ddagrig (V^*(1))\right )
\rightarrow H^1 \left (\bM^*(\chi)\right ) \rightarrow H^2\left(F_0\Ddagrig (V^*(1))\right ).
\]
Condition {\bf C2)} implies that 
\[
H^0 \left (\DdagrigE (V)/F_0\Ddagrig (V) \right )=
\F^0\CDcris  \left (\DdagrigE (V)/F_0\Ddagrig (V) \right )^{\Ph=1}=0,
\]
and by Poincar\'e duality for $(\Ph,\Gamma)$-modules, we have
\begin{equation*}
\label{H^2(Qp, bD^perp)=0}
H^2\left (F_0\Ddagrig (V^*(1))\right )=H^0 \left (\DdagrigE (V)/ F_0\Ddagrig (V)\right )^*=0.
\end{equation*}
The exactness of the bottom row is proved. 
The exactness of the upper row follows from i) and \cite[Corollary~1.4.6]{Ben11}.

iii) The long exact cohomology sequence associated to 
\[
0\rightarrow F_1\Ddagrig (V^*(1)) \rightarrow \DdagrigE (V^*(1)) 
\rightarrow \mathrm{gr}_2\DdagrigE (V^*(1)) \rightarrow 0
\]
together with i) show that the sequence 
\[
0\rightarrow H^1\left (F_1\Ddagrig (V^*(1))\right )\rightarrow 
H^1\left (\DdagrigE (V^*(1))\right ) \rightarrow 
H^1\left (\mathrm{gr}_2\DdagrigE (V^*(1))\right ) 
\]
is exact. In particular, the map 
$H^1\left (F_1\Ddagrig (V^*(1))\right )\rightarrow 
H^1 \left (\DdagrigE (V^*(1))\right )$ is injective. Moreover, since 
$\mathrm{gr}_2\DdagrigE (V^*(1))$ is the Tate dual of
$F_{-1}\Ddagrig (V),$ from \cite[Corollary~1.4.10]{Ben11} it follows that
\[
\dim_E  H^1_f\left (\mathrm{gr}_2\DdagrigE (V^*(1)\right )=
\dim_E H^1\left (F_{-1}\Ddagrig (V)\right )-
\dim_E H^1_f\left (F_{-1}\Ddagrig (V)\right )=0.
\]
Thus, $ H^1_f\left (\mathrm{gr}_2\DdagrigE (V^*(1)\right )=0,$
and from \cite[Corollary~1.4.6]{Ben11} it follows that
$H^1_f \left (F_1\Ddagrig (V^*(1))\right )= H^1_f\left (\DdagrigE (V^*(1))\right ).$

iv) The last statement follows from ii), iii) and isomorphism theorems. 

\end{proof}

\subsubsection{} Consider the exact sequence (\ref{modified Poitou-Tate sequence})
for the representation $V^*(1).$ Since $H^1_f(\Q,V^*(1))=0,$ it reads 
\begin{equation}
\label{dual Poitou-Tate sequence}
0\rightarrow H^1_{f,\{p\}}(\Q,V^*(1))\rightarrow
\frac{H^1(\Qp, V^*(1))}{H^1_f(\Qp, V^*(1))}
\rightarrow H^1_f(\Q, V)^*\rightarrow 0
\end{equation}
(We remark that the surjectivity of the last map follows from {\bf C5)} and the 
isomorphism 
$
\displaystyle \frac{H^1(\Qp, V^*(1))}{H^1_f(\Qp, V^*(1))}\simeq H^1_f(\Qp, V)^*.)
$
We denote by
\begin{equation}
\nonumber
\kappa_{D^\perp}\,\,:\,\,H^1_{f,\{p\}}(\Q,V^*(1))\rightarrow
\frac{H^1(\Qp, V^*(1))}{H^1_f(\Qp, V^*(1))+H^1\left (F_0\Ddagrig (V^*(1))\right )}
\end{equation}
the composition of the second arrow of this exact sequence with the canonical 
projection.

\begin{mylemma}
The map $\kappa_{D^\perp}$ is an isomorphism.
\end{mylemma}
\begin{proof}
a) First, we prove the injectivity of $\kappa_{D^\perp}.$ 
Assume that $\kappa_{D^\perp}(x)=0.$ Then 
$\res_p(x)\in H^1_f(\Qp, V^*(1))+H^1\left (F_0\Ddagrig (V^*(1))\right ).$ This implies that $\res_p(x)$ belongs to the orthogonal complement
of $H^1_f\left (F_0\Ddagrig (V)\right )$ in $H^1(\Qp, V^*(1)).$ On the other hand, 
from (\ref{dual Poitou-Tate sequence}) it follows that $\res_p(x)$ belongs
to the orthogonal complement of $\res_p \left (H^1_f(\Q,V)\right ).$
Then, by (\ref{decomposition of H^1_f using H^1_f(D)}), we have $\res_p (x)\in \res_p \left (H^1_f(\Q,V)\right )\cap H^1_f\left (F_0\Ddagrig (V)\right )=\{0\},$ and therefore $x=0.$

b) From (\ref{dual Poitou-Tate sequence}) and (\ref{decomposition of H^1_f using H^1_f(D)}) it follows that
\begin{multline}
\label{Lemma about dual kappa: first dimension} 
\dim_EH^1_{f,\{p\}}(\Q,V^*(1))=\\
=\dim_E H^1_f(\Qp,V)-\left (\dim_EH^1_f(\Qp,V)-
\dim_EH^1_f\left (F_0\Ddagrig (V)\right )\right )= 
\dim_E H^1_f\left (F_0\Ddagrig (V)\right ).
\end{multline}
On the other hand, $F_0\Ddagrig (V^*(1))$ is the Tate dual of 
$\mathrm{gr}_1\DdagrigE (V).$
From the tautological exact sequence 
\[
0\rightarrow F_0\Ddagrig (V)\rightarrow \DdagrigE (V)
\rightarrow \mathrm{gr}_1\DdagrigE (V)\rightarrow 0
\]
and the semisimplicity of $\Ph$ we obtain an exact sequence
\[
0\rightarrow H^1_f\left (F_0\Ddagrig (V)\right )\rightarrow H^1_f(\Qp,V)\rightarrow 
H^1_f\left ( \mathrm{gr}_1\DdagrigE V)\right )\rightarrow 0.
\]
Therefore
\begin{multline}
\nonumber
\dim_E H^1\left (F_0\Ddagrig (V^*(1))\right )-\dim_E H^1_f\left (F_0\Ddagrig (V^*(1))\right )
=\dim_E H^1_f\left ( \mathrm{gr}_1\Ddagrig (V)\right )=\\
=\dim_E H^1_f(\Qp,V)-\dim_E H^1_f \left (F_0\Ddagrig (V)\right ).
\end{multline}
Since $H^1 \left (F_0\Ddagrig (V^*(1))\right )\cap H^1_f(\Qp,V^*(1))=
H^1_f\left (F_0\Ddagrig (V^*(1))\right ) ,$ we obtain 
\begin{multline}
\nonumber
\dim_E \left (\frac{H^1(\Qp, V^*(1))}{H^1_f(\Qp, V^*(1))+H^1(\bD^\perp)}\right )=\\
\dim_EH^1(\Qp, V^*(1))-\dim_EH^1_f(\Qp, V^*(1))- H^1\left (F_0\Ddagrig (V^*(1))\right )+
H^1_f\left (F_0\Ddagrig (V^*(1))\right )=\\
=\dim_E H^1_f(\Qp, V)- H^1\left (F_0\Ddagrig (V^*(1))\right )+
H^1_f\left (F_0\Ddagrig (V^*(1))\right )=\dim_E H^1_f\left (F_0\Ddagrig (V)\right ).
\end{multline}
Comparing with (\ref{Lemma about dual kappa: first dimension}), we see that   the source and the target of $\kappa_{D^\perp}$ are of the same dimension.
Since $\kappa_{D^\perp}$ is injective, this proves the lemma.  
\end{proof}

\begin{myremark} It is not difficult to prove that there is 
a canonical isomorphism
\[
\frac{H^1(\Qp, V^*(1))}{H^1_f(\Qp, V^*(1))+H^1\left (F_0\Ddagrig (V^*(1))\right )}
\simeq \frac{H^1(\bD^* (\chi))}{H^1_f(\bD^* (\chi))}.
\]
\end{myremark}

\subsubsection{}
Let $ H^1(D^\perp,V^*(1))$ denote the inverse image of
\[
H^1\left (F_1\Ddagrig (V^*(1))\right )/\left (H^1_f(\Qp, V^*(1))+H^1_f\left (F_0\Ddagrig (V^*(1))\right )\right )
\] 
under $ \kappa_{D^\perp}.$ 
Taking into account (\ref{dual L-invariant first isomorphism}),   we 
see that  $\kappa_{D^\perp}$ induces an isomorphism
\begin{equation}
\label{definition of dual L-inv:main isomorphism}
H^1(D^\perp,V^*(1)) \simeq \frac{H^1(\bM^*(\chi))}{H^1_f(\bM^*(\chi))}.
\end{equation}
Consider  the composition map
\[
H^1(D^\perp,V^*(1))\rightarrow H^1 \left (F_1\Ddagrig (V^*(1))\right )\rightarrow H^1(\bM^*(\chi)).
\]
Taking into account (\ref{construction of i_(W^*(chi))})
and (\ref{decomposition of W^*(chi)}),
we obtain that 
$\dim_E H^1(D^\perp,V^*(1))=\dim \CDcris (\bM).$ 
Hence we have a diagram
\[
\xymatrix{
\CDcris (\bM^*(\chi )) \ar[r]^{\overset{i_{\bM^*(\chi),f}}\sim} &H^1_f(\bM^*(\chi))\\
H^1(D^\perp,V^*(1)) \ar[u]^{\rho_{D^\perp,f}} \ar[r] \ar[d]_{\rho_{D^\perp,c}}
&H^1(\bM^*(\chi)) \ar[u]_{p_{\bM^*(\chi),f}}
\ar[d]^{p_{\bM^*(\chi),c}}
\\
\CDcris (\bM^*(\chi)) \ar[r]^{\overset{i_{\bM^*(\chi),c}}\sim} &H^1_c(\bM^*(\chi)),}
\]
where $\rho_{D^\perp,f}$ and $\rho_{D^\perp,c}$ are defined as the unique maps making this diagram commute.
From (\ref{definition of dual L-inv:main isomorphism})  it follows that  $\rho_{D^\perp,c}$ is an isomorphism.

\begin{mydefinition} The determinant
\[
\mathscr L (V^*(1),D^\perp)= (-1)^e\det \left ( \rho_{D^\perp,f} \circ \rho^{-1}_{D^\perp,c}\,\mid \,\CDcris (\bM^*(\chi)) \right ),
\]
where $e=\dim_E \CDcris(\bM (\chi)),$
will be called  the $\mathscr L$-invariant associated to $V^*(1)$ and $D^\perp$.
\end{mydefinition}

\begin{enonce*}[remark]{Remark}
The sign $(-1)^e$ corresponds to the sign in the conjectural functional equation
for $p$-adic $L$-functions. We refer the reader to \cite[Sections~2.2.6 and 2.3.5]{Ben11} for more detail.
\end{enonce*}

The following proposition is a direct generalization of 
\cite[Proposition~2.2.7]{Ben11}.

\begin{myproposition} 
\label{proposition comparision l-inv}
Assume that $D$ is a regular submodule of $\Dc (V).$
Then
\[
\mathscr L(D^\perp, V^*(1))=(-1)^e\mathscr L(D, V).
\]
\end{myproposition}
\begin{proof} The proof repeats {\it verbatim}
the proof of \cite[Proposition~2.2.7]{Ben11}. We leave the details
to the reader.
\end{proof}

\section{Modular curves and $p$-adic representations }
\label{Section modular curves}

\subsection[]{Notation and conventions}
\subsubsection{} Let $M$ and $N$ be two positive integers such that $M+N\geqslant 5.$ We denote by $Y_1(M,N)$ the scheme over $\Z[1/MN]$ representing the functor
\begin{equation}
\nonumber 
S \mapsto \text{\rm isomorphism classes of $(E,e_1,e_2)$},
\end{equation} 
where $E/S$ is an elliptic curve, $e_1, e_2\in E(S)$ are such that
$M e_1=N e_2=0$  and the map
\begin{equation}
\nonumber
\begin{cases}
\Z/M\Z \times \Z/N\Z \rightarrow E,\\
(m,n)\mapsto me_1+ne_2
\end{cases}
\end{equation}
is injective. As usual, we set $Y_1(N)=Y(1,N)$ for $N\geqslant 4$
and write $(E,e)$ instead $(E,0,e_2).$
Recall that 
\[
Y(M,N)(\BC)= \Gamma (M,N)\backslash \bf H,
\]
where $\bf H$ is the Poincar\'e half-plane and 
\begin{equation}
\nonumber
\Gamma (M,N)
=\left \{ 
\left (
\begin{matrix}
a & b\\
c & d
\end{matrix}
\right )\in \mathrm{SL}_2(\Z) \left \vert \right.
a\equiv 1(M),
b\equiv 0(M),
c\equiv 0(N),
d\equiv 1(N)
\right \}.
\end{equation}
In particular, $Y_1(N)(\BC)=\Gamma_1(N)\backslash \mathbf H,$ where
\begin{equation}
\nonumber 
\Gamma_1 (N)
=\left \{ 
\left (
\begin{matrix}
a & b\\
c & d
\end{matrix}
\right )\in \mathrm{SL}_2(\Z) \left \vert \right.
a\equiv d\equiv 1(N),
c\equiv 0(N)
\right \}.
\end{equation}

\subsubsection{} Let  $a$ be a positive integer. 
We denote by ${\Pr}_i\,:\, Y_1(Na)\rightarrow Y_1(N)$ ($i=1,2$) the morphisms 
defined by 
\begin{align}
\label{the Pr maps} 
& {\Pr}_1\,:\, (E,e)\mapsto (E, a e),\\
\label{the Pr2 map}
& {\Pr}_2\,:\, (E,e)\mapsto (E/\left <N e\right > , e \mod{\left <N e\right >}),
\end{align}
where $\left <a e\right >$ denotes the cyclic group of order $N$ generated
by $ae  .$

\subsubsection{}
\label{subsubsection Y(N,p)}
Fix a prime number $p\geqslant 3$ such that $(p,N)=1.$
We denote by $Y (N,p)$
the scheme over $\Z[1/Np]$ which represents the functor
\begin{equation}
\nonumber 
S \mapsto \text{\rm isomorphism classes of $(E,e,C)$},
\end{equation} 
where $E/S$ is an elliptic curve,  $e\in E(S)$ is a point
of order $N$ and $C\subset E$ is a subgroup of order $Np$
such that $e\in C.$
We have $Y (N,p)(\BC)=\Gamma_p (N) \backslash \mathbf H,$ where
$\Gamma_p (N)=\Gamma_1(N)\cap \Gamma_0(p)$ and 
\begin{equation}
\nonumber 
\Gamma_0 (p)
=\left \{ 
\left (
\begin{matrix}
a & b\\
c & d
\end{matrix}
\right )\in \mathrm{SL}_2(\Z) \left \vert \right.
c\equiv 0(p)
\right \}.
\end{equation}
We have a canonical projection 
\begin{equation}
\label{definition pr'}
\begin{aligned}
& {\pr}'\,:\,Y_1(Np) \rightarrow Y(N,p),\\
&{\pr}' \,:\, (E,e) \mapsto (E, p e, \left <e \right >).
\end{aligned}
\end{equation} 
We also define  projections 
\begin{equation*}
{\pr}_i \,:\, Y(N,p) \rightarrow Y_1(N), \qquad i=1,2
\end{equation*}
by 
\begin{equation}
\label{definition of pr_i}
\begin{aligned}
& {\pr}_1 \,:\, (E,e,C) \mapsto (E,e),\\
& {\pr}_2 \,:\, (E,e,C) \mapsto (E/NC,e'),
\end{aligned}
\end{equation}
where $e'\in C$ is the unique element of $C$ such that 
$pe'=e.$

Note that we have commutative diagrams  

\begin{equation}
\label{commitative diagram pr}
\xymatrix{ Y_1(Np)\ar[r]^{{\pr}'} \ar[dr]^{{\Pro}_i} & Y(N,p) \ar[d]^{{\pr}_i}\\
 & Y_1(N).
}, \qquad i=1,2.
\end{equation}

\subsubsection{}
\label{subsubsection adic sheaves}
If $Y$ is an unspecified modular curve, we denote by $\lambda \,:\,\CE\rightarrow Y$  the universal elliptic curve over $Y$ and 
set $\cF_n=\R^1\lambda_*\Z/p^n\Z (1),$ 
$\cF=\R^1\lambda_*\Z_p (1)$ and $\cF_{\Qp}=\cF \otimes_{\Z_p}\Q_p.$
Let $\iota_D\,:\,D \hookrightarrow \CE$ be a subscheme. We assume that $D$ is  \'etale over $Y.$ Consider the diagram
\begin{equation}
\label{diagram definition of adic sheaves}
\xymatrix{
\CE[p^r]\left <D\right > \ar[d]^{\mathrm{p}_{D,r}} \ar[r] &\CE \ar[d]^{[p^r]}\\
D \ar[r]^{\iota_D} \ar[d]^{\lambda_D} & \CE \ar[dl]^{\lambda}\\
Y &
}
\end{equation}
where $\CE[p^r]\left <D\right > $ is the fiber product of $D$ and $\CE$
over $\CE .$ Define
\begin{equation}
\label{definition of adic sheaves}
\La_r (\cF_r\left <D\right >)=\lambda_{D,*}\circ \mathrm{p}_{D,r,*} (\Z/p^r\Z),
\qquad \La  (\cF \left <D\right >)= \left (\La_r (\cF_r\left <D\right >)\right )_{r\geqslant 1}.
\end{equation}
We consider $\La (\cF\left <D\right >)$ as an \'etale pro-sheaf.
If $D=Y$ and $\iota_D$ is a section  $s\,:\,Y\rightarrow \CE,$
we use the notation $\La (\cF\left <s\right >)$ to indicate the dependence
on $s.$ If $s$ is the zero section, we write $\La (\cF)$
instead  $\La (\cF\left <0\right >).$ Note that 
in this case $\CE[p^r]\left <D\right > = \CE [p^r],$ and 
\[
\La_r (\cF_r)= \Z/p^r\Z \left[\,\CE [p^r] \,\right].
\]
Therefore, $\La (\cF)$ can be viewed as the sheaf of Iwasawa algebras 
associated to the relative Tate module of $\CE .$

These sheaves were studied in detail in \cite{Ki15} and we refer the reader
to {\it op. cit.} for further information.

\subsubsection{}
\label{subsubsection sheaves F<s_N>}
Let  $M \mid N$ and let $\CE \rightarrow Y(M,N)$ be the 
universal elliptic curve.
 Let $s_N\,:\,Y(M,N) \rightarrow \CE$
denote the canonical section which sends  the class $(E,e_1,e_2)$ 
to  $e_2\in \CE [N].$ For each $r\geqslant 1,$ consider the cartesian square
\begin{equation}
\nonumber
\xymatrix{
\CE [p^r]\left <s_N \right > \ar[r] \ar[d] &\CE \ar[d]^{[p^r]}\\
Y(M,N) \ar[r]^{s_N} &\CE.
}
\end{equation}
We denote by
\begin{equation}
\label{definition of the Iwasawa sheaf}
\Lambda (\cF\left <N\right >):=\Lambda (\cF\left <s_N\right >)
\end{equation}
the associated sheaf. 
From the interpretation
of $Y(M,N)$ as moduli space it follows that $Y(M,Np^r) \simeq \CE [p^r]\left <s_N \right >,$ and therefore 
\begin{equation}
\nonumber
H^1_{\et}(Y(M,Np^r), \Z/p^r (1))\simeq H^1_{\et}(\CE [p^r]\left <s_N \right > , \Z/p^r (1))
\simeq H^1_{\et}(Y(M,N),\Lambda_r(\cF_r\left <N\right >) 
\end{equation}
Passing to projective limits, we obtain a canonical isomorphism
\begin{equation}
\label{isomorphism for cohomology of Iwasawa sheaf}
\underset{r}\varprojlim H^1_{\et}(Y(M,Np^r), \Z/p^r (1))
\simeq  H^1_{\et}(Y(M,N),\Lambda (\cF \left <N\right >)(1)).
\end{equation}
(see \cite[Section 4.5]{KLZ}).

\subsubsection{}For a module $M$ over a commutative ring $A$ we denote 
by $\Sym^k(M)$ (resp. $\TSym^k(M)$) the quotient   of 
 $S_k$-coinvariants  (resp. the 
submodule  of $S_k$-invariants) of $M^{\otimes k}.$

\subsection{$p$-adic representations}
\label{subsection p-adic representations}
\subsubsection{}
For each $k\geqslant 2$ we denote by 
$\Sym^{k}(\cF_{\Qp})$ the $k$th symmetric tensor power of the sheaf
$\cF_{\Qp}.$ The \'etale cohomology 
\[
H^1_{\mathrm{\acute et},c}\left (Y_1(N)_{\overline{\Q}}, \Sym^{k} (\cF_{\Qp}^{\vee} )\right )
\]
is equipped with a natural action of the Galois group $G_{\Q,S}$
and the Hecke and Atkin--Lehner   operators $T_l$ (for primes  $(l,N)=1$) 
and $U_l$ (for primes $l\vert N$),  which commute to each over.
Let $f=\underset{n=1}{\overset{\infty}\sum}a_nq^n,$ $q=e^{2\pi i z}$ be a normalized cuspidal eigenform of  level $N_f$ and weight $k_0=k+2\geqslant 2.$ We do not assume that $f$ is a newform. 
Fix a finite extension $E/\Qp$
containing  $a_n$ for all $n\geqslant 1.$
Deligne's $p$-adic
representation associated to $f$ is the maximal subspace 
\begin{equation}
\nonumber
W_f=H^1_{\mathrm{\acute et},c}\left (Y_1(N_f)_{\overline{\Q}}, \Sym^{k}(\cF_{\Qp}^{\vee})\right )_{(f)}
\end{equation}
of
\[
H^1_{\mathrm{\acute et},c}\left (Y_1(N_f)_{\overline{\Q}}, \Sym^{k}(\cF_{\Qp}^{\vee})\right )\otimes_{\Qp}E  
\]
on which the operators  $T_l$ (for $(l,N_f)=1$) and $U_l$ (for $l\vert N_f$) act as multiplication by $a_l$ for all primes $l .$ Then $W_f$ is a two dimensional 
$E$-vector space equipped with a continuous action of  $G_{\Q,S},$ 
which does not depend on the choice of the level in the following sense.  If $N_f\vert N$ then 
the canonical projection ${\Pr}_1 \,:\,Y_1(N)\rightarrow Y_1(N_f)$ induces a morphism 
\begin{equation}
\nonumber
{\Pr}_{1,*}\,:\,H^1_{\mathrm{\acute et},c}\left (Y_1(N)_{\overline{\Q}}, \Sym^{k}(\cF_{\Qp}^{\vee})\right )
\rightarrow
H^1_{\mathrm{\acute et},c}\left (Y_1(N_f)_{\overline{\Q}}, \Sym^{k}(\cF_{\Qp}^{\vee})\right ),
\end{equation}
which is  isomorphism on the $f$-components.
Note, that here in our notation we do not distinguish
between the sheaves  $\cF_{\Qp}^\vee$ on $Y_1(N)$ and $Y_1(N_f).$

From the Poincar\'e duality it follows that the dual representation $W_f^*$ can be realized as the  
quotient 
\begin{equation}
\label{dual Deligne's representaton}
W_f^*=H^1_{\mathrm{\acute et}}\left (Y_1(N_f)_{\overline{\Q}}, \TSym^{k}(\cF_{\Qp})(1)\right )_{[f]}
\end{equation}
of 
\[
H^1_{\mathrm{\acute et}}\left (Y_1(N_f)_{\overline{\Q}}, \TSym^{k}(\cF_{\Qp})(1)
\right )\otimes_{\Qp}E 
\]
by the submodule 
generated by the images of $T'_l-a_l$ (for   $(l,N_f)=1$)
and $U'_l-a_l$ (for $l\vert N_f$),  where $T_l'$ (resp. $U'_l$) denote the dual Hecke (resp. Atkin--Lehner) operators.

\subsubsection{} Let
\begin{equation}
\nonumber
\rho_f\,:\,G_{\Q,S}\rightarrow \mathrm{GL}(W_f)
\end{equation}
denote the representation of $G_{\Q,S}$ on $W_f.$ The following proposition summarizes some properties of the representation
$W_f.$

\begin{myproposition} Assume that $f\in S_{k_0}^{\mathrm{new}}(N_f,\ep_f )$ is a newform of level $N_f,$ weight $k_0=k+2\geqslant 2$ and nebentypus $\ep_f .$ Let $p$ be a prime
such that $(p,N_f)=1.$ Then

i) $\det \rho_{f}$ is isomorphic to $\ep_f \chi^{1-k_0}$ where $\chi$ is the cyclotomic character.

ii) $\rho_f$ is unramified outside the primes $l \mid N_fp.$

iii) If $l\neq p,$ then 
$$
\det (1-\mathrm{Fr}_lX \mid W_{f}^{I_l})= 1-a_lX+\ep_f(l)\,l^{k_0-1}X^2
$$
(Deligne--Langlands--Carayol theorem \cite{De71, La73, Ca86}).

iii) The restriction of $\rho_{f}$ on the decomposition group at $p$ is crystalline 
with Hodge--Tate weights $(0,k_0-1).$ Moreover 
\[
\det (1-\Ph X \mid \Dc(W_{f}))= 1-a_pX+\ep_f(p)\,p^{k_0-1}X^2
\]
(Saito's theorem \cite{Sa97}).

\end{myproposition}

\subsubsection{} 
\label{subsubsection conditions C}
We retain previous notation. Let $f(q)=\underset{n=1}{\overset{\infty}\sum} a_n q^n$ be a newform of weight $k_0\geqslant 2,$ level $N_f$ and nebentypus $\ep_f .$
Fix a prime number $p$ such that $(p,N_f)=1$ and denote by $\alpha (f)$ and 
$\beta (f)$
the roots of the Hecke polynomial $X^2-a_pX+\ep_f (p)p^{k_0-1}.$
Till the end of this chapter we assume that the following
conditions hold:

\begin{itemize}
\item[]{\bf 1)} $\alpha (f) \neq \beta (f).$ 
\item[]{\bf 2)} If $v_p(\alpha)=k_0-1$ then $\left. \rho_f \right \vert_{G_{\Qp}}$ 
does  not split.
\end{itemize}
We remark that {\bf 1)} conjecturally always holds but is known to be true only 
in the weight $2$ case \cite{CE98}. One expects (see for example \cite{GV04}) that {\bf 2)} does not hold 
({\it i.e.} $\rho_f$ locally splits) only if $f$  is a  CM form  with $p$ split. 

\subsubsection{}
The $p$-stabilization $f_{\alpha}(q)=f(q)-\beta (f)\cdot f(q^p)$ is an oldform
with respect to the subgroup $\Gamma_p(N_f).$ The map $\pr',$
defined in (\ref{definition pr'}),
induces  an isomorphism 
\begin{equation}
\nonumber
\label{isomorphism under pr'}
W_{f_{\alpha}}^*=
H^1_{\mathrm{\acute et}}\left (Y_1(N_fp)_{\overline{\Q}}, \TSym^{k}(\cF_{\Qp})(1)\right )_{[f_{\alpha}]} \xrightarrow{\overset{\pr'_*}\sim} 
H^1_{\mathrm{\acute et}}\left (Y(N_f,p)_{\overline{\Q}}, \TSym^{k}(\cF_{\Qp})(1)\right )_{[f_{\alpha } ]}.
\end{equation}
Moreover, the map
\begin{equation}
\nonumber
{\Pr}^\alpha_*\,\,:\,\,
H^1_{\mathrm{\acute et}}\left (Y_1(N_fp)_{\overline{\Q}}, \TSym^{k}(\cF_{\Qp})(1)
\right )
\rightarrow H^1_{\mathrm{\acute et}}\left (Y_1(N_f)_{\overline{\Q}}, \TSym^{k}(\cF_{\Qp})(1)\right )
 \end{equation}
 defined by 
 \begin{equation}
 \label{definition of pr^alpha}
{\Pr}^\alpha_*={\Pr}_{1,*}-\frac{\beta (f)}{p^{k_0-1}}\cdot {\Pr}_{2,*}
\end{equation}
factorizes through appropriate quotients and  
induces an isomorphism 
\begin{equation}
\label{isomorphism of representations  stabilized and
non stabilized }
{\Pr}^\alpha_*\,\,:\,\, W_{f_{\alpha }}^* \simeq W_f^*
\end{equation}
(see \cite[Proposition ~2.4.5]{KLZ}). Taking into account the diagram (\ref{commitative diagram pr}), we can be summarize this information 
in the following commutative diagram 
\begin{equation}
\label{diagram with pr' and Pr for eigenspaces}
\xymatrix{
W_{f_{\alpha }}^* \ar[rr]^(.3){\overset{\pr'_*}\sim} \ar[drr]_{{\Pr}^\alpha_*} & & H^1_{\mathrm{\acute et}}\left (Y(N_f,p)_{\overline{\Q}}, \TSym^{k}(\cF_{\Qp})(1)\right )_{[f_{\alpha } ]} \ar[d]^{\pr^{\alpha}_*}\\
& &W_f^*.
}
\end{equation}
Here  
\begin{equation}
\nonumber
{\pr}^\alpha_*={\pr}_{1,*}-\frac{\beta (f)}{p^{k_0-1}}\cdot {\pr}_{2,*}.
\end{equation}
We denote by 
\begin{equation}
\label{dual isomorphism of representations  stabilized and
non stabilized }
{\Pr}_{\alpha}^* \,\,:\,\, W_{f} \simeq W_{f_{\alpha }}
\end{equation}
the dual isomorphism.

\subsubsection{}
\label{subsubsection basis of Dieudonne module}
The newform $f$ defines a canonical 
generator $\omega_f$ of the one-dimensional $E$-vector space $\F^{k_0-1}\Dc (W_f)$
(see \cite[Section~11.3]{Ka04}). Note that ${\Pr}^{\alpha, *}(\omega_f )=
\omega_{f_\alpha},$ where ${\Pr}^{\alpha, *}\,:\, \Dc (W_f)
\rightarrow \Dc (W_{f_{\alpha}})$ denotes the isomorphism induced by 
(\ref{dual isomorphism of representations  stabilized and
non stabilized }).

Let $f^*=\underset{n=1}{\overset{\infty}\sum} \overline{a}_n q^n$ denote the 
complex conjugate of $f.$ The Atkin--Lehner operator $w_{N_f}$ acts on $f$ by
\begin{equation*}
w_{N_f}(f)=\lambda_{N_f}(f)f^*,
\end{equation*}
where $\lambda_{N_f}(f)$ is called the pseudo-eigenvalue of $f.$
The canonical pairing $W_f\times W_{f^*}\rightarrow E(1-k_0)$
induces a pairing
\begin{equation*}
\left [\,\,,\,\,\right ]\,\,:\,\, \Dc (W_f)\times \Dc (W_{f^*})\rightarrow 
\Dc (E(1-k_0)).
\end{equation*}
The filtered module $\Dc (E(1-k_0))$ has the canonical generator 
$e_{1-k_0}=\left (\ep\otimes t \right )^{\otimes (1-k_0)},$
where $\ep= (\zn)_{n\geqslant 0}$ and $t=\log [\ep]\in \Bd$ is 
the associated uniformizer of the field of de Rham periods (note that 
$e_{1-k_0}$ does not depend on the choice of $\ep$). Since $\alpha (f) \neq 
\beta (f),$
we have $\Dc (W_f)=\Dc (W_f)^{\Ph=\alpha (f)}\oplus \Dc (W_f)^{\Ph=\beta (f)}.$
From {\bf 2)}, it follows that $\omega_{f^*}$ is not an eigenvector 
of $\Ph,$ and we denote by $\eta_f^\alpha$ the unique element of $\Dc (W_f)^{\Ph=\alpha (f)}$
such that
\begin{equation*}
\left [\eta_f^{\alpha}, \omega_{f^*}\right ]=e_{1-k_0}.
\end{equation*}
We also denote by $\omega_f^\beta$ the unique element of $\Dc (W_f)^{\Ph=\beta (f)}$
such that 
\begin{equation*}
\omega_f^\beta \equiv \omega_f \mod{\Dc (W_f)^{\Ph=\alpha (f)}}.
\end{equation*}
Note that $\{\eta_f^{\alpha}, \omega_f^\beta \}$ is a basis of $\Dc (W_f).$

\subsubsection{}
\label{subsubsection f*} Set 
\begin{equation}
\nonumber
\alpha (f^*)=\frac{p^{k_0-1}}{\beta (f)}, \qquad 
\beta (f^*)=\frac{p^{k_0-1}}{\alpha (f)}.
\end{equation}
Then $\alpha (f^*)$ and  $\beta (f^*)$ are the roots of the Hecke polynomial
of $f^*$ at $p.$ We denote by $\{\eta_{f^*}^\alpha, \omega_{f^*}^\beta \}$
the corresponding basis of $\Dc (W_{f^*}).$ It is easy to check that 
it is dual to the basis $\{\eta_f^{\alpha}, \omega_f^\beta \}.$
\begin{equation}
\nonumber
\end{equation}

\subsection{Overconvergent \'etale cohomology} 

\subsubsection{} In this section, we review the construction of 
$p$-adic representations associated to Coleman families \cite{Han15, LZ}. It relies heavily on the  overconvergent Eichler--Shimura isomorphism  of Andreatta, Iovita and  Stevens \cite{AIS}. Let $\mathcal W=\Hom_{\text{\rm cont}} (\Zp^*, \mathbf{G}_m)$ denote the weight space. As usual, we consider 
$\mathcal W$ as a rigid analytic space over some fixed finite 
extension $E$ of $\Qp.$ Namely, since  $\Zp^*\simeq \mu_{p-1}\times 
(1+p\Zp)^*,$ each continuous character $\eta \,:\,\Zp^*\rightarrow L^*$
is completely determined by its restriction on $\mu_{p-1}$ and 
its value at $1+p.$ This identifies $\mathcal W$ with the union of 
$p-1$ rigid open balls of radius $1.$ Let $\mathcal W^*$ denote
the subspace of weights $\kappa $ such that $v_p(\kappa (z)^{p-1}-1)\geqslant \frac{1}{p-1}.$ Note that $\mathcal W^*$ contains all characters
$\kappa $ of the form $\kappa (z)=z^m,$ $m\in \Z.$  If $U$ is an open disk, we 
denote by $\CO_U^0$ the ring of rigid analytic functions on $U$ bounded by $1$
and set $\CO_U=\CO_U^0[1/p].$ We remark that $\CO_U^0= O_E[[u]]$ for some $u$
and  denote by $\frak m_U$ the maximal ideal of $\CO^0_U.$
The inclusion $U\subset \mathcal W$ fixes a character 
\[
\kappa_U \in \mathcal W (U)=\Hom_{\text{\rm cont}}(\Zp^*, \CO_U^*).
\]
If $x\in U (L),$ we can consider $x$ as a homomorphism  $x\,:\, \CO_U\rightarrow L.$ Let $\kappa_x :\,\Zp^*\rightarrow L^*$ denote the character parametrized by $x.$
Then we have
\[
\kappa_x=x\circ \kappa_U.
\]

\subsubsection{}  Consider the 
following compact subsets of $\Zp^2:$
\[
T_0=\Zp^*\times \Zp, \qquad T_0'=p\Zp\times \Zp^*.
\]
For any weight $\kappa \in \mathcal W^*(L)$ we denote by $A_\kappa^0(T_0)$
(respectively $A_\kappa^0(T_0')$) the module of functions $f\,:\,T_0
\rightarrow O_L$ (respectively $f\,:\,T_0'
\rightarrow O_L$) such that:

\begin{itemize}
\item[]{$\bullet$} $f$ is  homogeneous of weight $\kappa$ {\it i.e.}
\[
f(ax,ay)=\kappa (a) f(x,y), \qquad \forall a\in\Zp.
\]
\item[]{$\bullet$} $f(1,z)= \underset{n=0}{\overset{\infty}\sum} c_nz^n,$
$c_n\in O_L,$ where $(c_n)_{n\geqslant 0}$ converges to $0.$
\end{itemize}
Analogously, for any open disk  $U\subset \mathcal W^*$ we denote by $A_U^0(T_0)$
(respectively $A_U^0(T_0')$) the module of functions $f\,:\,T_0
\rightarrow \CO^0_U$ (respectively $f\,:\,T_0'
\rightarrow \CO^0_U$) such that:

\begin{itemize}
\item[]{$\bullet$} $f$ is  homogeneous of weight $\kappa_U,$ {\it i.e.}
\[
f(ax,ay)=\kappa_U(a) f(x,y), \qquad \forall a\in\Zp.
\]
\item[]{$\bullet$} $f(1,z)= \underset{n=0}{\overset{\infty}\sum} c_nz^n,$
$c_n\in \CO^0_U,$ where $(c_n)_{n\geqslant 0}$ converges to $0$
in the $\frak m_U$-adic topology.
\end{itemize}
 Define 
\begin{equation}
\nonumber
D_\kappa^0=\Hom_{\mathrm{cont}}(A_\kappa^0(T_0), O_L),\qquad
D_U^0=\Hom_{\mathrm{cont}}(A_U^0(T_0), \CO^0_U)
\end{equation}
and 
\begin{equation}
\nonumber
D_\kappa=D_\kappa^0[1/p], \qquad
D_U=D_U^0[1/p].
\end{equation}
We have the specialization map 
\begin{equation}
\label{specialization of distributions}
\begin{aligned}
&\spec_\kappa \,:\,D_U \rightarrow D_\kappa,\\ 
&\spec_\kappa (\mu_U) (f)=  \kappa(\mu_U(f_U)) , \qquad 
\textrm{where $f_U(x,y)=f(1,y/x) \kappa_U(x)$}
\end{aligned}
\end{equation}
(see \cite[Section~3.1]{AIS}). 

For each positive integer $k,$ we denote by $P_k^0$ the $O_E$-module of 
homogeneous polynomials of degree $k$ with coefficients in $O_E.$
We remark that there exists a canonical isomorphism
\begin{equation}
\label{homogeneous polynomials vs O_E^2}
\Hom_{O_E}(P_k^0,O_E) \simeq \TSym^k (O_E^2).
\end{equation}
The restriction of distributions on $P_k^0$ provides us with a map 
$D^0_k\rightarrow \TSym^k (O_E^2).$ Composing this map with 
(\ref{specialization of distributions}), 
we obtain a map 
\begin{equation}
\nonumber
\theta_k\,:\,D^0_U\rightarrow \TSym^k(O_E^2).
\end{equation}

\subsubsection{} Let $\La (T_0),$ $\La (T_0')$ and $\La (\Zp^2)$ 
denote the modules of $p$-adic measures with values in $O_E$ on
$T_0,$ $T_0'$ and $\Zp^2$ respectively. We remark that $\La (\Zp^2)$
is canonically isomorphic to the Iwasawa algebra of $\Zp^2$ and for 
each integer $k\geqslant 0$ we have the moment map 
\begin{equation}
\label{definition of moments}
\begin{aligned}
&\mom^k \,:\, \La (\Zp^2) \rightarrow \TSym^k(O_E^2),\\
&\mom^k (h)=h^{\otimes k}, \qquad h\in \Zp^2.
\end{aligned}
\end{equation}
(see \cite[Section~2]{Ki15}). For $T=T_0, T_0',$ we have a commutative diagram
\begin{equation}
\nonumber
\xymatrix{
\La (T) \ar[r] \ar[d] & D_U^0(T) \ar[d]^{\theta_k}\\
\La (\Zp^2) \ar[r]^(.4){\mom^k} & \TSym^k (O_E^2).
}, 
\end{equation}
where $k\in U\cap \Z$ (see \cite[Proposition~4.2.10]{LZ}).

\subsubsection{} Let $N\geqslant 4$ and let $p$ be an odd prime such that
$(p,N)=1.$ The fundamental group
of $Y(N,p)(\mathbf C)$ is $\Gamma_p(N)=\Gamma_1(N)\cap \Gamma_0(p).$ 
Its $p$-adic completion $\widehat{\Gamma_p(N)}$  is isomorphic to the 
$p$-Iwahori subgroup $U_0(p)$ and  it acts on the pro-$p$-covering 
$Y_1(p^\infty, Np^\infty)$
of $Y(N,p).$ This defines a morphism $\pi_1^{\,\mathrm{\acute et}} (Y(N,p))\rightarrow U_0(p).$
Thus the natural action of $U_0(p)$ on $D_U^0(T_0)$ and $D_U^0(T'_0)$ provides these objects  with an action of $\pi_1^{\,\mathrm{\acute et}} (Y(N,p)).$  Therefore, $D_U^0(T_0)$ and $D_U^0(T'_0)$
define  pro-\'etale sheaves on $Y(N,p),$ which we will denote by 
$\mathfrak D_U^0(\cF)$ and $\mathfrak D_U^0(\cF ')$ 
respectively.

\subsubsection{} 
\label{subsubsection subgroups D and D'}
Let $\mathcal E \rightarrow Y(N,p)$ denote the universal
elliptic curve over $Y(N,p).$ Let $C\subset \mathcal E[p]$ denote the canonical
subgroup of order $p$ of $\CE [p].$ Set $D=\CE [p]-C$ and $D'=C-\{0\}$
and consider the pro-sheaves $\La (\cF\left <D\right >)$ and $\La (\cF\left <D'\right >)$ defined by (\ref{diagram definition of adic sheaves}-
\ref{definition of adic sheaves}). 

\begin{myproposition} 
\label{proposition about overconvergent sheaves}
i) The sheaves  $\La (\cF\left <D\right >)$,
$\La (\cF\left <D'\right >)$ and $\La (\cF)$ 
are induces by the modules $\La (T_0),$  $\La (T'_0)$  and
$\La (\Zp^2)$ equipped with the natural action of $\pi_1^{\mathrm{\,\acute et}} (Y(N,p)).$

ii) The natural inclusions $\La (T_0) \rightarrow D_U(T_0)$ and 
$\La (T'_0) \rightarrow D_U(T'_0)$ induce morphisms of sheaves
$\La (\cF\left <D\right >) \rightarrow \mathfrak D_U^0(\cF)$ and
$\La (\cF\left <D'\right >) \rightarrow \mathfrak D_U^0(\cF ').$

iii) For any $k\in U\cap \Z,$ we  have commutative diagrams
\begin{equation}
\nonumber
\xymatrix{
\La (\cF\left <D\right >) \ar[r] \ar[d]^{[p]_*} &\mathfrak D_U(\cF)\ar[d]^{\theta_k}\\
\La (\cF) \ar[r]^(.5){\mom^k} &\TSym^k(\cF_{\Qp}),
}
\qquad 
\xymatrix{
\La (\cF\left <D'\right >) \ar[r] \ar[d]^{[p]_*} &\mathfrak D_U(\cF')\ar[d]^{\theta_k}\\
\La (\cF) \ar[r]^(.5){\mom^k} &\TSym^k(\cF_{\Qp}),
}
\end{equation}
where $\mom^k$ is the moment map on sheaves induced by (\ref{definition of moments}).
\end{myproposition}
\begin{proof} See \cite[Propositions~4.4.2 and 4.4.5]{LZ}.
\end{proof}

\subsubsection{} In \cite{AIS}, Andreatta, Iovita and Stevens defined 
the \'etale cohomology with coefficients in the sheaves $\mathfrak D_U(\cF)$
and $\mathfrak D_U(\cF').$ Set
\begin{equation}
\nonumber 
W (U)^0=H^1_{\mathrm{\acute et}}(Y(N,p)_{\overline\Q}, \mathfrak D_U(\cF))(-\kappa_U), \qquad 
W' (U)^{0}=H^1_{\mathrm{\acute et}}(Y(N,p)_{\overline\Q}, \mathfrak D_U(\cF')(1))
\end{equation}
and 
\begin{equation}
\nonumber
W (U)=W (U)^0\otimes_{O_E}E , \qquad W' (U)=W'(U)^{0}\otimes_{O_E}E .
\end{equation}
We remark that $W(U)$ and $W'(U)$ are $\CO_U$-modules  equipped with 
a continuous linear action of the Galois group $G_{\Q,S}$ and  an action of Hecke operators which commute to each other. 

\subsubsection{}  Assume that  $k\in U.$ 
The map $x\mapsto x+k$ defines a canonical bijection between  $U-k$ and $U$ 
and, therefore,  an isomorphism $t_k\,:\,\CO^0_{U-k} \simeq \CO^0_U.$
If $F\in A_{U-k}^0(T_0')$ and $G\in P_k^0$ is a homogeneous polynomial 
of degree $k,$ then $t_k\circ (FG)\in A_U^0(T_0'),$ and we have a well defined 
map $A_{U-k}^0(T_0')\otimes P_k^0 \rightarrow A_{U}^0(T_0').$
Passing to the duals and using the isomorphism (\ref{homogeneous polynomials vs O_E^2}) we obtain a map 
\begin{equation}
\nonumber
\beta_k^* \,:\, D_U^0(T_0') \rightarrow D_{U-k}^0(T_0') \otimes \TSym^k(O_E^2).
\end{equation}
We use the same notation for the induced map of sheaves
\begin{equation}
\label{definition of beta}
\beta_k^* \,:\, \mathfrak D_U^0(\cF') \rightarrow \mathfrak D_{U-k}^0(\cF') \otimes \TSym^k(\cF ).
\end{equation}
Let 
\begin{equation}
\nonumber
\delta_k\,:\, A_{U}^0(T_0')\rightarrow A_{U-k}(T_0') \otimes P_k^0
\end{equation}
be the map defined by
\begin{equation}
\nonumber
\delta_k (F)=\frac{1}{k!}\underset{i+j=k}\sum
\frac{\partial^k F(x,y)}{\partial x^i\partial y^j}\otimes x^iy^j.
\end{equation}
Passing to the duals we obtain a map
\begin{equation}
\nonumber
\delta_k^* \,:\,D_U^0(T_0') \rightarrow D_{U-k}(T_0') \otimes \TSym^k(O_E^2)   .
\end{equation}
We use the same notation for the induced map of sheaves
\begin{equation}
\label{definition of delta}
\delta_k^* \,:\, \mathfrak D_{U-k}^0(\cF') \otimes \TSym^k(\cF ) \rightarrow \mathfrak D_U(\cF') .
\end{equation}
Set $\displaystyle \ell =\frac{\log_p\kappa_U(1+p)}{\log_p(1+p)}.$ Then $\ell \in \CO_U,$ and $\ell (x)=x$ for each $x\in U\cap \Z.$ Let
\begin{equation}
\nonumber
\binom{\ell}{k}=\frac{\ell (\ell-1)\cdots (\ell-k+1)}{k!}.
\end{equation}
Then 
\begin{equation}
\nonumber
\delta_k^*\circ \beta_k^*=\binom{\ell}{k}
\end{equation}
(see \cite[Proposition~5.2.1]{LZ}). 

\newpage

\subsection{Coleman families}
\label{section Coleman families}
\subsubsection{}
 Let  $f(q)=\underset{n=1}{\overset{\infty}\sum } a_n q^n$  be a newform of   weight $k_0 \geqslant 2,$ level $N_f$ and nebentypus $\ep_f.$  We assume that the conditions {\bf C1-2)} of Section~\ref{subsection p-adic representations} hold for some fixed odd prime $p \not\vert \, N_f.$ 
Define 
\begin{equation}
\nonumber
I_f=\{x\in \Z \mid x\geqslant 2, \quad  x\equiv k_0\mod{(p-1)}\}.
\end{equation}
We identify $I_f$ with a subset of $\mathcal W^*.$
Let $U\subset \mathcal W^*$ be an open disk centered in $k_0.$
For any $F\in \CO_U$ and $x\in I_f,$ we denote by $F_x$
the value of $F$ at $x.$
For any sufficiently large $r\geqslant 1,$ we consider  $E \left <w/p^r \right >$
as the ring of analytic functions on the closed disk $D(k_0,p^{-r})\subset U.$ 
Recall that for each $F(w)\in E \left <w/p^r \right >$ we set
$\mathcal A^{\mathrm{wt}}(F)(x)=F((1+p)^{x-k_0}-1)$
(see Section~\ref{subsection A^{wt}}).   Then $F_x=\mathcal A^{\mathrm{wt}}(F)(x).$
The following proposition summarizes the main properties of 
Coleman families we need in this paper.   

\begin{myproposition} 
\label{proposition coleman families}
Assume that $v_p(\alpha (f))<k_0-1.$ 
Then  for a sufficiently small open disk $U$ centered in $k_0$ the
following conditions hold:
\medskip

1) There exists 
a unique formal power series  
\begin{equation}
\nonumber
\f=\underset{n=1}{\overset{\infty}\sum}
\ba_nq^n  \in \CO_U[[q]]
\end{equation}
with coefficients $\CO_U$ such that

1a) For each  $x \in I_f \cap U$ such that 
$v_p(\alpha (f) ) \neq x/2-1,$
the specialization $\f_{x}$ at $x$ is a $p$-stabilization of a newform $f^0_x$ of 
weight $x$ and level $N_f .$

1b) $\f_{k_0}=f_\alpha .$ 
\medskip

2) Fix $D(k_0,p^{-r})\subset U $ and denote by $A_{\f}$ its $E$-affinoid algebra. 
Let
\begin{equation}
\nonumber
W_{\f}=W(U)_{(\f)}\otimes_{\CO_U}A_{\f},
\end{equation}
where  $W(U)_{(\f)}$ is the maximal submodule of the $\CO_U$-module $W(U)$ on which the operators $T_l$ (for $(l,N_f)=1$) and $U_l$ (for $l\vert N_f$)
act as multiplication by $\ba_l$ for all primes $l.$ Then 

2a) $W_{\f}$ is a free $A_{\f}$-module of rank $2$ equipped with a continuous 
linear action  of $G_{\Q,S}.$ 

2b) The specialization of $W_{\f}$ at each integer $x\geqslant 2$ is isomorphic to  Deligne's representation associated to $\f_x .$

2c) The  $(\Ph,\Gamma)$-module $\bD_{\f}=\DdagrigAf (W_{\f})$   
has a triangulation of the form

\begin{equation}
\nonumber
0\rightarrow F^+\bD_{\f}\rightarrow \bD_{\f} \rightarrow F^- \bD_{\f} \rightarrow 0,
\end{equation}
where
\begin{equation}
\begin{aligned}
\nonumber
&F^+\bD_{\f}=\CR_{A_{\f}} (\bdelta^+_{\f}), &&
\bdelta_{\f}^+(p)=\ba_p, &&&{\bdelta_{\f}^+\vert}_{\Zp^*}=1;
\\
&F^-\bD_{\f}=
\CR_{A_{\f}} (\bdelta_{\f}^-),
&&
\bdelta_{\f}^-(p)=\ep (p) \ba_p^{-1},
&&&
{\delta_{\f}^-\vert}_{\Zp^*}= \bchi_{\f}^{-1},
\end{aligned}
\end{equation}
and 
\begin{equation}
\nonumber
\bchi_{\f}(\gamma)=\chi (\gamma)^{k_0-1}
\exp \left (\log_p(1+w)\frac{\log (\left <  (\chi (\gam)) \right >}
{\log (1+p)} \right )
\end{equation}
denotes the character $\chi^{k_0-1}\bchi$ for the algebra $A_{\f}.$
\medskip

3) Let $W'_{\f }=W'(U)_{[\,\f \,]}\otimes_{\CO_U}A_{\f},$
where  $W'(U)_{[\,\f \,]}$ denotes the maximal quotient of the $\CO_U$-module $W'(U)$ by the submodule generated by the images of $T_l'-\ba_l$ (for $(l,N_f)=1$) and $U_l'-\ba_l$ (for $l\vert N_f$). There exists
a  pairing
\[
W'_{\f}\times W_{\f}\rightarrow A_{\f},
\]
which induces a canonical isomorphism 
\[
W'_{\f}\simeq W_{\f}^*:=\Hom_{A_{\f}}(W_{\f}, A_{\f}).
\]
\end{myproposition}
\begin{proof} 1) This is the central result of Coleman's theory 
\cite{Col} together with \cite[Lemma~2.7]{Bel12}.

The statements 2a) and 2b) and 3) follow from the theory of Andreatta, Iovita and Stevens. See \cite[Theorem~4.6.6]{LZ} and the unpublished preprint \cite{Han15} for comments and  more detail. 

The statement 2c) is a theorem of Liu  \cite{Liu15}.
\end{proof}

\subsubsection{}
We say that $x\in I_f$ is {\it classical} if $v_p(\alpha ) \neq x/2-1$
and denote by $f_x^0$ the newform of level $N_f$ whose $p$-stabilization is 
$\f_x .$   For each classical weight  $x$ we have isomorphisms
\begin{equation*}
\bD_{\f}\otimes_{A_{\f}}\left (A_{\f}/\mathfrak m_x \right )\simeq 
\Ddagrig (W_{\f_x})\simeq \Ddagrig (W_{f^0_x}),
\end{equation*}
where the second isomorphism is induced by (\ref{dual isomorphism of representations  stabilized and
non stabilized }) for $f^0_x.$

The $(\Ph,\Gamma )$-module $F^+\bD_{\f}$ is crystalline of Hodge--Tate 
weight $0$ and the operator $\Ph$ acts on $\CDcris (F^+\bD_{\f})$
as multiplication by $\ba_p .$ The $(\Ph,\Gamma )$-module $F^-\bD_{\f}(\bchi_{\f})$
is crystalline of Hodge--Tate weight $-1$ and $\Ph$ acts on 
$\CDcris (F^-\bD_{\f}(\bchi_{\f}))$ as multiplication by
$\ep_f (p) p^{-1} \ba_p^{-1}.$

\subsubsection{} 

Define 
\begin{equation}
\begin{aligned}
\label{definition of C(f)}
C (f)=\left (1-\frac{\beta (f) }{p\alpha (f) } \right )\cdot
\left (1-\frac{\beta (f) }{\alpha (f) } \right ).
\end{aligned}
\end{equation}

\begin{myproposition} 
\label{proposition interpolation eigenvectors}
Let $r$ be a  sufficiently large integer. Then

1)  There exists an 
element $\boeta_{\f}\in \CDcris (F^+\bD_{\f})$ 
such that for all classical $x\in I_{f}$ the specialization 
$\boeta_{\f}(x):=\mathcal A^{\mathrm{wt}}(\boeta_{\f})(x)$ of $\boeta_{\f}$ at weight $x$ satisfies

\begin{equation}
\nonumber
\boeta_{\f}(x) =\lambda_{N_f}^{-1}(f^0_x)\,C(f^0_x)^{-1} {\Pr}_{\alpha}^* (\eta_{f^0_x}),
\end{equation}
where the map ${\Pr}_{\alpha}^*$ is defined in (\ref{dual isomorphism of representations  stabilized and
non stabilized }) and 
$C(f^0_x)$ is (\ref{definition of C(f)}) for the form $f^0_x .$

2) There exists an 
element $\bxi_{\f}\in \CDcris (F^-\bD_{\f}(\bchi_{\f}))$ 
such that for all classical $x\in I_{f}$ the specialization 
$\xi_{\f}(x):=\mathcal A^{\mathrm{wt}}(\xi_{\f})(x)$ of $\xi_{\f}$ at  $x$ satisfies
\begin{equation}
\nonumber
\bxi_{\f}(x) ={\Pr}_{\alpha}^*(\omega_{f^0_x})\otimes e_{x-1} \mod{\CDcris (F^+\bD_{\f_x}(\chi^{x-1} ))} 
\end{equation}
where $e_{x-1}$ is the canonical generator of $\CDcris (\CR_E(\chi^{x-1}))\simeq 
\Dc (E(x-1)).$

\end{myproposition}
\begin{proof}
This is Theorem 6.4.1 and Corollary 6.4.3 of \cite{LZ}.
\end{proof}

\subsubsection{} We review the construction of the overconvergent
projector defined by Loeffler and Zerbes in \cite[Section~5.2]{LZ}.
For any integer $j\geqslant 1,$ the maps $\beta_j^*$ and $\delta_j^*$,
defined by (\ref{definition of beta}) and (\ref{definition of delta}),
induce morphisms
\begin{equation}
\nonumber
\begin{aligned}
&\beta_j^*\,:\, W'(U) \rightarrow 
H^1_{\mathrm{\acute et}}\left (Y(N_f,p)_{\overline\Q}, \mathfrak D_{U-j}(\cF') \otimes \TSym^j(\cF )(1) \right ),
 \\
&\delta_j^*\,:\,H^1_{\mathrm{\acute et}}\left (Y(N_f,p)_{\overline\Q},
 \mathfrak D_{U-j}(\cF')\otimes \TSym^j(\cF )(1)\right )
 \rightarrow W'(U)
\end{aligned}
\end{equation}
such that $\displaystyle\delta_j^*\circ \beta_j^*=\binom{\ell}{j}.$

Let $\f$ be the Coleman family passing through $f_{\alpha}$
as in Proposition~\ref{proposition coleman families} and 
let
\begin{equation}
\nonumber
\pi_{\f,U}\,:\,W'(U) \rightarrow W'(U)_{[\,\f \,]}
\end{equation}
denote the canonical projection.  

\begin{myproposition} 
\label{proposition map pi_{bold f,U}}
Assume that the open set $U$ satisfies 
assumptions from Proposition~\ref{proposition coleman families}.
Then

i) The image of the composition 
\[
\pi_{\f,U}\circ \delta_j^*\,:\,H^1_{\mathrm{\acute et}}\left (Y(N_f,p)_{\overline\Q},
 \mathfrak D_{U-j}(\cF')\otimes \TSym^j(\cF )(1)\right )\rightarrow W'(U)_{[\,\f \,]}
\]
is contained  in $ \binom{\ell}{j}\, W'(U)_{[\,\f \,]}.$

ii) There exists a unique map 
\begin{equation}
\nonumber
\pi_{\f,U}^{[j]}
\,:\, 
H^1_{\mathrm{\acute et}}\left (Y(N_f,p)_{\overline\Q},
 \mathfrak D_{U-j}(\cF')\otimes \TSym^j(\cF )(1)\right )
 \rightarrow W'(U)_{[\,\f \,]}
\end{equation}
such that $=\binom{\ell}{j}\pi_{\f,U}^{[j]}= \pi_{\f,U}\circ \delta_j^*.$

iii) We have a commutative diagram
\begin{equation}
\nonumber
\xymatrix{
W'(U) \ar[ddr]^{\pi_{\f,U}} \ar[dd]^{\beta_j^*} &\\
& &\\
H^1_{\mathrm{\acute et}}\left (Y(N_f,p)_{\overline\Q},
 \mathfrak D_{U-j}(\cF')\otimes \TSym^j(\cF )(1)\right ) 
 \ar[r]_(.7){\pi_{\f,U}^{[j]}} & W'(U)_{[\,\f\,]}.
 }
\end{equation}
\end{myproposition}
\begin{proof} See \cite[Proposition~5.2.5]{LZ}.
\end{proof}

\begin{enonce*}[remark]{Remark 4.4.8} If $U$ contains none of 
the integers $\{0,1,\ldots ,j-1\},$ then the function $\binom{\ell}{j}$ is
invertible on $U.$
\end{enonce*}

If $A_{\f}$ is the affinoid algebra of a closed disk centered in $k$
as in Proposition~~\ref{proposition coleman families}, we denote by
\begin{equation}
\label{definition of pi bold }
\pi_{\f}^{[j]}\,:\, H^1_{\mathrm{\acute et}}\left (Y(N_f,p)_{\overline\Q},
 \mathfrak D_{U-j}(\cF')\otimes \TSym^j(\cF )(1)\right )\rightarrow W^*_{\f}
\end{equation}
the composition of $\pi^{[j]}_{\f,U}$ with the natural map
$W'(U)_{[\,\f\,]}\rightarrow W_{\f}^*\simeq W'(U)_{[\,\f\,]}\otimes A_{\f}.$

\section{Beilinson--Flach elements}

\subsection{Eisenstein classes}

\subsubsection{} In this section, we review the theory of Beilinson--Flach elements 
introduced first by Beilinson \cite{Bei84} and extensively studied the last years
by Bertolini, Darmon, Rotger \cite{BDR15a, BDR15b}, Lei, Loeffler, Zerbes \cite{LLZ14} and Kings, Loeffler and Zerbes \cite{KLZb, KLZ}. We follow  \cite{KLZ, KLZb} closely. 
We maintain notation of Section~\ref{Section modular curves}.
Let $N\geqslant 4$ be a fixed integer.
We denote by
\begin{equation}
\nonumber 
\Eis_{b,N}^k \in H^1_{\et}\left (Y_1(N), \TSym^k (\cF_{\Qp})(1)\right ), \qquad k\geqslant 0,
\quad b\in \Z/N\Z
\end{equation}
the \'etale realization of the Beilinson--Levin  motivic Eisenstein elements
\footnote{We normalize this element  as in \cite{KLZ}.}
constructed in \cite{BL94}. Note that for $k=0,$ we have 
\begin{equation}
\nonumber
b^2\Eis_{1,N}^0- \Eis_{b,N}^0=\partial (\,_bg_{0,1/N}),
\end{equation}
where $\partial \,:\,\mathcal O(Y_1(N))^*\rightarrow 
H^1_{\et}(Y_1( N),  {\Q_p}(1))$ denotes the Kummer map and 
$\,_bg_{0,1/N}$ is the Siegel unit as defined in \cite{Ka04}.

\subsubsection{}
Set 
\[
H^i_{\et}\left (Y_1(Np^\infty ), \TSym^k (\cF) (1)\right )=\underset{n}\varprojlim
H^i_{\et}\left (Y_1(Np^n), \TSym^k (\cF_n) (1)\right ),
\]
where the projective limit is taken with respect to the trace map.
The Siegel units $(\,_bg_{0,1/Np^n})_{n\geqslant 0}$ form a norm compatible 
system \cite{Ka04} and therefore we have a well defined element
\begin{equation}
\nonumber
\,_b\EI_{N}:=(\partial (\,_b g_{0,1/Np^n}))_{n\geqslant 0} \in 
H^1_{\et}\left (Y_1(Np^\infty ), \Z_p (1)\right ) \simeq 
H^1_{\et}\left (Y_1(N), \Lambda (\bm{\cF} \left <N\right >)(1)\right ),
\end{equation}
where $\Lambda (\bm{\cF} \left <N\right >)$ denotes the Iwasawa sheaf 
(\ref{definition of the Iwasawa sheaf})
associated to the canonical section  $s_N\,:\, Y_1(N) \rightarrow \CE [N].$ 
(Here we use the isomorphism (\ref{isomorphism for cohomology of Iwasawa sheaf}).)
Consider the map 
\begin{equation}
\nonumber
\mom^k_{\left <N \right >}\,:\, H^1_{\et}\left (Y_1(N), \Lambda (\cF\left <N\right >)(1)\right )
\xrightarrow{[N]} H^1_{\et}(Y_1(N), \Lambda (\cF)(1))
\xrightarrow{\mom^k} 
H^1_{\et}\left (Y_1(N), \TSym^k (\cF_{\Qp}) (1)\right ),
\end{equation}
where the first arrow is induced by the multiplication by $N$ 
on the universal elliptic curve 
and the second one is induced by the moment map from
Proposition~\ref{proposition about overconvergent sheaves}, iii).

The main property of the elements $\,_b\EI_{N}$ is that they interpolate
Eisenstein elements, namely  
\[
\mom_{\left <N\right >}^k (\,_b\EI_{N})=b^2\Eis_{1,N}^k-b^{-k}\Eis^k_{b,N}.
\]
We refere the reader to \cite[Theorem~4.5.1]{KLZ} for the proof and further detail.

\subsection{Rankin-Eisenstein classes}

\subsubsection{} 
Let $Y_1(N)^2=Y_1(N)\times Y_1(N).$  We denote by  $\p_i\,:\,Y_1(N)^2\rightarrow Y_1(N)$ ($i=1,2$) the projections onto the 
first and second copy of $Y_1(N)$ respectively and by
$\Delta \,:\,Y_1(N) \rightarrow  Y_1(N)^2$ the diagonal map. 
For any  integers $k,l\geqslant 0,$ we consider the sheaf on $Y_1(N)^2$ defined by 
\[
\TSym^{[k,l]}\left (\cF_{\star}\right )=\p_1^*\left (\TSym^k\left (\cF_{\star}\right )\right )\otimes 
\p_2^*\left (\TSym^{l}\left (\cF_{\star}\right )\right ), \qquad  \star\in \{n, \emptyset , \Q_p\}.
\]
Note that $\Delta^*(\TSym^{[k,l]}\left (\cF_{\star}\right ))=\TSym^k\left (\cF_{\star}\right )\otimes \TSym^{l}\left (\cF_{\star}\right ).$
Let $j$ be an integer such that 
\begin{equation}
\nonumber
0\leqslant j \leqslant \min\{k,l\}
\end{equation}
In this situation, Kings, Loeffler and Zerbes \cite[Section~5]{KLZb}
defined  a map
\[
\CGm^{[k,l,j]}\,:\,\TSym^{k+l-2j}\left (\cF_{\Qp}\right )
\rightarrow \TSym^k\left (\cF_{\Qp}\right )\otimes \TSym^{l}\left (\cF_{\Qp}\right ) (-j),
\]
called  the Clebsch--Gordan map in {\it op. cit.}.
We will use the same notation for the induced map on cohomology
\begin{equation}
\nonumber
H^1_{\et}\left ((Y_1(N), \TSym^{k+l-2j}\left (\cF_{\Qp}\right )(1)\right )\\
\rightarrow H^1_{\et}\left ((Y_1(N), \TSym^k\left (\cF_{\Qp}\right )\otimes \TSym^{l}\left (\cF_{\Qp}\right ) (1-j)\right ).
\end{equation}
Taking the composition of this map with the  Gysin map 
\[
H^1_{\et}\left (Y_1(N),\Delta^*(\TSym^{[k,l]}(\cF_{\Q_p}))(1-j)\right )\rightarrow 
H^3_{\et}\left (Y_1(N)^2, \TSym^{[k,l]}(\cF_{\Q_p}) (2-j)\right ),
\] 
induced by the diagonal embedding $\Delta,$ we get a map
\begin{equation}
\label{clebsch-gordan1}
H^1_{\et}\left (Y_1(N), \TSym^{k+l-2j}(\cF_{\Q_p})(1) \right )\rightarrow 
H^3_{\et}\left (Y_1(N)^2, \TSym^{[k,l]}(\cF_{\Q_p}) (2-j)\right ).
\end{equation}
The spectral sequence 
\[
H^r_S\left (\Q, H^s_{\et}(Y_1(N)^2_{\overline{\Q}}, \TSym^{[k,k']}(\cF_{\Q_p}) (2-j)\right )
\Rightarrow H^{r+s}_{\et}\left (Y_1(N)^2, \TSym^{[k,k']}(\cF_{\Q_p}) (2-j)\right )
\]
gives rise to a map
\[
H^3_{\et}\left (Y_1(N)^2, \TSym^{[k,l]}(\cF_{\Q_p}) (2-j)\right )
\rightarrow H^1_S\left (\Q, H^2_{\et}(Y_1(N)^2_{\overline{\Q}}, \TSym^{[k,l]}(\cF_{\Q_p}) (2-j))\right ).
\]
Composing this map with (\ref{clebsch-gordan1}), we get
a map 
\begin{equation}
\label{generalgisin}
\Delta^{[k,l,j]}\,:\,H^1_{\et}\left (Y_1(N),  \TSym^{k+l-2j}(\cF_{\Q_p})(1)\right )
\longrightarrow 
H^1_S\left (\Q, H^2_{\et}\left (Y_1(N)^2_{\overline{\Q}}, \TSym^{[k,l]}(\cF_{\Q_p}) (2-j)\right )\right ).
\end{equation}
\begin{mydefinition}
The elements 
\begin{equation}
\nonumber
\Eis_{b,N}^{[k,l,j]}=\Delta^{[k,l,j]} \left (\Eis_{b,N}^{k+l-2j}\right ),
\qquad 0\leqslant j \leqslant \min\{k,l\}, \quad b\in \Z/N\Z,
\end{equation}
are called Rankin--Eisenstein classes. 
\end{mydefinition}

\subsubsection{} We define Rankin--Iwasawa classes following \cite[Section~5]{KLZ}.
Set
\begin{equation}
\nonumber
\La (\cF \left <N \right >)^{[j]}=\La (\cF \left <N \right >)\otimes \TSym^j(\CF).
\end{equation}
From the definition of the sheaves $\La_r(\cF_r\left <N\right  >)$
(see Sections~\ref{subsubsection adic sheaves} and \ref{subsubsection sheaves F<s_N>}) it is clear that the diagonal
embeddings $\CE [p^r]\left < s_N\right  >\rightarrow \CE [p^r]\left <s_N\right  >\times_{Y_1(Np^n)} \CE [p^r]\left <s_N\right  >$ induce morphisms of sheaves
\begin{equation}
\nonumber
\La_r(\cF_r\left <N\right  >) \rightarrow \La_r(\cF_r\left <N\right  >)
\otimes \La_r(\cF_r\left <N\right  >).
\end{equation}
Tensoring this map with the Clebsch--Gordan map
\[
\CGm^{[j,j,j]}\,:\,\Zp
\rightarrow \TSym^j(\cF_{\Qp})\otimes \TSym^{j}(\cF_{\Qp}) (-j),
\]
and passing to inverse limits, we get a map
\begin{equation}
\nonumber
\La (\cF \left <N \right >) \rightarrow 
\La (\cF \left <N \right >)^{[j]}\widehat{\otimes} \La (\cF \left <N \right >)^{[j]} (-j).
\end{equation}
This induces a map on cohomology
\begin{equation}
\label{Iwasawa theoretic Clebsch-Gordan map}
H^1_{\et}\left (Y_1(N),\La (\cF \left <N \right >)(1)\right ) \rightarrow
H^1_{\et}\left (Y_1(N), \La (\cF \left <N \right >)^{[j]}\widehat{\otimes} \La (\cF \left <N \right >)^{[j]} (1-j)\right ). 
\end{equation}
Define
\[
\La (\cF \left <N \right >)^{[j,j]}=
\p_1^*\left ( \La (\cF \left <N \right >)^{[j]} \right )
\otimes 
\p_2^*\left ( \La (\cF \left <N \right >)^{[j]} \right ).
\]
Then the diagonal embedding induces the Gysin map
\begin{equation}
\nonumber
H^1_{\et}\left (Y_1(N), \La (\cF \left <N \right >)^{[j]}\widehat{\otimes} \La (\cF \left <N \right >)^{[j]} (1-j)\right )
\rightarrow
H^3_{\et}\left (Y_1(N)^2,\La (\cF \left <N \right >)^{[j,j]}(2-j)\right ).
\end{equation}
Taking the composition of this map with (\ref{Iwasawa theoretic Clebsch-Gordan map}),
we obtain a map
\begin{equation}
\label{clebsh-gordan iwasawa} 
H^1_{\et}\left (Y_1(N),\La (\cF \left <N \right >)(1)\right ) \rightarrow
H^3_{\et}\left (Y_1(N)^2,\La (\cF \left <N \right >)^{[j,j]}(2-j)\right ).
\end{equation}
Composing this map with the map
\[
H^3_{\et}\left (Y_1(N)^2,\La (\cF \left <N \right >)^{[j,j]}(2-j)\right )
\rightarrow 
H^1_S\left (\Q, H^2_{\et}\left (Y_1(N)^2_{\overline \Q},\La (\cF \left <N \right >)^{[j,j]}(2-j)\right )\right )
\]
induced by the Grothendick spectral sequence, we obtain an Iwasawa theoretic analog
of the map (\ref{generalgisin})
\begin{equation}
\nonumber
{\Delta}_{\Lambda}^{[j]}\,:\,
H^1_{\et}\left (Y_1(N),\La (\cF \left <N \right >)(1)\right ) \rightarrow 
H^1_S\left (\Q, H^2_{\et} \left (Y_1(N)^2_{\overline \Q},\La (\cF \left <N \right >)^{[j,j]}(2-j)\right )\right ).
\end{equation}

\begin{mydefinition} The elements 
\begin{equation}
\nonumber
\,_b\RI_{N}^{[j]}=\Delta_{\La}^{[j]}\left (\,_b\EI_{N}\right )
\in 
H^1_S\left (\Q, H^2_{\et}\left (Y_1(N)^2_{\overline{\Q}}, \La (\cF \left <N \right >)^{[j,j]} (2-j)\right )\right ), \qquad j\geqslant 0
\end{equation}
are called Rankin--Iwasawa classes.
\end{mydefinition}

We remark that these classes interpolate $p$-adically the elements 
$\Eis_{1,N}^{[k,l,j]}$ (see \cite[Proposition~5.2.3]{KLZ})
and refer the reader to {\it op. cit.} for the proof and further results.

\subsubsection{}
In this subsection,  we assume that $N\geqslant 4$ 
and $(p,N)=1.$ We have a commutative diagram
\[
\xymatrix{
\CE_{Np}\ar[r] \ar[d] &Y_1(Np)\ar[d]^{\pr'}\\
\CE_{N,p}\ar[r]  &Y (N,p),}
\]
where $\CE_*$ denotes the relevant universal elliptic curve and
$\pr'$ is the map defined in  Section~\ref{subsubsection Y(N,p)}.
Recall that $\CE_{Np}$ is equipped with a canonical subscheme
$D_{Np}$ of points of order $Np$ together with 
the canonical section $s_{Np}\,:\,Y_1(Np) \rightarrow D_{Np}.$
The universal curve $\CE_{N,p}$ is equipped with a canonical subscheme 
$D'$ of points of degree $p$ (see Section~\ref{subsubsection subgroups D and D'}). 
The map $\pr'$ together with   multiplication by $N$ induce  finite  morphisms
\[
\CE_{Np}[p^r]\left <s_{Np}\right > \rightarrow \CE_{N,p}[p^r]\left <D'\right >,
\qquad r\geqslant 1,
\]
and therefore we have a map
\begin{equation}
\label{definition of tr'}
\tr'_*\,:\, H^2_{\et}\left (Y_1(Np)^2_{\overline{\Q}}, \La (\cF \left <Np \right >)^{[j,j]}\right )
\rightarrow 
H^2_{\et}\left (Y(N,p)^2_{\overline{\Q}}, \La (\cF \left <D'\right >)^{[j,j]}\right ).
\end{equation}
Analogously, the map $\pr_1\,:\, Y(N,p)\rightarrow Y_1(N)$ (see (\ref{definition of pr_i})) together with multiplication by $p$ induce finite morphisms  
\[
\CE_{N,p}[p^r]\left <D'\right > \rightarrow \CE_{N}[p^r],
\qquad r\geqslant 1.
\]
This gives us a map
\begin{equation}
\nonumber
\pr_{1,*}\,:\, H^2_{\et}\left (Y_1(N,p)^2_{\overline{\Q}}, \La (\cF \left <D' \right >)^{[j,j]}\right )
\rightarrow 
H^2_{\et}\left (Y(N)^2_{\overline{\Q}}, \La (\cF)^{[j,j]}\right ).
\end{equation}

\begin{mydefinition}
We denote by
\begin{equation}
\label{definition of BF in families on Y(N,p)}
\,_b\RI_{N(p)}^{[j]}=\tr'_*\left ( \,_b\RI_{1,Np}^{[j]}\right )    \in 
H^1_S\left (\Q, H^2_{\et}\left (Y_1(N,p)^2_{\overline{\Q}}, \La (\cF \left <D'\right >)^{[j,j]} (2-j)\right )\right )
\end{equation}
the image of the Beilinson--Flach element $\,_b\RI_{Np}^{[j]}$ under the map
$\pr'_*$ induced by $\pr'.$  
\end{mydefinition}

Note that
\begin{equation}
\nonumber 
\pr_{1,*}\left (\,_b\RI_{N(p)}^{[j]}\right )= \,_b\RI_{N}^{[j]}.
\end{equation}

\subsection{Beilinson--Flach elements}
\label{subsection Beilinson--Flach elements}
\subsubsection{} 
Let $f=\underset{n=1}{\overset{\infty}\sum } a_n q^n$ and 
$g=\underset{n=1}{\overset{\infty}\sum }b_n q^n$ be two  eigenforms of weights $k_0=k+2$ and $l_0=l+2$ with $k,l\geqslant 0$ and  levels $N_f,$
$N_g$ respectively. 
By (\ref{dual Deligne's representaton}),  we have canonical projections 
\begin{equation}
\nonumber
\begin{aligned}
&\pi_f \,:\, H^1_{\et}\left (Y_1(N_f)_{\overline{\Q}}, \TSym^{k}(\cF) (1)\right )\otimes_{\Z_p}E \rightarrow W_f^*,\\
&\pi_g \,:\, H^1_{\et}\left (Y_1 (N_g)_{\overline{\Q}}, \TSym^{l}(\cF) (1)
\right )\otimes_{\Z_p}E \rightarrow W_g^*.
\end{aligned}
\end{equation}
Let  $N$ be any positive integer divisible by $N_f$ and $N_g$
and such that $N$ and $N_fN_g$ have the same prime divisors.
Without loss of generality, assume that $W_f$ and $W_g$ are defined over the same
field $E.$ Set $W_{f,g}=W_f\otimes_E  W_g .$ 
K\"unneth theorem gives an isomorphism 
\begin{equation}
\nonumber
H^2_{\et}\left (Y_1(N)^2_{\overline{\Q}}, \TSym^{[k,l]}(\cF_{\Q_p}) (2)\right )
\simeq 
H^1_{\et}\left (Y_1(N)_{\overline{\Q}}, \TSym^{k}(\cF_{\Q_p}) (1)\right )
\otimes
H^1_{\et}\left (Y_1(N)_{\overline{\Q}}, \TSym^{l}(\cF_{\Q_p}) (1)\right ).
\end{equation}
We also have the  maps  induced on cohomology by the projections
(\ref{the Pr maps}):
\begin{equation}
\nonumber
\begin{aligned}
&
H^1_{\et}\left (Y_1(N)_{\overline{\Q}}, \TSym^{k}(\cF_{\Q_p}) (1)\right )
\rightarrow 
H^1_{\et}\left (Y_1(N_f)_{\overline{\Q}}, \TSym^{k}(\cF_{\Q_p}) (1)\right ),\\
&
H^1_{\et}\left (Y_1(N)_{\overline{\Q}}, \TSym^{l}(\cF_{\Q_p}) (1)\right )
\rightarrow 
H^1_{\et}\left (Y_1(N_g)_{\overline{\Q}}, \TSym^{l}(\cF_{\Q_p}) (1)\right ).
\end{aligned}
\end{equation}
Composing K\"unneth decomposition with these projections
and $\pi_f\otimes\pi_g,$  we obtain a map
\[
\pr_{f,g}^{[j]}\,:\, 
H^2_{\et} \left (Y_1(N)^2_{\overline{\Q}}, \TSym^{[k,l]}\cF (2-j)
\right )\otimes_{\Z_p}E
\rightarrow W_{f,g}^*(-j).
\]

\begin{mydefinition}
\label{definition of Beilinson-Flach}
The elements 
\begin{equation}
\nonumber 
\BFrm_{f,g}^{[j]}=\pr_{f,g}^{[j]} \left (\Eis^{[k,l,j]}_{1,N}\right )
\in H^1_S(\Q, W_{f,g}^*(-j)), \qquad 0\leqslant j\leqslant \min\{k,l\}
\end{equation}
are called Belinson--Flach elements associated to the forms $f$ and $g.$
\end{mydefinition}

One can prove that the definition of $\BFrm_{f,g}^{[j]}$ does not depend on the choice of $N.$

\subsubsection{} In this subsection, we assume that $f$ and $g$
are newforms  of nebentypus $\ep_f$ and $\ep_g$ and   
$p$ is an odd prime  such that $(p, N_fN_g)= 1.$ We denote by $\alpha (f)$ and $\beta (f)$
(respectively by $\alpha (g)$ and $\beta (g)$) the roots of the Hecke polynomial
of $f$ (respectively $g$) at $p.$  We assume that $\alpha (f)\neq \beta (f)$
and $\alpha (g)\neq \beta (g).$
As before,  $f_{\alpha}$ and $f_{\beta}$ (resp. $g_{\alpha}$ and 
$g_{\beta }$) denote the stabilizations of $f$ (respectively $g$). 
Recall the isomorphisms (\ref{isomorphism of representations  stabilized and non stabilized }) for $f$ and $g$
\begin{equation}
\nonumber 
{\Pr}^{\alpha}_*\,\,:\,\, W_{f_{\alpha }}^* \simeq W_f^*,
\qquad 
{\Pr}^{\alpha}_*\,\,:\,\, W_{g_{\alpha }}^* \simeq W_g^*,
\end{equation}
which we denote by the same symbol to simplify notation.
These isomorphisms induce isomorphisms on Galois cohomology 

\begin{align}
\nonumber
&\left ({\Pr}_*^{\alpha}, {\Pr}_*^{\alpha} \right ) \,:\, 
H_S^1(\Q, W_{f_{\alpha},g_{\alpha}}^*) 
\rightarrow 
H_S^1(\Q, W_{f,g}^*),\\
\nonumber
&\left (\id, {\Pr}_*^{\alpha}\right )\,:\,
H_S^1(\Q, W_{f,g_{\alpha}}^*)
\rightarrow 
H_S^1(\Q, W_{f,g}^*).
\end{align}

\begin{myproposition} 
\label{proposition stabilization formulas}
For any $0\leqslant j\leqslant \min\{k,l\}$ we have

\[
\begin{aligned}
\nonumber
&i) \quad ( {\Pr}^{\alpha }_* , {\Pr}^{\alpha}_*) \left (\BFrm_{f_{\alpha},g_{\alpha}}^{[j]}\right )
=\left (1-\frac{\alpha (f)\beta (g)}{p^{j+1}} \right )
\left (1-\frac{\beta (f)\alpha (g)}{p^{j+1}} \right )
\left (1-\frac{\beta (f)\beta (g)}{p^{j+1}} \right )
\BFrm_{f,g}^{[j]}.\\
&
ii) \quad  
(\id, {\Pr}^{\alpha}_*) \left (\BFrm_{f,g_{\alpha}}^{[j]}\right )=
\left (1-\frac{\alpha (f)\beta (g)}{p^{j+1}} \right )
\left (1-\frac{\beta (f)\beta (g)}{p^{j+1}} \right )
\BFrm_{f,g}^{[j]}.
\end{aligned}
\]
\end{myproposition}
\begin{proof} The first formula  is proved in \cite[Theorem~5.7.6]{KLZ}.
The second formula is stated in Remark~7.7.7 of {\it op. cit.} 
For convenience of the reader, we give a short proof here.

Let $N=\mathrm{lcm} (N_f,N_g).$ Consider the commutative diagram
\begin{equation}
\label{diagram stabilization formula}
\xymatrix{
H^1\left (Y_1(Np)_{\overline\Q}, \TSym^{k}(\cF_{\Qp})(1)\right )
\ar[d]^{\widetilde\Pr_{1,*}} \ar[drr]^(.6){\widetilde\pi_{f_{\alpha}}} & &\\
H^1\left (Y_1(N_fp)_{\overline\Q}, \TSym^{k} (\cF_{\Qp})(1)\right )
\ar[d]^{\Pr^{\alpha}_*} \ar[rr]_(.7){\pi_{f_{\alpha}}}
& & W_{f_{\alpha}}^* \ar[d]^{\Pr^{\alpha}_*}\\
H^1\left (Y_1(N_f)_{\overline\Q}, \TSym^{k} (\cF_{\Qp})(1)\right )
\ar[rr]^(.7){\pi_f}& & W_f^*
}
\end{equation}
and the analogous commutative diagram with $f_{\beta}$ instead $f_{\alpha}.$
Here we denote by $\widetilde\Pr_{1}$ the map $(\ref{the Pr maps}) $ for 
$Y_1(Np)$ over $Y_1(N)$ to distinguish it from the map  $(\ref{the Pr maps}) $
for $Y_1(N_fp)$ over $Y_1(N_f),$ which we denote simply by $\Pr_1.$
Set 
\begin{equation}
\nonumber
\pr_f= \pi_f  \circ {\Pr}_{1,*}\circ \widetilde {\Pr}_{1,*}\,:\,
H^1\left (Y_1(Np)_{\overline\Q}, \TSym^{k}(\cF_{\Qp})(1)\right ) \rightarrow 
W_f^*.
\end{equation}
By definition, we have 
\begin{equation}
\label{stabilization proof 1st} 
\BFrm_{f,g_{\alpha}}^{[j]}= \pr_{f,g_{\alpha}}^{[j]}\left (\Eis_{1,Np}^{[k,l,j]}\right ),
\end{equation}
where the map $\pr_{f,g_{\alpha}}^{[j]}$ is induced on Galois cohomology
by the projection $(\pr_f, \pr_{g_{\alpha}})$ twisted by the  $(-j)$th power 
of the cyclotomic character.

From (\ref{definition of pr^alpha}), it follows that
\begin{equation}
\nonumber
{\Pr}_{1,*}=
\frac{\alpha (f) \cdot {\Pr}^{\alpha}_* - \beta (f)\cdot{\Pr}^{\beta}_*}
{\alpha (f)-\beta (f)}.
\end{equation}
This formula together with  the commutativity of (\ref{diagram stabilization formula}) show that 
\begin{equation}
\label{stabilization proof 2nd}
\pr_f=\pi_f \circ \left (\frac{\alpha (f)\cdot {\Pr}^{\alpha}_* - \beta (f)\cdot {\Pr}^{\beta}_*}
{\alpha (f)-\beta (f)}\right ) \circ \widetilde {\Pr}_{1,*}
=\frac{\alpha (f)\cdot \left ({\Pr}^{\alpha}_*\circ \pi_{f_{\alpha}}\right ) - 
 \beta (f) \cdot \left ({\Pr}^{\beta}_* \circ \pi_{f_{\beta}}\right )}
{\alpha (f)-\beta (f)}.
\end{equation}
From (\ref{stabilization proof 1st}) and (\ref{stabilization proof 2nd}),
we obtain that 
\begin{equation}
\nonumber 
\BFrm_{f,g_{\alpha}}^{[j]}
=\frac{1}{\alpha (f)-\beta (f)}
\left ( \alpha (f)\cdot \left ({\Pr}_*^{\alpha},\mathrm{id} \right )
\left ({\BFrm}_{f_{\alpha}, g_{\alpha}}^{[j ]}\right )-\beta (f)\cdot
\left ({\Pr}_*^{\beta},\mathrm{id} \right )
\left ({\BFrm}_{f_{\beta}, g_{\alpha}}^{[j ]}\right )
\right ),
\end{equation}
and therefore
\begin{multline}
\nonumber 
(\id, {\Pr}^{\alpha}_*) \left (\BFrm_{f,g_{\alpha}}^{[j]}\right )=\\
=\frac{1}{\alpha (f)-\beta (f)}
\left ( \alpha (f)\cdot \left ({\Pr}_*^{\alpha},{\Pr}_*^{\alpha} \right )\left ({\BFrm}_{f_{\alpha}, g_{\alpha}}^{[j ]}\right )-\beta (f) \cdot
\left ({\Pr}_*^{\beta}, {\Pr}_*^{\alpha}  \right )
\left ({\BFrm}_{f_{\beta}. g_{\alpha}}^{[j ]}\right )
\right ).
\end{multline}
Applying  part i) to compute 
$\left ({\Pr}_*^{\alpha},{\Pr}_*^{\alpha} \right )\left ({\BFrm}_{f_{\alpha}, g_{\alpha}}^{[j ]}\right )$
 and $\left ({\Pr}_*^{\beta}, {\Pr}_*^{\alpha}  \right )\left ({\BFrm}_{f_{\beta}, g_{\alpha}}^{[j ]}\right ),$ we obtain ii).

\end{proof}

\subsubsection{} We maintain previous assumptions. Let $f$ and $g$
be two newforms satisfying conditions {\bf M1-3)}. 
Consider  the composition
\begin{multline}
\label{definition of twisted moment map}
\mom_{\left <N \right >}^{[j],i}\,\,:\,\,
H^1\left (Y_1(N)_{\overline\Q}, \Lambda (\cF \left <N\right >)^{[j]} \right )
\xrightarrow{\mom_{\left <N \right >}^{i-j}\otimes {\id}}
\\
H^1 \left (Y_1(N)_{\overline\Q}, \TSym^{i-j}(\cF_{\Qp}) \otimes \TSym^{j}(\cF_{\Qp}) \right )
\rightarrow
H^1 \left (Y_1(N)_{\overline\Q}, \TSym^{i}(\cF_{\Qp})
\right  ), 
\end{multline}
where the last map is induced by the natural map 
$\TSym^{i-j}(\cF_{\Qp}) \otimes \TSym^{j}(\cF_{\Qp}) 
\rightarrow \TSym^{i}(\cF_{\Qp}).$ 
For all  $0\leqslant j \leqslant \min\{k,l\}$ we have a map
\begin{multline}
\label{composition projection of Beilinson-Flach on (f,g)}
H^1_S\left (\Q, H^2_{\et}\left (Y_1(N)_{\overline \Q},  \Lambda (\cF \left <N\right >)^{[j,j]} (2-j)\right )\right )
\xrightarrow{\left ( \mom_{\left <N \right >}^{[j],k},
\mom_{\left <N \right >}^{[j],l}\right )}\\
H^1_S\left (\Q, H^2_{\et}\left (Y_1(N)_{\overline \Q},  \TSym^{[k,l]}(\cF_{\Qp}) (2-j) \right )\right )
\xrightarrow{\pr_{f,g}^{[j]}} 
H^1_S(\Q, W_{f,g}^*(-j)).
\end{multline}

\begin{mydefinition}
For any integer $0\leqslant j \leqslant \min\{k,l\},$ 
we denote by $\,_b\BFrm_{f,g}^{[j]}$ the image of the element 
$\,_b\RI^{[j]}_{N}$ under the composition 
(\ref{composition projection of Beilinson-Flach on (f,g)}).
\end{mydefinition} 
 
We have 
\begin{equation}
\label{relation between BF and _bBF}
\,_b\BFrm_{f,g}^{[j]}=(b^2-b^{2j-k-l}\ep_f^{-1}(b) \ep_g^{-1}(b))\cdot \BFrm_{f,g}^{[j]}, \qquad (b,N_fN_g)=1.
\end{equation}
(see \cite[Proposition 5.2.3]{KLZ}).

\subsection{Stabilized Beilinson--Flach families}
\label{subsection Stabilized Beilinson--Flach families}

\subsubsection{}
Let $f$ and $g$ be two newforms. Denote by $\alpha (f)$ and $\beta (f)$
(resp. by $\alpha (g)$ and $\beta (g)$) the roots of the Hecke polyunomial 
of $f$ (resp. $g$) at $p, $ $(p,N_fN_g)=1.$
We will always assume that the following conditions  hold:
\begin{itemize}
\item[]{\bf M1)}  $\alpha (f)\neq \beta (f)$
and $\alpha (g)\neq \beta (g);$
\item[]{\bf M2)}  $v_p(\alpha (f))<k_0-1$ and $ v_p(\alpha (g)) <l_0-1.$
\end{itemize}
As before, $f_\alpha$ and $g_\alpha$ 
denote the stabilizations of $f$ and $g$ with respect to $\alpha (f)$ and $\alpha (g)$ respectively. 
Let  $\f=\underset{n=1}{\overset{\infty}\sum}
\ba_nq^n \in A_{\f}[[q]]$ and $\g=\underset{n=1}{\overset{\infty}\sum}
\bb_nq^n  \in A_{\g}[[q]]$ denote the  
Coleman families passing through $f_\alpha$ and $g_\alpha .$
We fix open disks $U_f$ and $U_g$ and  affinoid algebras $A_{\f}=E \left <w_1/p^r\right >$ and $A_{\g}=E \left <w_2/p^r\right >$
such that  the conditions of Propositions~\ref{proposition coleman families}
and \ref{proposition interpolation eigenvectors} hold for $f$ and $g.$ 
Then 
\begin{equation}
\nonumber
W_{\f,\g}=W_{\f}\widehat \otimes_E W_{\g}
\end{equation}
is a $p$-adic Galois representation of rank $4$ with coefficients 
in $A=A_{\f}\widehat \otimes_E A_{\g}\simeq E \left <w_1/p^r, w_2/p^r\right >.$
Let $N=\mathrm{lcm}(N_f,N_g).$ 


\subsubsection{} 
By Proposition~\ref{proposition about overconvergent sheaves}, ii) 
there exists a natural morphism of sheaves $\Lambda (\cF\left <D'\right >)
\rightarrow \frak D_{U_f-j}(\cF').$
It induces a morphism $\Lambda (\cF\left <D'\right >)^{[j]}\rightarrow 
\frak D_{U_f-j}(\cF')\otimes \TSym^{j}(\cF ).$ Consider the composition
\begin{align}
\label{map projection on W_{bold g}}
&H^1\left (Y(N,p)_{\overline\Q}, \Lambda (\cF\left <D'\right >)^{[j]}(1)\right )
\xrightarrow{{\Pr}^{(N,p)}_{(N_g,p)}}
H^1\left (Y(N_g,p)_{\overline\Q}, \Lambda (\cF\left <D'\right >)^{[j]}(1)\right )
\xrightarrow{\kappa_g} \\
\nonumber
& H^1\left (Y(N_g,p)_{\overline\Q},\frak D_{U_g-j}(\cF')\otimes \TSym^{j}(\cF) (1)\right )
\xrightarrow{\pi_{\g}^{[j]}} W_{\g}^*,
\end{align}
where the first map is induced by the projection $Y(N,p)\rightarrow Y(N_g,p)$ and  the last map is defined by (\ref{definition of pi bold }).
We also have the analogous morphism for the family $\f.$
Composing these maps with K\"unneth's isomorphism
\begin{equation}
\label{Kunneth isomorphism}
H^2\left (Y(N,p)_{\overline\Q}, \Lambda (\cF\left <D'\right >)^{[j,j]}(2)
\right )
\simeq H^1\left (Y(N,p)_{\overline\Q},\Lambda (\cF\left <D'\right >)^{[j]}(1)\right )^{\otimes 2}
\end{equation}
we obtain a map
\begin{equation}
\label{map defining BFfrak}
H^2\left (Y(N,p)_{\overline\Q}, \Lambda (\cF\left <D'\right >)^{[j,j]}(2)
\right ) \rightarrow W^*_{\f,\g}.
\end{equation}
This map induces a map on Galois cohomology
\begin{equation}
\pr^{[j]}_{\f,\g}\,:\, 
H^1_S\left (\Q,H^2\left (Y(N,p)_{\overline\Q}, \Lambda (\cF\left <D'\right >)^{[j,j]}(2-j)
\right ) \right )\rightarrow H^1_S\left (\Q, W^*_{\f,\g}(-j)\right ).
\end{equation}

\begin{mydefinition}
We define stabilized Beilinson--Flach classes associated to $\f$ and $\g$ by
\begin{equation}
\nonumber
\,_b\BF_{\f,\g}^{[j]}=\pr^{[j]}_{\f,\g} \left (\,_b\RI^{[j]}_{ N(p)}\right ),
\end{equation}
where $\,_b \RI^{[j]}_{N(p)}$ is the Rankin--Iwasawa element defined by
(\ref{definition of BF in families on Y(N,p)}).
\end{mydefinition}

We denote again by $\spec^{\f,\g}_{x,y}\,:\,H^1_S\left (\Q, W^*_{\f,\g}(-j)\right )
\rightarrow H^1_S\left (\Q, W^*_{\f_x,\g_y}(-j)\right )$ the  morphism induced by the specialization map $ W^*_{\f,\g}\rightarrow  W^*_{\f_x,\g_y}.$

\begin{myproposition}
\label{proposition specialization of two variable Beilinson Flach elements}
 i) For all integers $x,y$ such that  $0\leqslant j\leqslant \min \{x,y\}-2$
one has 
\begin{equation}
\nonumber
\binom{x-2}{j} \cdot \binom{y-2}{j}\cdot
\spec^{\f,\g}_{x,y}\left (\,_b\BF_{\f,\g}^{[j]}\right )=
\,_b\BFrm_{\f_x,\g_y}^{[j]}.
\end{equation}

ii) Let $\lambda= v_p(\alpha (f))+ v_p(\alpha (g)).$ There exists a unique element 
\begin{equation}
\nonumber
\,_b\BF_{\f,\g}^{\Iw}\in H^1_{\Iw,S}(\Q, W^*_{\f,\g})\otimes_{\Lambda} \CH_E^{[\lambda ]}(\Gamma)
\end{equation}
such that for any integer $j\geqslant 0$
one has
\begin{equation}
\nonumber
{\spec}^c_{-j}
 \left (\,_b\BF_{\f,\g}^{\Iw}\right )=
\frac{(-1)^j}{j!}\left (1-\frac{p^j}{\ba_p \bb_p} \right )
\,_b\BF_{\f,\g}^{[j]}.
\end{equation}
\end{myproposition}
\begin{proof} i) The first statement follows directly from the definition of 
the maps $\pi_{\f}^{[j]}$ and $\pi_{\g}^{[j]}$ (see \cite[Proposition~5.3.4]{LZ}).

ii) The second statement is proved in \cite[Proposition~2.3.3 and the proof of Theorem~5.4.2]{LZ}.
\end{proof} 

\newpage
\subsection{Semistabilized Beilinson--Flach elements}
\subsubsection{} Define
\begin{equation}
\nonumber
W_{f,\g}=W_f\otimes_E W_{\g}.
\end{equation}
This is a $p$-adic representation of $G_{\Q,S}$ with coefficients in $A_{\g}.$
For any $0\leqslant j\leqslant k,$ consider the 
composition of maps
\begin{multline}
\label{map projection on W_f}
H^1\left (Y(N,p)_{\overline\Q}, \Lambda (\cF\left <D'\right >)^{[j]}(1)\right )
\xrightarrow{\mom^{k-j}_{\left <p\right >} \otimes \id }
H^1\left (Y(N,p)_{\overline\Q}, \TSym^{k}(\cF_{\Qp})(1)\right )
\\
\xrightarrow{{\Pr}^{(N,p)}_{N_f}} 
H^1\left (Y(N_f)_{\overline\Q}, \TSym^{k}(\cF_{\Qp})(1)\right )
\xrightarrow{\pi_f}  W_f^*.
\end{multline}
Composing K\"unneth's isomorphism (\ref{Kunneth isomorphism}) with 
(\ref{map projection on W_{bold g}}) and (\ref{map projection on W_f}), 
we obtain a map
\begin{equation}
\label{map defining BFfrak}
H^2\left (Y(N,p)_{\overline\Q}, \Lambda (\cF\left <D'\right >)^{[j,j]}(2)
\right ) \rightarrow W^*_{f,\g}.
\end{equation}
We denote by
\begin{equation}
\nonumber
\pr^{[j]}_{f,\g}\,:\, H^1_S\left (\Q,H^2\left (Y(N,p)_{\overline\Q}, \Lambda (\cF\left <D'\right >)^{[j,j]}(2-j)
\right ) \right )\rightarrow H^1_S\left (\Q, W^*_{f,\g}(-j)\right )
\end{equation}
the induced map on Galois cohomology. 

\begin{mydefinition} Assume that $0\leqslant j \leqslant k.$ The elements 
\begin{equation}
\,_b\BFfrak_{f,\g}^{[j]}=\pr^{[j]}_{f,\g} \left (\,_b \RI^{[j]}_{N(p)}  \right ),
\end{equation}
will be called  semistabilized Beilinson--Flach elements.
\end{mydefinition}

\subsubsection{}  For each  $y\in \Spm (A_{\g}),$  we denote again by 
\[
\spec^{\g}_{y}\,:\,H^1_S (\Q, W^*_{f,\g}(-j) ) \rightarrow
H^1_S (\Q, W^*_{f,\g_y}(-j) )
\]
the morphism induced by  the specialization
map $\spec^{\g}_{y}\,:\, W_{\g} \rightarrow W_{\g_y}.$  Recall that
\begin{equation}
\nonumber
I_g=\{ y\in \Z \cap \Spm (A_{\g}) \mid y\geqslant 2, \quad y\equiv l_0\mod{(p-1)}\}.
\end{equation}

\begin{myproposition} 
\label{proposition interpolation of semistabilized classes}
i) For each $y\in I_g$ such that $y\geqslant j+2$
we have
\begin{equation}
\nonumber
\,_b\BFrm_{f,\g_y}^{[j]}=\binom{y-2}{j}\cdot \spec^{\g}_y 
\left (\,_b\BFfrak_{f,\g}^{[j]} \right).
\end{equation}
ii) In particular,
\begin{equation}
\nonumber
\binom{y-2}{j}\cdot 
\left (\id, {\Pr}^{\alpha}_* \right )\circ \spec^{\g}_y 
\left (\,_b\BFfrak_{f,\g}^{[j]} \right)
=\left (1-\frac{\alpha (f) \cdot \beta (g^0_y)}{p^{j+1}} \right )
\left (1-\frac{\beta (f)\cdot \beta (g_y^0)}{p^{j+1}} \right )
\,_b\BFrm_{f,g_y^0}^{[j]}.
\end{equation}
\end{myproposition}
\begin{proof} 
i) Since the moment maps commute with the traces (see \cite[Proposition~2.2.2]{Ki15}),  we have a commutative diagram
{
\begin{equation}
\nonumber
\xymatrix{
H^2_{\et} \left (Y_1(Np)_{\overline\Q},\Lambda (\cF \left < Np \right >)^{[j,j]}\right )
\ar[rrr]^{\left (\mom^{[j],k}_{\left < Np \right >},\mom^{[j],y-2}_{\left <Np\right >}\right )}  \ar[d]^{\tr'_*}
&&&H^2_{\et} \left (Y_1(Np)_{\overline\Q},\TSym^{[k,y-2]}(\cF_{\Qp})   \right )
\ar[d]^{\pr'_*}
 \\
H^2_{\et} \left (Y_1(N,p)_{\overline\Q},\Lambda (\cF \left <D'\right >)^{[j,j]}\right )
\ar[rrr]^{\left (\mom^{[j],k}_{\left < p \right >},\mom^{[j],y-2}_{\left < p\right >}\right )}  
&&&H^2_{\et} \left (Y_1(N,p)_{\overline\Q},\TSym^{[k,y-2]}(\cF_{\Qp})   \right )
}
\end{equation}
}
Here the maps $\mathrm{tr}'_*$ and $\mom^{[j],*}_{\left <* \right >}$
are defined by  (\ref{definition of tr'}) and (\ref{definition of twisted moment map}) respectively. Taking into account (\ref{definition of BF in families on Y(N,p)}) and 
the definition of $\,_b\BFrm_{f,\g_y}^{[j]},$ we obtain that 
\begin{equation}
\label{formula comparision BF}
\,_b\BFrm_{f,\g_y}^{[j]}=(\pi_f,\pi_{\g_y})\circ \left ({\Pr}^{(N,p)}_{N_f}, {\Pr}^{(N,p)}_{(N_g,p)}\right )\circ \left (\mom^{[j],k}_{\left <p \right >},  \mom^{[j],y-2}_{\left <
p \right >}\right ) \left (\BF^{[j]}_{1,N(p)}\right ).
\end{equation}
The elements $\,_b\BFrm_{f,\g_y}^{[j]}$ and 
$\,_b\BFfrak_{f,\g}^{[j]}$
are defined via the maps   (\ref{composition projection of Beilinson-Flach on (f,g)}) and (\ref{map defining BFfrak}) respectively.

Consider the following   diagram
\begin{equation}
\label{big diagram section semistabilized Beilinson-Flach}
\xymatrix{
H^1_{\et}(Y(N_g,p)_{\overline\Q}, \Lambda (\cF\left <D'\right >)^{[j]}(1))
\ar[r]^{\mom^{[j],y-2}_{\left < p\right >}}
\ar[d]^{\alpha}
&
H^1_{\et}(Y(N_g,p)_{\overline\Q}, \TSym^{y-2}(\cF_{\Qp})(1))
\ar[d]_{\pi_{\g_y}}
\\
H^1_{\et}\left (Y(N_g,p)_{\overline\Q},\frak D_{U_g-j}(\cF')\otimes \TSym^{j}(\cF) (1)\right )
\ar[ur]^{\theta_{y-j-2}\otimes \id}
\ar[d]^{\delta^*_{j}}
\ar[dr]^(.6){{\binom{\ell}{j}\pi_{\g}^{[j]}}}
& W^*_{\g_y}\\
W'(U) \ar[r]^{\pi_{\g}} &W_{\g}^* \ar[u]^{\spec_y}
}
\end{equation}

Compairing (\ref{formula comparision BF}) with (\ref{map projection on W_{bold g}})
and (\ref{map projection on W_f}), it is easy to see that we only need to prove 
the formula
\begin{equation}
\label{formula proof of main relation between BF}
\pi_{\g_y}\circ \mom^{[j],y-2}_{\left <p \right >}=\binom{y-2}{j}\spec^{\g}_y\circ 
\pi_{\g}^{[j]}\circ \alpha,
\end{equation}
which in turn follows from the commutativity of the diagram (\ref{big diagram section semistabilized Beilinson-Flach}).
The commutativity of the upper triangle follows from Proposition~\ref{proposition about overconvergent sheaves}, iii). The commutativity of the lower triangle 
follows from Proposition~\ref{proposition map pi_{bold f,U}}.
Directly from the definition of the maps $\theta_{y-j-2},$ $\delta^*_{j}$ 
and $\pi_{\g}$ it follows that $\pi_{\g_y}'\circ (\theta_{y-j-2}, \id)=
\spec^{\g}_y\circ \pi_{\g}\circ \delta^*_{j}.$ Therefore, the diagram 
(\ref{big diagram section semistabilized Beilinson-Flach})
commutes, and   (\ref{formula proof of main relation between BF}) is proved.

ii) The second statement follows from i), Proposition~\ref{proposition stabilization formulas}, ii) and (\ref{relation between BF and _bBF}).
\end{proof}

In the following proposition, we compare  Beilinson--Flach Euler systems 
$\,_b\BF_{[\f,\g]}^{\Iw}$ with semistabilized Beilinson--Flach elements. 
This result plays a key role in the proof of the main theorem of this paper.

\begin{myproposition} 
\label{proposition Iwasawa vs semistabilized classes}
Let
\[
\,_b\BFfrak^{\Iw}_{f,\g}=({\Pr}^\alpha_*, \id)\circ \spec^{\f}_{k_0}\left (_b\BF^{\Iw}_{\f,\g}\right ) \in H^1_{\Iw,S}\left (\Q, W_{f,\g}^* \right )\otimes_{\Lambda}\CH_E^{[\lambda]}(\Gamma).
\]
Then for any integer $0\leqslant j\leqslant \min \{k,l\}$ there exists
a neighborhood $U_g=\Spm (A_{\g})$ such that 
\[
\spec^c_{-j}
\left (\,_b\BFfrak^{\Iw}_{f,\g} \right )=
\frac{(-1)^j}{j!} \cdot \binom{k}{j} \cdot \left (
1-\frac{p^j}{\alpha (f)\cdot \bb_p}
\right )\cdot 
\left (1-\frac{\beta (f)\cdot \bb_p}{p^{j+1}}\right ) \cdot
\,_b\BFfrak_{f,\g}^{[j]}  
\]
\end{myproposition}
\begin{proof}  Shrinking, if necessarly, the neighborhood $U_{\g},$ we can assume that, as an $A_{\g}$-module,  $H^1_S(\Q, W_{f,\g})\simeq A_{\g}^r\oplus T,$ where $T$ is a  $\mathfrak m_{l_0}$-primary torsion module. 
Let $y\in \Spm (A_{\g})$ be an integral weight such that
$y\geqslant l_0.$ From Propositions~\ref{proposition stabilization formulas} ii),  \ref{proposition specialization of two variable Beilinson Flach elements}
and \ref{proposition interpolation of semistabilized classes} it follows that
the map $(\id, {\Pr}^{\alpha}_*)\circ \spec^{\g}_y$ sends the both sides 
of the formula to
\begin{multline}
\nonumber
\frac{(-1)^j}{j!}\cdot \binom{k}{j}\cdot \binom{y-2}{j}\cdot 
\left (
1-\frac{p^j}{\alpha (f) \alpha(g^0_y)}
\right )\cdot 
\left (1-\frac{\alpha (f) \beta (g^0_y)}{p^{j+1}}\right ) \times 
\\
\times \left (1-\frac{\beta (f)   \alpha (g^0_y)}{p^{j+1}}\right )\cdot 
\left (1-\frac{\beta (f) \beta (g^0_y)}{p^{j+1}}\right )
\,_b\BF_{f,g^0_y}^{[j]}.
\end{multline}
Since ${\Pr}^\alpha_*$ is an isomorphism, this shows that the specializations
of the both sides coincide at infinitely many points, including $l_0.$ 
This proves the proposition.
\end{proof}

\newpage
\section{Triangulations}
\label{section triangulations}

\subsection{Triangulations}
\label{subsection triangulations}
\subsubsection{}
Let $f=\underset{n=1}{\overset{\infty}\sum } a_n q^n$ and 
$g=\underset{n=1}{\overset{\infty}\sum }b_n q^n$ be two eigenforms of weights $k_0=k+2$ and $l_0=l+2,$ levels $N_f,$
$N_g$ and nebentypus $\ep_f$ and $\ep_g$ respectively. Fix an odd prime $p$
such that $(p, N_fN_g)= 1$ and denote by $\alpha (f)$ and $\beta (f)$
(respectively by $\alpha (g)$ and $\beta (g)$) the roots of the Hecke polynomial
of $f$ (respectively $g$) at $p.$  We will always assume that the  conditions {\bf M1-2)} from Section~\ref{subsection Stabilized Beilinson--Flach families}  hold.
Let $f_\alpha$ and $g_\alpha$ denote the $p$-stabilizations of $f$ and $g$ with respect to $\alpha (f)$ and $\alpha (g)$ respectively. 
Denote by $\f=\underset{n=1}{\overset{\infty}\sum}
\ba_nq^n \in A_{\f}[[q]]$ and $\g=\underset{n=1}{\overset{\infty}\sum}
\bb_nq^n  \in A_{\g}[[q]]$
Coleman families passing through $f_\alpha$ and $g_\alpha .$
Shrinking the neighborhoods of $k_0$ and $l_0$ in the weight space, we can assume that the affinoid algebras $A_{\f}=E \left <w_1/p^r\right >$ and $A_{\g}=E \left <w_2/p^r\right >$
satisfy the conditions of Propositions~\ref{proposition coleman families}
and \ref{proposition interpolation eigenvectors}. 
Then 
\begin{equation}
\nonumber
W_{\f,\g}=W_{\f}\widehat \otimes_E W_{\g}
\end{equation}
is a $p$-adic Galois representation of rank $4$ with coefficients 
in $A=A_{\f}\widehat \otimes_E A_{\g}\simeq E \left <w_1/p^r, w_2/p^r\right >.$
 Let $\bD_{\f}=\DdagrigAf (W_{\f})$ and $\bD_{\g}=\DdagrigAg (W_{\g})$
and let 
\begin{equation}
\begin{aligned} 
\label{triangulations of D_f and D_g}
& 0\rightarrow F^+\bD_{\f} \rightarrow \bD_{\f} \rightarrow F^-\bD_{\f}
\rightarrow 0,\\
& 0\rightarrow F^+\bD_{\g} \rightarrow \bD_{\g} \rightarrow F^-\bD_{\g}
\rightarrow 0
\end{aligned}
\end{equation}
be the canonical  triangulations of $\bD_{\f}$ and $\bD_{\g}$ respectively
(see Proposition~\ref{proposition coleman families}, 2c)).
We denote by $\boeta_{\f}$ and $\bxi_{\f}$  (respectively by $\boeta_{\g}$ and $\bxi_{\g}$) the elements constructed in Proposition~\ref{proposition interpolation eigenvectors}.
Set $\bD_{\f,\g}=\bD_{\f}\widehat\otimes_{\CR_E} \bD_{\g}.$ Then 
$\bD_{\f,\g}=\DdagrigA (W_{\f,\g}).$ We denote by $(F_i \bD_{\f,\g})_{i=-2}^2$
the triangulation 
\begin{equation}
\nonumber
\{0\} \subset F_{-1}\bD_{\f,\g} \subset F_0\bD_{\f,\g}
\subset F_1\bD_{\f,\g} \subset F_2\bD_{\f,\g}
\end{equation}
given by
\begin{equation}
\label{definition of triangulation}
F_i\bD_{\f,\g}=
\begin{cases}
\{0\}, &\text{if $i=-2,$}\\
F^+\bD_{\f}\widehat\otimes_{\CR_E} F^+\bD_{\g}, &\text{if $i=-1,$}\\
F^+\bD_{\f}\widehat\otimes_{\CR_E} \bD_{\g}, &\text{if $i=0,$}\\
(F^+\bD_{\f}\widehat\otimes_{\CR_E} \bD_{\g})+ 
(\bD_{\f}\widehat\otimes_{\CR_E} F^+\bD_{\g}), &\text{if $i=1,$}\\
\bD_{\f,\g}, &\text{if $i=2.$}
\end{cases}
\end{equation}
We will denote by $(\gr_i \bD_{\f,\g})_{i=-2}^2$ the associated graded
module. In particular,
\begin{equation}
\nonumber
\gr_0 \bD_{\f,\g}=F^+\bD_{\f}\widehat\otimes_{\CR_E} F^-\bD_{\g},
\qquad 
\gr_1 \bD_{\f,\g}=F^-\bD_{\f}\widehat\otimes_{\CR_E} F^+\bD_{\g}.
\end{equation}
Note that 
\begin{equation}
\begin{aligned}
\label{formulas for phi action on M}
&\gr_0 \bD_{\f,\g}\simeq \CR_A(\bdelta_{\f,\g} \bchi_{\g}^{-1}), 
\qquad 
\bdelta_{\f,\g}(p)=\ep_{g}(p)\ba_{p}\bb_{p}^{-1}, \qquad
{\bdelta_{\f,\g}\vert }_{\Zp^*}=1,\\
\end{aligned}
\end{equation}

\subsubsection{} We denote by $(F_i\bD_{\f,\g}^*)_{i=-2}^2$ the dual filtration
on $\bD^*_{\f,\g}$
\[
F_i\bD_{\f,\g}^*=\Hom_{\CR_A} 
\left (\bD_{\f,\g}/F_{-i}\bD_{\f,\g}, \CR_A\right ).
\]
\begin{mylemma} 
\label{lemma duality of filtration}
Let $\alpha (f^*)=p^{k_0-1}/\beta (f)$ and 
$\alpha (g^*)=p^{l_0-1}/\beta (g)$ and let $\f^*$ and $\g^*$ denote
the Coleman families passing through the stabilizations of $f^*$ and 
$g^*$ with respect to $\alpha (f^*)$ and $\alpha (g^*).$ Then the 
filtrations $(F_i\bD_{\f,\g}^*)_{i=-2}^2$ and  $(F_i\bD_{\f^*,\g^*})_{i=-2}^2$
are compatible with the duality $\bD_{\f^*,\g^*}
\times  \bD_{\f,\g}\rightarrow \CR_A(\bchi_{\f}^{-1}\bchi_{\g}^{-1}).$ Namely
\begin{equation}
F_i\bD_{\f,\g}^*\simeq F_i\bD_{\f^*,\g^*}  (\bchi_{\f}\bchi_{\g}).
\end{equation} 
\end{mylemma}
\begin{proof} The proof is straightforward and left to the reader.
\end{proof} 

\subsubsection{}
Set 
\begin{equation}
\label{definition of M_f,g}
\bM_{\f,\g}=\gr_0\bD_{\f,\g} (\chi\bchi_{\g})\simeq \CR_A(\bdelta_{\f,\g}\chi ).
\end{equation}
Then $\bM_{\f,\g}$ is a crystalline $(\Ph,\Gamma)$-module of Hodge--Tate weight $1$
\footnote{We call Hodge--Tate weights the jumps of the Hodge--Tate filtration on the associated filtered module. In particular, the Hodge--Tate weight of $\Qp (1)$
is $-1.$} 
and $\CDcris (\bM_{\f,\g})$ is a free $A$-module of rank one generated by
\begin{equation}
\label{definition of m bold f bold g}
m_{\f,\g}:=\boeta_{\f}\otimes \bxi_{\g}\otimes e_1,
\end{equation}
where $\boeta_{\f}$ and $\bxi_{\g}$ are defined in Proposition~\ref{proposition interpolation eigenvectors}. 
From Lemma~\ref{lemma duality of filtration} it follows  that  
$\gr_1\bD_{\f^*,\g^*}(\bchi_{\f})$ is
the Tate dual  of $\bM_{\f,\g}:$
\begin{equation}
\label{definition of M_f,g*(chi)}
\bM_{\f,\g}^*(\chi)=\gr_1 \bD_{\f,\g}^*( \bchi^{-1}_{\g}) \simeq 
\gr_1\bD_{\f^*,\g^*}(\bchi_{\f})\simeq
\CR_A(\bdelta_{\f,\g}^{-1}).
\end{equation} 


\subsubsection{} Let 
\[
\bD_{f,\g}=\bD^{\dagger}_{\mathrm{rig}, A_{\g}} (W_{f,\g})
\simeq \DdagrigE (W_f)\otimes_{\CR_E}\bD^{\dagger}_{\mathrm{rig}, A_{\g}} (W_{\g}).
\]  
The isomorphism (\ref{dual isomorphism of representations  stabilized and
non stabilized }) identifies $\bD_{f,\g}$ with 
the specialization of
the $(\Ph,\Gamma)$-module $\bD_{\f,\g}$ at $f_\alpha.$ In particular, for each
$j\in \Z$ we have a tautological short exact sequence 
\begin{equation}
\nonumber
0\rightarrow \gr_1\bD_{f,\g}^*(\chi^{j})  \rightarrow 
\bD_{f,\g}^*/F_0\bD_{f,\g}^*  (\chi^{j})  \rightarrow 
\gr_2\bD_{f,\g}^*(\chi^{j})\rightarrow 0.
\end{equation}

\begin{mylemma} 
\label{lemma exact sequence triangulation}
1) For each $j\in \Z$  the induced sequence 
\begin{equation}
\nonumber
0\rightarrow H^1\left (\gr_1\bD_{f,\g}^*(\chi^{j}) \right ) \rightarrow 
 H^1 \left ( \bD_{f,\g}^*/F_0\bD_{f,\g}^*   (\chi^{j}) \right ) \rightarrow 
 H^1 \left (\gr_2\bD_{f,\g}^*(\chi^{j})\right )
\end{equation}
is exact.

2) Assume that $\displaystyle j\neq 1-\frac{k_0+l_0}{2}.$  Then for a sufficiently 
small neighborhood $U_g$ the $A_{\g}$-module $H^1 \left (\gr_2\bD_{f,\g}^*(\chi^{j})\right )$ is free
of rank one. 

3) Assume that $j\leqslant 0$ and $y\in I_g$ is such that
\[
\frac{k_0+y}{2}\not\in \{1-j, 2-j\}.
\]
Then $H^1_g\left (\gr_2\bD_{f,\g_y}^*(\chi^{j})\right )=0.$

\end{mylemma}
\begin{proof} 1) We only need to prove that  $H^0\left (\gr_2\bD_{f,\g}^*(\chi^{j})\right )=0.$ From Proposition~\ref{proposition coleman families}
it follows  that  $\gr_2\bD_{f,\g}^*\simeq \CR_{A_{\g}}(\psi_{f,\g}),$
where $\psi_{f,\g}(p)=\alpha (f)^{-1}\bb_p^{-1}$ and $\left. \psi_{f,\g}\right \vert_{\Zp^*}=1.$ By Proposition~\ref{proposition cohomology of rank 1 modules in families}, it is sufficient to check that $\psi_{f,\g}(p)p^{-j}\neq 1 ,$
but this follows from the fact that $\vert \bb_p(y) \vert=p^{\frac{y-1}{2}}$
for any $y\in I_g.$

2) We have 
\[
\left (\psi_{f,\g}(p)p^{-j}\right ) (l_0)=
\alpha(f)^{-1}\alpha (g)^{-1}p^{-j}.
\]
If $\displaystyle j\neq 1-\frac{k_0+l_0}{2},$ then
\[
\left \vert \alpha(f)^{-1}\alpha (g)^{-1}p^{-j}\right \vert=
p^{-j-\frac{k_0-1}{2}-\frac{l_0-1}{2}}=p^{-j+1-\frac{k_0+l_0}{2}}\neq 1.
\]
Therefore $\psi_{f,\g}(p)p^{-j}-1\in A_{\g}$ does not vanish at $l_0,$
and 2) follows from Proposition~\ref{proposition cohomology of rank 1 modules in families}, 2b).

3) By \cite[Corollary ~1.4.5]{Ben11}, $H^1_g\left (\gr_2\bD_{f,\g_y}^*(\chi^{j})\right )=0$ if the following conditions hold:
\begin{itemize}
\item[a)]{} $j\leqslant 0.$

\item[b)]{} $H^0\left (\gr_2 \bD_{f,\g_y}^*(\chi^{j})\right )=0.$

\item[c)]{} $\CDcris \left (\gr_2 \bD_{f,\g_y}^*(\chi^{j})\right )^{\Ph=p^{-1}}=0.$
\end{itemize}
Since $\Ph$ acts on $\CDcris \left (\gr_2\bD_{f,\g_y}^*(\chi^{j})\right )$
as the multiplication by $\alpha(f)^{-1}\bb (y)^{-1}p^{-j},$ the same argument
as for 2) applies and shows that b) and c) hold.
\end{proof}

\subsection{Local properties of Beilinson--Flach elements}
\label{subsection Local properties of Beilinson--Flach elements}
\subsubsection{} We maintain previous notation and conventions. 
Fix $0\leqslant j\leqslant k.$
Consider the diagram
\begin{equation}
\label{diagram semistabilized zeta}
\xymatrix{
&   &  H^1(\Qp, W_{f,\g}^*(-j)) \ar[d] &   \\
0 \ar[r] &H^1\left (\gr_1\bD_{f,\g}^*(\chi^{-j})\right ) \ar[r]
&H^1\left (\bD_{f,\g}^*/F_0 \bD_{f,\g}^* (\chi^{-j})\right ) \ar[r]
&H^1\left (\gr_2\bD_{f,\g}^*(\chi^{-j})\right). 
}
\end{equation}
Recall that the bottom row is exact by Lemma~\ref{lemma exact sequence triangulation}.
Let 
$\res_p \left (\,_b\BFfrak_{f,\g}^{[j]}\right )\in H^1(\Qp, W_{f,\g}(-j)) $
denote the image of the semistabilized Beilinson--Flach element 
under the localization map. 

\begin{mydefinition} We denote by $\,_b\Zfrak_{f,\g}^{[j]}$ the image 
of $\res_p \left (\,_b\BFfrak_{f,\g}^{[j]}\right )$ under the map
\[
H^1(\Qp, W_{f,\g}^*(-j))\rightarrow H^1\left (\bD_{f,\g}^*/F_0\bD_{f,\g}^*(\chi^{-j})\right ).
\]
\end{mydefinition}

\begin{myproposition} 
\label{proposition local property os semistabilized BF}
Assume that $\displaystyle j\neq \frac{k_0+l_0}{2}-1.$
Then for a sufficiently small neighborhood $U_g$ of $l_0$
the image of $\,_b\Zfrak_{f,\g}^{[j]}$ under the map
$H^1\left (\bD_{f,\g}^*/F_0\bD_{f,\g}^* (\chi^{-j})\right ) \rightarrow  
H^1\left (\gr_2\bD_{f,\g}^*(\chi^{-j})\right )$ is zero, and therefore
\[
\,_b\Zfrak_{f,\g}^{[j]}\in H^1\left (\gr_1\bD_{f,\g}^*(\chi^{-j})\right ).
\]
\end{myproposition}
\begin{proof} 
Let $z$ denote the image of $\,_b\Zfrak_{f,\g}^{[j]}$ in $H^1\left (\gr_2\bD_{f,\g}^*(\chi^{-j})\right )$ under the natural projection.
From Lemma~\ref{lemma exact sequence triangulation}, 3) it follows that,
for a sufficiently small $U_g,$ the cohomology  $H^1\left (\gr_2\bD_{f,\g}^*(\chi^{-j})\right )$ is a free module of rank one over the principal ideal domain $A_{\g}.$

On the other hand, by  \cite[Proposition~5.4.1]{KLZ} 
\[
\res_p \left (\,_b\BFrm_{f,g}^{[j]}\right )\in H^1_f(\Qp, W^*_{f,g}(-j)).
\]
 Taking into account Proposition~\ref{proposition interpolation of semistabilized classes}, we obtain that 
\begin{equation}
\nonumber
\spec^{\g}_y\left (\res_p \left (\,_b\BFfrak_{f,\g}^{[j]}\right )\right )\in 
H^1_g(\Qp, W^*_{f,\g_y}(-j)), \qquad \forall y\in I_g \quad\text{such that $y\geqslant j+2.$}
\end{equation}
Hence
\begin{equation}
\nonumber 
\spec^{\g}_y (z )\in H^1_g\left(\gr_2\bD_{f,\g_y}^*(\chi^{-j})\right), \qquad \forall y\in I_g \quad\text{such that $y\geqslant j+2.$}
\end{equation}
From Lemma~\ref{lemma exact sequence triangulation}, 3) it follows that 
$H^1_g\left(\gr_2\bD_{f,\g_y}^*(\chi^{-j})\right)=0$ for infinitely many
values of $y\in I_g.$  Therefore $\spec^{\g}_y(z)=0$ at infinitely many
values of $y\in I_g.$ Since   $H^1\left (\gr_2\bD_{f,\g}^*(\chi^{-j})\right )$
is free by Lemma~\ref{lemma exact sequence triangulation}, 3), this implies that $z=0.$
\end{proof}

\subsubsection{} 
\label{subsubsection localisation of Beilinson-Flach}
In this subsection, we record a corollary of 
Proposition~\ref{proposition local property os semistabilized BF}.
Assume that $f$ and $g$ are of the same weight $k_0.$ Since $W_{f_\alpha, g_\alpha}
\simeq W_{f,g},$ the specialization of the triangulation $\left (F_i\bD_{\f,\g}\right )_{i=-2}^2$ at $(k_0,k_0)$ defines a triangulation $\left (F_i\bD_{f,g}\right )_{i=-2}^2$ of  $\bD_{f,g}.$ It is clear, that this triangulation can be defined 
directly in terms of $W_{f,g}$ by formulas analogous to (\ref{definition of triangulation}).

Consider the diagram
\begin{equation}
\label{diagram non stabilized zeta}
\xymatrix{
& & H^1(\Qp, W_{f,g}(k_0)) \ar[d] &   \\
0\ar[r] &H^1\left (\gr_1\bD_{f,g}(\chi^{k_0})\right ) \ar[r] 
&H^1\left (\bD_{f,g}/F_0 \bD_{f,g}(\chi^{k_0})\right ) \ar[r] 
&H^1\left (\gr_2\bD_{f,g}(\chi^{k_0})\right).
}
\end{equation}
Using the  canonical isomorphism $W_{f^*,g^*}^*(2-k_0)\simeq W_{f,g} (k_0),$ we can  consider $\BFrm_{f^*,g^*}^{[k_0-2]}$   as an element 
\[ 
\BFrm_{f^*,g^*}^{[k_0-2]}\in H^1_S(\Q, W_{f,g}(k_0)).
\] 

\begin{mydefinition} 
\label{definition of Zrm}
We denote by $\Zrm_{f^*,g^*}^{[k_0-2]}$ the image of $\res_p \left (\BFrm_{f^*,g^*}^{[k_0-2]}\right )\in H^1(\Qp, W_{f,g}(k_0))$ in $H^1\left (\bD_{f,g}/F_0\bD_{f,g}(\chi^{k_0})\right ).$
\end{mydefinition}

\begin{mycorollary} 
\label{corollary about Zrm}
Assume that $\alpha (f) \alpha (g)\neq p^{k_0-1}$
and $\beta (f)\alpha (g)\neq p^{k_0-1}.$ Then 

1) $H^1_f \left (\gr_1\bD_{f,g}(\chi^{k_0}))
=H^1(\gr_1\bD_{f,g}(\chi^{k_0})\right );$

2) $\Zrm_{f^*,g^*}^{[k_0-2]}\in H^1\left (\gr_1\bD_{f,g}(\chi^{k_0})\right ).$
\end{mycorollary}
\begin{proof}  
1) The $(\Ph,\Gamma)$-module $\gr_1\bD_{f,g}(\chi^{k_0})$ is of 
Hodge--Tate weight $-1,$ and $\Ph$ acts on  $\CDcris (\gr_1\bD_{f,g}(\chi^{k_0}))$
as multiplication by $\beta (f) \alpha (g)/p^{k_0}.$ In particular, 
$\CDcris (\gr_1\bD_{f,g}(\chi^{k_0}))^{\Ph=1}=\CDcris (\gr_1\bD_{f,g}(\chi^{k_0}))^{\Ph=p^{-1}}=0.$ This implies 1).

2) The $(\Ph,\Gamma)$-module $\gr_2\bD_{f,g}(\chi^{k_0})$ is of 
Hodge--Tate weight $k_0-2\geqslant 0,$ and $\Ph$ acts on  $\CDcris (\gr_2\bD_{f,g}(\chi^{k_0}))$ as multiplication by $\beta (f) \beta (g)/p^{k_0}.$ 
If $\beta (f)\beta (g) \neq p^{k_0-1}$ this implies that 
$\CDcris (\gr_2\bD_{f,g}(\chi^{k_0}))^{\Ph=1}=\CDcris (\gr_2\bD_{f,g}(\chi^{k_0}))^{\Ph=p^{-1}}=0.$ Therefore $H^1_g (\gr_2\bD_{f,g}(\chi^{k_0}))=0,$
and by the same argument as in the proof of Proposition~\ref{proposition local property os semistabilized BF}, we conclude that  $\Zrm_{f^*,g^*}^{[k_0-2]}\in  H^1(\gr_1\bD_{f,g}(\chi^{k_0}))$ in this case. 

In the general case, the proof is slightly different. Consider the diagram (\ref{diagram semistabilized zeta})
for the forms $f^*$ and $g^*$ instead $f$ and $g$ and $j=k_0-2.$ Then 
the canonical isomorphism $W^*_{f^*,g^*_{\alpha}}(2-k_0)\simeq W_{f,g}(k_0)$
identifies the specialization of this diagram at weight $k_0$ with the diagram
(\ref{diagram non stabilized zeta}). From Proposition~\ref{proposition local property os semistabilized BF} we have 
\[
\spec_{k_0}^{\g^*}\left (\,_b\Zfrak_{f^*,\g^*}^{[k_0-2]}\right )\in H^1\left (\gr_1\bD_{f^*,g^*_\alpha}^*(\chi^{2-k_0})\right ).
\]
On the other hand, Proposition~\ref{proposition interpolation of semistabilized classes} and (\ref{relation between BF and _bBF}) imply that  
\begin{equation}
\nonumber
\left (\id, {\Pr}^{\alpha}_* \right )\circ \spec^{\g^*}_{k_0} \left (\,_b\Zfrak_{f^*,\g^*}^{[k_0-2]} \right)=
\left (1-\frac{p^{k_0-1}}{\beta (f) \alpha (g) } \right )
\left (1-\frac{p^{k_0-1}}{\alpha (f) \alpha (g)} \right )
\,_b\Zrm_{f^*,g^*}^{[k_0-2]},
\end{equation}
where $\,_b\Zrm_{f^*,g^*}^{[k_0-2]}=(b^2- \ep_f^{-1}(b) \ep_g^{-1}(b)) 
\Zrm_{f^*,g^*}^{[k_0-2]}.$
Since the terms in brackets do not vanish,   $\Zrm_{f^*,g^*}^{[k_0-2]}\in 
 H^1(\gr_1\bD_{f,g}(\chi^{k_0})),$ and 2) is proved.
\end{proof}

\subsubsection{} Consider the diagram
\begin{equation}
\nonumber
\xymatrix{
& & H^1_{\Iw}(\bD_{\f,\g}^* ) \ar[d] &   \\
0\ar[r] &H^1_{\Iw} \left (\gr_1\bD_{\f,\g}^*\right ) \ar[r] 
&H^1_{\Iw} \left (\bD_{\f,\g}^*/F_0\bD_{\f,\g}^* \right ) \ar[r] 
&H^1_{\Iw}\left (\gr_2\bD_{\f,\g}^* \right ). 
}
\end{equation}
Let 
\[
\res_p \left (\,_b\BF_{\f,\g}^{\Iw}\right)\in H^1_{\Iw}(\Q_p,W_{\f,\g}^*)\otimes_{\Lambda}\CH_E(\Gamma) \simeq H^1_{\Iw} \left (\bD_{\f,\g}^* \right )
\] 
denote the localization of the  Beilinson--Flach Iwasawa class $\,_b\BF_{\f,\g}^{\Iw}.$

\begin{mydefinition} 
\label{definition of Zf,g in family}
We denote by $\,_b\Z^{\Iw}_{\f,\g}$ the image of 
$\res_p \left (\,_b\BF_{\f,\g}^{\Iw}\right)$ under the map
$H^1_{\Iw} (\bD_{\f,\g}^* ) \rightarrow  H^1_{\Iw}  \left (\bD_{\f,\g}^*/F_0 \bD_{\f,\g}^* \right ).$
\end{mydefinition}

We have the following analog of Proposition~\ref{proposition local property os semistabilized BF}.

\begin{myproposition} 
\label{proposition local property of  two-variable BF}
The image of   $\,_b\Z^{\Iw}_{\f,\g}$ under the map
$H^1_{\Iw} \left (\bD_{\f,\g}^*/F_0 \bD_{\f,\g}^* \right ) \rightarrow  H^1_{\Iw}\left (\gr_2\bD_{\f,\g}^*\right )$
is zero, and therefore 
\[
\,_b\Z^{\Iw}_{\f,\g}\in H^1_{\Iw} \left (\gr_1\bD_{\f,\g}^*\right ).
\]
\end{myproposition}
\begin{proof} See \cite[Theorem~7.1.2]{LZ}.
\end{proof}

We record the following corollary of Proposition~\ref{proposition Iwasawa vs semistabilized classes}.

\begin{myproposition} 
\label{proposition Iwasawa vs semistabilized zeta elements}
Let
\[
\,_b\Zfrak^{\Iw}_{f,\g}=({\Pr}^\alpha_*, \id)\circ \spec^{\f}_{k_0}\left (_b\Z^{\Iw}_{\f,\g}\right ) \in H^1_{\Iw}(\gr_1\bD^*_{f,\g}).
\]
Then for any integer $0\leqslant j\leqslant \min \{k,l\}$ there exists
a neighborhood $U_g=\Spm (A_{\g})$ such that 
\[
\spec^c_{-j}
\left (\,_b\Zfrak^{\Iw}_{f,\g} \right )=
\frac{(-1)^j}{j!} \cdot \binom{k}{j} \cdot \left (
1-\frac{p^j}{\alpha (f)\cdot \bb_p}
\right )\cdot 
\left (1-\frac{\beta (f)\cdot \bb_p}{p^{j+1}}\right ) \cdot
\,_b\Zfrak_{f,\g}^{[j]}. 
\]
\end{myproposition}

\section{ $p$-adic $L$-functions}
\label{section p-adic L-functions}.
\subsection{Three-variable $p$-adic $L$-function}

\subsubsection{} We maintain notation and assumptions of Section~\ref{section triangulations}. In particular, 
we assume that the forms  $f$ and $g$ satisfy conditions {\bf M1-2)} of Section~\ref{subsection Stabilized Beilinson--Flach families}. We will also assume that 
$\ep_f\ep_g\neq 1.$
Let $\bM_{\f,\g}$ 
be the $(\Ph,\Gamma)$-modules defined by (\ref{definition of M_f,g}).
Recall that $\bM_{\f,\g}^*(\chi\bchi_{\g})=\gr_1\bD_{\f,\g}^*$ by
(\ref{definition of M_f,g*(chi)}).  We have   pairings on Iwasawa cohomology 
\begin{equation}
\nonumber
\begin{aligned}
\nonumber
&\left \{ \,\cdot \,,\, \cdot \,\right \}_{\bM_{\f,\g}} \,:\,H^1_{\Iw}(\bM_{\f,\g}^*(\chi))
\times
H^1_{\Iw}(\bM_{\f,\g})^{\iota}
\rightarrow \CH_A(\Gamma),\\
\nonumber
&\left \{ \,\cdot \,,\, \cdot \,\right \}_{\bM_{\f,\g}(\bchi_{\g}^{-1})} \,:\,H^1_{\Iw}(\bM_{\f,\g}^*(\chi \bchi_{\g}))
\times
H^1_{\Iw}(\bM_{\f,\g}(\bchi_{\g}^{-1}))^{\iota}
\rightarrow \CH_A(\Gamma).
\end{aligned}
\end{equation}
Let $m_{\f,\g}$ denote the canonical generator of $\bM_{\f,\g}$ defined by
(\ref{definition of m bold f bold g}). Set 
\begin{equation}
\nonumber
\widetilde m_{\f,\g}=m_{\f,\g}\otimes (1+X).
\end{equation}
Recall that for an unspecified large exponential map $\Exp$ we 
set  $\Exp^c=c \circ \Exp,$ where $c\in \Gamma$ is  the unique element of order $2$ (see Section~\ref{subsubsection large logarithm}). 
For any $\z \in H^1_{\Iw}(\bM_{\f,\g}^*(\chi \bchi_{\g}))$ define
\begin{equation}
\nonumber
\mathfrak L_p (\z, \omega^a, x, y,s)=
{\mathcal A}_{\omega^a} \left (
\left \{\z, \Tw_{\bchi_{\g}^{-1}}\circ \Exp^c_{\bM_{\f,\g}}(\widetilde m_{\f,\g})^{\iota} \right \}_{\bM_{\f,\g}(\bchi_{\g}^{-1})} 
\right )(x,y,s),
\end{equation}
where   the transform ${\mathcal A}_{\omega^a}$ was defined in (\ref{transform A}),
and
\[
x=k_0+\log (1+w_1)/\log (1+p), \qquad y=l_0+\log (1+w_2)/\log (1+p).
\]

\begin{mylemma}
\label{lemma first improved L}
We have
\begin{equation}
\nonumber
\mathfrak L_p (\z, \omega^a, x, y,s)=
{\mathcal A}_{\omega^{a-l_0+1}}\left ( \mathfrak{Log}_{\bM_{\f,\g}, m_{\f,\g}} \left (\Tw_{\bchi_{\g}^{-1}}(\z)\right )\right ) (x,y,s-y+1),
\end{equation}
where $\mathfrak{Log}$ is  the map defined in Section~\ref{subsubsection large logarithm}.
\end{mylemma}
\begin{proof}
 By Lemma~\ref{Lemma twisted Iwasawa pairing}, we have
\begin{equation}
\nonumber
\mathfrak L_p(\z,\omega^{a}, x, y,s)
={\mathcal A}_{\omega^{a}}
\left ( \Tw_{\bchi_{\g}^{-1}}\circ \left \{\Tw_{\bchi_{\g}^{-1}}(\z), \Exp^c_{\bM_{\f,\g}}(\widetilde m_{\f,\g})^{\iota} \right \}_{\bM_{\f,\g}} \right ).
\end{equation}
Writing $\Tw_{\bchi_{\g}}=\Tw_{\bchi}\circ \Tw_{l_0-1}$ and taking into 
account (\ref{definition of Tw_m}) and
 (\ref{property of Tw_bchi})
we get 
\begin{multline} 
\nonumber
{\mathcal A}_{\omega^{a}}
\left ( \Tw_{\bchi_{\g}^{-1}}\circ \left \{\Tw_{\bchi_{\g}^{-1}}(\z), \Exp^c_{\bM_{\f,\g}}(\widetilde m_{\f,\g})^{\iota} \right \}_{\bM_{\f,\g}} \right )
(x,y,s)
=\\
={\mathcal A}_{\omega^{a-l_0+1}}
\left ( \Tw_{\bchi^{-1}}\circ \left \{ \Tw_{\bchi_{\g}^{-1}}(\z) ,  \Exp^c_{\bM_{\f,\g}}(\widetilde m_{\f,\g})^{\iota} \right \}_{\bM_{\f,\g}} \right )
(x,y,s-l_0+1)=\\
={\mathcal A}_{\omega^{a-l_0+1}}
\left (  \left \{ \Tw_{\bchi_{\g}^{-1}}(\z) ,  \Exp^c_{\bM_{\f,\g}}(\widetilde m_{\f,\g})^{\iota} \right \}_{\bM_{\f,\g}}\right )
(x,y,s-y+1),
\end{multline}
and the lemma is proved.
\end{proof}

\subsubsection{}
Let $L_p(\f, \g, \omega^a ) (x,y ,s)$ denote 
the three-variable $p$-adic $L$-function. 
Recall the element  
\[
\,_b\Z^{\Iw}_{\f,\g}\in H^1\left (\gr_1\bD^*_{\f,\g}\right )
\]
defined in Section~\ref{subsection Local properties of Beilinson--Flach elements}.

\begin{mytheorem}[\sc Kings--Loeffler--Zerbes] 
\label{theorem relation L with Beilinson-Flach element}
One has
\begin{equation}
\nonumber
B_b(\omega^a, x,y,s) L_p(\f, \g, \omega^a ) (x,y ,s)=(-1)^a \mathfrak L_p(\,\,_b\Z^{\Iw}_{\f,\g},\omega^{a-1}, x, y,s-1),
\end{equation}
where 
\begin{equation}
\label{formula for factor B}
B_b(\omega^a,x,y,s)=\frac{G(\ep_f^{-1})G(\ep_g^{-1})}{\lambda_{N_f} (\f)(x)} 
\left (b^2-\omega (b)^{2a-k_0-l_0+2}
\left < b\right >^{2s-x-y+2}\ep_f^{-1}(b)\ep_g^{-1}(b)\right )
 \end{equation}
and $G(-)$ denotes the corresponding Gauss sum.
\end{mytheorem}
\begin{proof} The theorem was first proved in the ordinary case in \cite[Theorem~10.2.2]{KLZ}. The non-ordinary case is treated 
in \cite[Theorem~7.1.5]{LZ}.
\end{proof}


\subsection{The first improved $p$-adic $L$-function}
\label{subsection The first improved $p$-adic $L$-function}
In  this section we assume that $k_0=l_0.$
Recall that $\bM_{\f,\g}^*(\chi)=F^{-+}\bD^*_{\f,\g}(\bchi_{\g}^{-1})$
(see (\ref{definition of M_f,g*(chi)})).
Let  $\Tw_{\bchi_{\g}^{-1}}\left (\,_b\Z^{\Iw}_{\f,\g}\right )$  denote the image of 
$\,_b\Z^{\Iw}_{\f,\g}$ in $ H^1_{\Iw} \left (\bM_{\f,\g}^*(\chi ) \right )$ under the canonical map 
$m\mapsto m\otimes {\bchi_{\g}^{-1}}.$  

\begin{mydefinition}
\label{definition of Zfg bold}
We denote by 
\begin{equation}
\nonumber
\,_b{\Z}_{\f,\g}\in H^1(\bM_{\f,\g}^*(\chi)).
\end{equation}
the image of  $\Tw_{\bchi_{\g}^{-1}}\left (\,_b\Z^{\Iw}_{\f,\g}\right )$ under the canonical
projection  $H^1_{\Iw}(\bM_{\f,\g}^*(\chi))\rightarrow  H^1 (\bM_{\f,\g}^*(\chi)).$
\end{mydefinition}
We have
\[
B_b(\omega^{k_0},k_0,s,s)= 
\frac{G(\ep_f^{-1})G(\ep_g^{-1})}{\lambda_{N_f} (\f)(x)} 
\left (b^2-\omega (b)^{2}
\left < b\right >^{s-k_0+2}\ep_f^{-1}(b)\ep_g^{-1}(b)\right ).
\]
Since $\ep_f\ep_g\neq 1,$ we can and will choose $b$ such that 
$B_b(\omega^{k_0},k_0,k_0,k_0)\neq 0.$

Recall that $\bM_{\f,\g}$ is a crystalline module of Hodge--Tate
weight $1.$  We denote by  
\[
\exp_{\bM_{\f,\g}}\,:\,\CDcris (\bM_{\f,\g})\rightarrow H^1(\bM_{\f,\g}).
\]
the Bloch--Kato exponential map for $\bM_{\f,\g}.$ 

\begin{mydefinition} 
\label{definition of the first improved L-function}
We define the first improved $p$-adic $L$-function
as the analytic function given by  
\begin{equation}
\nonumber
L_p^{\mathrm{wc}}(\f,\g,s)=B_b(\omega^{k_0},k_0,s,s)^{-1}\cdot{\mathcal A}^{\mathrm{wt}}\left (\left  <\,_b{\Z}_{\f,\g} , \exp_{\bM_{\f,\g}}(m_{\f,\g})
\right >
_{\bM_{\f,\g}}
\right ) (k_0, s),
\end{equation}
where $\left < \,\,,\,\,\right >_{\bM_{\f,\g}}
\,:\,H^1(\bM_{\f,\g}^*(\chi))\times H^1(\bM_{\f,\g}) \rightarrow A$ is the local duality. 
\end{mydefinition}

\begin{myproposition} 
\label{proposition first improved L-function}
Assume that $k_0=l_0.$ Then in a sufficiently small neighborhood of $k_0,$
one has
\begin{equation}
\nonumber
 L_p(\f, \g, \omega^{k_0} ) (k_0,s,s)=
(-1)^{k_0}\left (1-\frac{\bb_{p}(s)}{\ep_{g}(p)\ba_{p}(k_0)}\right )
 \left (1-\frac{\ep_g(p)\ba_{p}(k_0)}{p\bb_{p}(s)}\right )^{-1}
 L_p^{\mathrm{wc}}(\f,\g,s).
\end{equation}
In particular, $L_p^{\mathrm{wc}}(\f,\g,s)$ does not depend on the choice of $b.$
 \end{myproposition}
\begin{proof} By Lemma~\ref{lemma first improved L}, we have
\begin{equation} 
\label{proposition first improved L:formula1}
\mathfrak L_p(\,_b\Z^{\Iw}_{\f,\g},\omega^{k_0-1}, k_0, s,s-1)
={\mathcal A}_{\omega^{0}}
\left ( \left \{ \Tw_{\bchi_{\g}^{-1}} \left (\,_b\Z^{\Iw}_{\f,\g}\right ),  \Exp^c_{\bM_{\f,\g}}(\widetilde m_{\f,\g})^{\iota} \right \}_{\bM_{\f,\g}} \right )
(k_0,s,0).
\end{equation}
By (\ref{formulas for phi action on M}), the action of $\Ph$ on 
$\bM_{\f,\g}$ is given by  $\Ph (m_{\f,\g})=\left (p^{-1}\ep_g(p)\ba_p \bb_p^{-1}\right ) m_{\f,\g}.$  Applying the first formula of  Corollary~\ref{corollary large exponential},
we obtain
\begin{multline}
\label{proposition first improved L:formula2}
{\mathcal A}_{\omega^{0}}\left ( \left \{\Tw_{\bchi_{\g}^{-1}}\left (\,_b\Z^{\Iw}_{\f,\g}\right ),  \Exp^c_{\bM_{\f,\g}}(\widetilde m_{\f,\g})^{\iota} \right \}_{\bM_{\f,\g}}\right )(k_0,s,0)=\\
\left (1-\frac{\bb_{p}(s)}{\ep_{g}(p)\ba_{p}(k_0)}\right )
 \left (1-\frac{\ep_g(p)\ba_{p}(k_0)}{p\bb_{p}(s)}\right )^{-1}
{\mathcal A}^{\mathrm{wt}}\left (\left  <\,_b{\Z}_{\f,\g} , \exp_{\bM_{\f,\g}}(m_{\f,\g})
\right >
_{\bM_{\f,\g}} 
\right ) (k_0, s).
\end{multline}
The proposition follows from formulas (\ref{proposition first improved L:formula1}),
(\ref{proposition first improved L:formula2}) and Theorem~\ref{theorem relation L with Beilinson-Flach element}.
\end{proof}

\subsection{The second improved $p$-adic $L$-function}
\label{subsection The second improved $p$-adic $L$-function}
\subsubsection{} We continue to assume that $k_0=l_0.$ Let 
\[
\bM_{f,\g}=\gr_0\bD_{f,\g}(\chi\bchi_{\g}),
\qquad 
\bN_{f,\g}= \gr_0\bD_{f,\g}(\chi^{k_0-1})=\bM_{f,\g}(\chi^{-1}\bchi^{-1}).
\]
Define
\begin{equation}
\label{definition of the semistabilized generator m}
\mm_{f,\g}=\frac{1}{C(f)\cdot \lambda_{N_f}(f)}\eta_f\otimes \bxi_{\g}\otimes e_1
\in \CDcris (\bM_{f,\g}).
\end{equation}
Let 
$\left (\Pr_{\alpha}^*,\id \right )\,:\,\CDcris (\bM_{f,\g})\rightarrow
\CDcris (\bM_{f_{\alpha},\g})$ denote the map 
induced by the map (\ref{dual isomorphism of representations  stabilized and
non stabilized }). By Proposition~\ref{proposition interpolation eigenvectors},
\begin{equation}
\label{image of m under Pr}
\spec^{\f}_{k_0} (m_{\f,\g})=\left ({\Pr}_{\alpha}^*,\id \right ) (\mm_{f,\g}).
\end{equation}
Consider the composition
\begin{equation}
\nonumber
H^1_{\Iw}(\bM_{f,\g}) \xrightarrow{\Tw_{\bchi_{\g}^{-1}}}
H^1_{\Iw}(\bM_{f,\g}(\bchi_{\g}^{-1}))\xrightarrow{\spec^c_{k}}
H^1(\bN_{f,\g}),
\end{equation}
where $k=k_0-2$ and  $\spec^c_{k}$ is the map (\ref{specialization of Iwasawa cohomology}). We have
\[
B_b(\omega^{k_0-1},k_0,s,k_0-1)=\frac{G(\ep_f^{-1})G(\ep_g^{-1})}{\lambda_{N_f} (\f)(k_0)} 
\left (b^2-
\left < b\right >^{k_0-s}\ep_f^{-1}(b)\ep_g^{-1}(b)\right ).
\]
Therefore  $B_b(\omega^{k_0-1},k_0,k_0,k_0-1)\neq 0$  for any $b\neq 1.$
\begin{mydefinition}
\label{definition second improved L-function}
We define the second improved $p$-adic $L$-function as
the analytic function given by
\begin{multline}
\nonumber
L_p^{\mathrm{wt}}(\f,\g,s)= \Gamma (k_0-1)^{-1}\cdot
B_b(\omega^{k_0-1},k_0,s,k_0-1)^{-1}\times\\ 
\times \mathcal A^{\mathrm{wt}}\left (
\left <\,_b\Zfrak_{f,\g}^{[k_0-2]}, 
\spec^c_{k_0-2}\circ \Tw_{\bchi^{-1}_{\g}}\circ \Exp^c_{\bM_{f,\g}}
(\widetilde \mm_{f,\g})^{\iota}
\right >_{\bN_{f,\g}}
\right ) (s),
\end{multline}
where  $\left <\,\,,\,\,\right >_{\bN_{f,\g}}$  is the local duality pairing
and $\,_b\Zfrak_{f,\g}^{[k]}$ is the element constructed in Section~\ref{subsection Local properties of Beilinson--Flach elements}.
\end{mydefinition}

\begin{myproposition} 
\label{proposition second improved L-function}
Assume that $k_0=l_0.$ Then on a  sufficiently small  neighborhood 
$U_g=\Spm (A_{\g})$ of $k_0$ one has 
\begin{equation}
\nonumber
L_p(\f, \g, \omega^{k_0-1} ) (k_0,s,k_0-1)
=-\left (1-\frac{p^{k_0-2}}{\ba_p(k_0)\bb_p(s)}\right )
\left (1-\frac{\ep_f(p)\bb_p(s)}{\ba_p(k_0)}\right )
L_p^{\mathrm{wt}}(\f,\g,s).
 \end{equation}
\end{myproposition}

\begin{proof}  
Let $\bN_{\f,\g}= \gr_0\bD_{\f,\g}(\chi^{k_0-1}).$
By Theorem~\ref{theorem relation L with Beilinson-Flach element} and 
Lemma~\ref{Lemma twisted Iwasawa pairing} one has
\begin{multline}
\label{second improved function first formula}
 G(f,g) B_b(k_0,s, k_0-1) \cdot L_p(\f, \g, \omega^{k_0-1} ) (k_0,s,k_0-1)
=\\
(-1)^{k_0-1}\mathcal A_{\omega^{k_0-2}}\left (\left \{\,_b \Z^{\Iw}_{\f,\g},
\Tw_{\bchi_{\g}^{-1}} \circ \Exp^c_{\bM_{\f,\g}} (\widetilde m_{\f,\g})^{\iota} \right \}_{\bM_{\f,\g}(\bchi_{\g}^{-1})} \right ) (k_0,s,k_0-2)=\\
(-1)^{k_0-1}\mathcal A_{\omega^{0}}\left (\left \{\Tw_{2-k_0}\left (\,_b \Z^{\Iw}_{\f,\g}\right ),
\Tw_{k_0-2}\circ \Tw_{\bchi_{\g}^{-1}} \circ \Exp^c_{\bM_{\f,\g}} (\widetilde m_{\f,\g})^{\iota} \right \}_{\bN_{\f,\g}} \right )(k_0,s,0)=\\
(-1)^{k_0-1}\mathcal A^{\mathrm{wt}}\left (\left <\spec^c_{2-k_0}\left (\,_b \Z^{\Iw}_{\f,\g}\right ),
\spec^c_{k_0-2}\circ \Tw_{\bchi_{\g}^{-1}} \circ \Exp^c_{\bM_{\f,\g}} (\widetilde m_{\f,\g})^{\iota} \right >_{\bN_{\f,\g}}\right )(k_0,s).
\end{multline}
Let $\,_b\Zfrak_{f,\g}^{\Iw}=({\Pr}^{\alpha}_*,\id)\circ \spec^{\f}_{k_0}
\left (\,_b \Z^{\Iw}_{\f,\g}\right ).$ Since the maps ${\Pr}^{\alpha}_*$
and ${\Pr}_{\alpha}^*$ are dual to each other, from (\ref{image of m under Pr})
we obtain that
\begin{multline}
\label{second improved function second formula}
\mathcal A^{\mathrm{wt}}\left (\left <\spec^c_{2-k_0}\left (\,_b \Z^{\Iw}_{\f,\g}\right ),
\spec^c_{k_0-2}\circ \Tw_{\bchi_{\g}^{-1}} \circ \Exp^c_{\bM_{\f,\g}} (\widetilde m_{\f,\g})^{\iota} \right >_{\bN_{\f,\g}}\right )(k_0,s)=\\
\mathcal A^{\mathrm{wt}}\left (\left <\spec^c_{2-k_0}\left (\,_b \Zfrak^{\Iw}_{f,\g}\right ),
\spec^c_{k_0-2}\circ \Tw_{\bchi_{\g}^{-1}} \circ \Exp^c_{\bM_{f,\g}} (\widetilde 
 \mm_{f,\g})^{\iota} \right >_{\bN_{f,\g}}\right )(s).
\end{multline}
By Proposition~\ref{proposition Iwasawa vs semistabilized zeta elements},
\begin{equation}
\nonumber
\spec^c_{2-k_0}\left (\,_b \Zfrak^{\Iw}_{f,\g}\right )=
\frac{(-1)^{k_0}}{(k_0-2)!}  \cdot \left (
1-\frac{p^{k_0-2}}{\alpha (f)\cdot \bb_p}
\right )\cdot 
\left (1-\frac{\beta (f)\cdot \bb_p}{p^{k_0-1}}\right ) \cdot
\,_b\Zfrak_{f,\g}^{[k_0-2]}. 
\end{equation}
Combining this formula with (\ref{second improved function first formula}-\ref{second improved function second formula}) and taking into account that $\alpha (f)\beta (f)
=\ep_f(p)p^{k_0-1},$ we obtain the wanted formula.
\end{proof}

\subsection{The functional equation} 
\subsubsection{}
In this subsection, we establish a  functional equation for our improved $p$-adic  $L$-functions.   In addition to {\bf M1-2)}, we assume that the following conditions hold:

\begin{itemize}

\item[]{\bf M3)} The characters  $\ep_f,$ $\ep_g$ and $\ep_f\ep_g$ are primitive modulo $N_f,$ $N_g$ and $\mathrm{lcm}(N_f,N_g)$ respectively.

\end{itemize}
We remark that {\bf M3)} implies that $\ep_f\ep_g \neq 1.$
In particular, the case  $f\neq g^*$ is excluded.

Write $\lambda_{N_f}(f)= p^{k_0/2}w(f),$ $\lambda_{N_g}(g)= p^{l_0/2}w(g),$
$w(\ep_f\ep_g)=G(\ep_f\ep_g)N^{-1/2}$ and set 
\begin{equation}
\nonumber 
w(f,g)= (-1)^{l_0} \cdot w(f)\cdot  w(g)\cdot   w(\ep_f\ep_g) \cdot \frac{\ba_{d_g}^c(k_0)\cdot \bb_{d_f}^c(y)}
{d_g^{(k_0-1)/2}\cdot d_f^{(l_0-1)/2}},
\end{equation}
where $c$ denotes the complex multiplication. The complete Rankin--Selberg $L$-function $\Lambda (f,g,s) =\Gamma (s) \Gamma(s-l_0+1) (2\pi)^{l_0-1-2s}L(f,g,s)$
has a holomorphic continuation to all $\mathbf C$ and satisfies the functional equation 
\begin{equation}
\nonumber
\Lambda (f,g,s)=\ep (f,g,s)\cdot \Lambda (f^*,g^*, k_0+l_0-1-s)
\end{equation}
where 
$\ep (f,g,s) = w(f,g) \cdot \left (NN_fN_g\right )^{\frac{k_0+l_0-1}{2}-s}$ (see, for example, \cite[Section~9.5]{Hi93}).

Denote by $f^*_\alpha$ and $g^*_\alpha$ the $p$-stabilizations of $f^*$ and
$g^*$ with respect to the roots $\alpha (f^*)=p^{k_0-1}/\beta (f)$ and 
$\alpha (g^*)=p^{k_0-1}/\beta (g)$ respectively and by $\f^*$ and $\g^*$ 
the Coleman families passing through  $f^*_\alpha$ and $g^*_\alpha .$ 

\begin{myproposition}
The three variable $p$-adic function $L_p(\f,\g,\omega^a)(x,y,s)$ satisfies the  functional equation
\[
L_p(\f,\g,\omega^a)(x,y,s)=\ep_p^{[\f,\g,a]}(x,y,s)\cdot L_p(\f^*,\g^*,\omega^{a^*})(x,y,x+y-s-1),
\]
where $a^*=k_0+l_0-a-1$ and 
\begin{equation}
\nonumber
\ep_p^{[\f,\g,a]}(x,y,s)=w(f,g)\cdot \left (NN_fN_g\right )^{\frac{k_0+l_0-1}{2}-a} \cdot \left <NN_fN_g\right >^{a-s}\cdot 
\left <N_f\right >^{\frac{y-l_0}{2}} \cdot
\left <N_g\right >^{\frac{x-k_0}{2}}.
\end{equation}
\end{myproposition}
\begin{proof} This proposition follows from the interpolation properties of  
$L_p(\f,\g,\omega^a)(x,y,s)$ and the functional equation for the complex 
Rankin--Selberg $L$-function.
We leave the details to the reader.
\end{proof}
\begin{mycorollary}  
\label{functional equation for improved functions}
The improved $L$-functions are related by the functional equation
\begin{equation}
\nonumber
L_p^{\mathrm{wc}}(\f,\g,s)=A_p^{[\f,\g]} (s) \cdot
\left (1-\frac{\ep_f(p) \ep_g(p)p^{k_0-2}}{\ba_p(k_0)\bb_p(s)}\right )
\cdot
 \left (1-\frac{\ep_g(p)\ba_{p}(k_0)}{p\bb_{p}(s)}\right )
\cdot  L_p^{\mathrm{wt}}(\f^*,\g^*,s),
\end{equation}
where 
\[
A_p^{[\f,\g]} (s)=  (-1)^{k_0-1}w(f,g) \cdot (NN_fN_g)^{-1/2}\cdot \left <NN_fN_g\right >^{k_0-s}\cdot 
\left <N_f\right >^{\frac{s-k_0}{2}}.
\]
\end{mycorollary}
\begin{proof}
Set $\f^*=\underset{n=1}{\overset{\infty}\sum} \ba^*_n q^n$ and $\g^*=\underset{n=1} {\overset{\infty}\sum}\bb^*_n q^n.$
The functional equation gives
\[
L_p(\f,\g,\omega^{k_0})(k_0,s,s)=\ep_p^{[\f,\g,k_0]}(k_0,s,s) \cdot L_p(\f^*,\g^*,\omega^{k_0-1})(k_0,s,k_0-1).
\]
Applying Propositions~\ref{proposition first improved L-function} and \ref{proposition second improved L-function} and taking into account 
that $\ba^*=\ep_f^{-1}(p)\ba$ and $\bb^*=\ep_g^{-1}(p)\bb,$ we get
\begin{multline}
\nonumber
\left (1-\frac{\bb_{p}(s)}{\ep_{g}(p)\ba_{p}(k_0)}\right )
 \left (1-\frac{\ep_g(p)\ba_{p}(k_0)}{p\bb_{p}(s)}\right )^{-1}
 L_p^{\mathrm{wc}}(\f,\g,s)=\\
=A_p^{[\f,\g]}(s)
\left (1-\frac{\ep_f(p) \ep_g(p)p^{k_0-2}}{\ba_p(k_0)\bb_p(s)}\right )
\left (1-\frac{\bb_p(s)}{\ep_g(p)\ba_p(k_0)}\right )
L_p^{\mathrm{wt}}(\f^*,\g^*,s).
\end{multline}
Since the function $\displaystyle \left (1-\frac{\bb_p(s)}{\ep_g(p)\ba_p(k_0)}\right )$ is not identically zero, we can cancel it in this equation. This gives us
the wanted formula. 
\end{proof}

\begin{myremark} One has $A_p^{[\f,\g]}(k_0)=(-1)^{k_0-1}\ep (f,g,k_0).$
\end{myremark}

\subsection{Functional equation for zeta elements}
\label{subsection zeta elements}

\subsubsection{} In this section, we interpret the functional equation
for $p$-adic $L$-functions  in terms  of Beilinson--Flach elements
and prove Theorem~II. 
We assume that $f$ and $g$ are newforms of the same weight $k_0\geqslant 2$
which satisfy conditions {\bf M1-3)}. Set $V_g=W_g(k_0)$ and $V_{f,g}=W_{f,g}(k_0).$
We consider the canonical basis of $\Dc (V_{f,g})$ formed  by the eigenvectors
\begin{equation}
\begin{aligned}
\label{canonical basis of Dc (V)}
&d_{\alpha\alpha}=\eta_f^{\alpha}\otimes \eta_g^{\alpha}\otimes e_{k_0},
&&d_{\alpha\beta}=\eta_f^{\alpha}\otimes \omega_g^{\beta}\otimes e_{k_0},\\
&d_{\beta\alpha}=\omega_f^{\beta}\otimes \eta_g^{\alpha}\otimes e_{k_0},
&&d_{\beta\beta}=\omega_f^{\beta}\otimes \omega_g^{\beta}\otimes e_{k_0}.
\end{aligned}
\end{equation}
Let $D=\eta_f^{\alpha}\otimes \Dc(V_g).$  We associate to $D$ the filtration
 $(D_i)_{i=-2}^2$ on $\Dc (V_{f,g})$ 
defined by 
\begin{equation}
\label{filtration on Dcris}
D_i=\begin{cases}
0, &\text{if $i=-2$},\\
E\cdot d_{\alpha\alpha}, &\text{if $i=-1$},\\
E\cdot d_{\alpha\alpha}+E\cdot d_{\alpha\beta}, &\text{if $i=0$},\\
E\cdot d_{\alpha\alpha}+E\cdot d_{\alpha\beta}+E\cdot d_{\beta\alpha}
, &\text{if $i=1$},\\
\Dc (V_{f,g}), &\text{if $i=2$}.
\end{cases}
\end{equation}
Note that $D_0=D.$ This filtration defines a unique triangulation 
$\left (F_i\DdagrigE (V_{f,g})\right )_{i=-2}^2$ of $\Ddagrig (V_{f,g})$
such that $D_i=\CDcris (F_i\DdagrigE (V_{f,g}))$ for all $-2\leqslant i\leqslant 2.$
From definition it follows that the isomorphism 
\[
({\Pr}_\alpha^*,{\Pr}_\alpha^*)\,\,:\,\, W_{f,g} \simeq \spec^{\f,\g}_{k_0,k_0}
\left ( W_{\f,\g}\right ),
\]
identifies $\left (F_i\DdagrigE (V_{f,g})\right )_{i=-2}^2$ with the specialization
at $(f_\alpha, g_\alpha)$ of the triangulation 
\linebreak
$\left (F_i\bD_{\f,\g}(\chi^{k_0})\right )_{i=-2}^2$ constructed in Section~\ref{subsection triangulations}.

To simplify notation, set 
\[
\bM_{f,g}=\gr_0\DdagrigE (V_{f,g}), \qquad 
\bN_{f^*,g^*}=\gr_0\Ddagrig (V_{f,g}^*(1)).
\]
Since $\spec^{\f,\g}_{k_0,k_0}(\bM_{\f,\g})=\gr_0\bD_{f_\alpha,g_\alpha}(\chi^{k_0})$ 
and 
\[
\spec^{\f,\g}_{k_0,k_0}(\bN_{\f^*,\g^*})=\bM_{f^*_\alpha, g^*_\alpha}(\chi^{-1})
=\gr_0\bD_{f^*_\alpha, g^*_\alpha}(\chi^{k_0-1}),
\]
we have canonical isomorphisms
\begin{equation}
\nonumber
\bM_{f,g}\simeq \spec^{\f,\g}_{k_0,k_0}(\bM_{\f,\g}), \qquad
\bN_{f^*,g^*}\simeq \spec^{\f^*,\g^*}_{k_0,k_0}(\bN_{f^*,\g^*}).
\end{equation}
Thus our  notation agrees with that of 
 Sections~\ref{subsection triangulations} and \ref{subsection The second improved $p$-adic $L$-function}.

By (\ref{definition of M_f,g*(chi)}), we have
\begin{equation}
\label{isomorphism forthe dual  Mfg}
\bM_{f,g}^*(\chi) \simeq \gr_1 \DdagrigE (V_{f,g}^*(1)), \qquad
V_{f,g}^*(1)\simeq W_{f^*,g^*}(k_0-1).
\end{equation}
Fix canonical generators
\begin{equation}
\begin{aligned}
\nonumber
&d_{\alpha \beta}\in \CDcris (\bM_{f,g}),\\
&n_{\alpha\beta}=\eta_{f^*}^\alpha\otimes \omega_{g^*}^\beta \otimes e_{k_0-1}\in \CDcris (\bN_{f^*,g^*}).
\end{aligned}
\end{equation}
Note that by Proposition~\ref{proposition interpolation eigenvectors}
\begin{equation}
\label{comparision generators of Mfg}
({\Pr}_\alpha^*,{\Pr}_\alpha^*) (d_{\alpha \beta})=
\lambda_{N_f}(f)\cdot C(f) \cdot \spec^{\f,\g}_{k_0,k_0}(m_{\f,\g}).
\end{equation}
We denote by 
\begin{equation}
\nonumber
\exp \,:\,\CDcris (\bM_{f,g}) \rightarrow H^1(\bM_{f,g}) 
\end{equation}
and 
\begin{equation}
\nonumber
\log \,:\, H^1\left (\gr_1\DdagrigE (V_{f,g})\right )
\rightarrow \CDcris \left (\gr_1\DdagrigE (V_{f,g})\right )
\end{equation}
the Bloch--Kato exponential and logarithm maps respectively. 

\subsubsection{} Let 
\begin{equation}
\label{definition of Zrm independent on b} 
\Zrm_{f^*,g^*}^{[k_0-2]}\in H^1 \left (\gr_1\DdagrigE (V_{f,g})\right )
\end{equation}
denote the element constructed in Definition~\ref{definition of Zrm}.
Choose $b$ such that $\ep_f(b)\ep_g(b)\neq 1$ and set

\begin{equation}
\begin{aligned}
\label{definition of specialization of Iwasawa class}
&\,_b\widetilde\Zrm_{f_\alpha,g_\alpha}^{[k_0-1]}=\spec^{\f,\g,c}_{k_0,k_0,1-k_0}\left (\,_b\Z^{\Iw}_{\f,\g}\right ) \in H^1 (\bM^*_{f_\alpha,g_\alpha}(\chi)),\\
&\,_b\widetilde\Zrm_{f,g}^{[k_0-1]}=({\Pr}^\alpha_*,{\Pr}^\alpha_*)
\left ( \,_b\widetilde\Zrm_{f_\alpha,g_\alpha}^{[k_0-1]}\right )\in
H^1 \left (\gr_1\DdagrigE (V_{f,g}^*(1))\right ).
\end{aligned}
\end{equation}
(here we use the isomorphism (\ref{isomorphism forthe dual  Mfg})).
Note that $\,_b\widetilde\Zrm_{f_\alpha,g_\alpha}^{[k_0-1]}=\spec^{\f,\g}_{k_0,k_0}(\,_b\Z_{\f,\g}),$ where $_b\Z_{\f,\g}$ is the element introduced in Definition~\ref{definition of Zfg bold}.  Set
\[
\widetilde\Zrm_{f,g}^{[k_0-1]}=b^{-2}(1-\ep_f(b)\ep_g(b))^{-1}
\,_b\widetilde\Zrm_{f,g}^{[k_0-1]}.
\]
We remark that $\Zrm_{f^*,g^*}^{[k_0-2]}$  is constructed directly from the Beilinson--Flach element $\BFrm_{f^*,g^*}^{[k_0-2]}$ whereas the construction of $\widetilde\Zrm_{f,g}^{[k_0-1]}$ relies on Proposition~\ref{proposition specialization of two variable Beilinson Flach elements} and  involves  $p$-adic interpolation and  Iwasawa twist.

\begin{mytheorem}
\label{theorem functional equation for zeta elements}
Assume that $\beta (f)\alpha (g)\neq p^{k_0-1}.$ Then 
the elements $\Zrm_{f^*,g^*}^{[k_0-2]}$ and
$\widetilde\Zrm_{f,g}^{[k_0-1]}$ are related by the equation
\begin{equation}
\nonumber
\frac{\left <\widetilde\Zrm_{f,g}^{[k_0-1]}, \exp (d_{\alpha\beta}) \right >_{\bM_{f ,g}}}
{
G(\ep_f^{-1}) G(\ep_g^{-1})
}
=(-1)^{k_0-1}\ep (f,g,k_0)\cdot 
\CE (V_{f,g}, D_{-1}) \cdot
\frac{\left [ \log \left (\Zrm_{f^*,g^*}^{[k_0-2]}\right ), n_{\alpha\beta}
\right ]_{\bN_{f^* ,g^*}}}{(k_0-2)! G(\ep_f) G(\ep_g)},
\end{equation}
where 
\[
\CE (V_{f,g}, D_{-1})= \det \left (1-p^{-1}\Ph^{-1} \mid D_{-1} \right )
\det \left (1-\Ph \mid \Dc (V_{f,g})/D_{-1}\right ).
\]
\end{mytheorem}
\begin{proof}
From Definition~\ref{definition of the first improved L-function} we have
\begin{equation}
\nonumber
L_p^{\mathrm{wc}}(\f,\g,k_0)= B_b(\omega^{k_0},k_0,k_0,k_0)^{-1}
\left <\,_b\widetilde\Zrm_{f_\alpha,g_\alpha}^{[k_0-1]}, \exp_{\bM_{f_\alpha,g_\alpha}}(\spec^{\f,\g}_{k_0,k_0}(m_{\f,\g})) \right >_{\bM_{f_\alpha ,g_\alpha}},
\end{equation}
where 
\[
B_b(\omega^{k_0},k_0,k_0,k_0)=\frac{b^2 G(\ep_f^{-1}) G(\ep_g^{-1})}
{\lambda_{N_f}(f)}\cdot (1-\ep_f^{-1}(b)\ep_g^{-1}(b)) \neq 0.
\]
Taking into account (\ref{comparision generators of Mfg}) and (\ref{definition of specialization of Iwasawa class}), we obtain that
\begin{equation}
\label{computation of L(f,g,k)}
L_p^{\mathrm{wc}}(\f,\g,k_0)=\frac{\left <\widetilde\Zrm_{f,g}^{[k_0-1]}, \exp (d_{\alpha\beta}) \right >_{\bM_{f ,g}}}{ C(f) G(\ep_f^{-1}) G(\ep_g^{-1})
}.
\end{equation}
On the other hand,
\begin{equation}
\label{functional equation: formula 1}
B_b(\omega^{k_0-1},k_0,k_0,k_0)=\frac{G(\ep_f) G(\ep_g)}
{\lambda_{N_f}(f^*)}\cdot (b^2-\ep_f(b)\ep_g(b)).
\end{equation}
The Frobenius $\Ph$ acts on $\CDcris (\bM_{f^*,g_\alpha^*})$ as multiplication
by 
\[
\frac{\alpha (f^*)\beta (g^*)}{p^{k_0}}=\frac{p^{k_0-2}}{\beta (f) \alpha (g)}.
\]
By Proposition~\ref{proposition interpolation of semistabilized classes}, i),
one has
\[
\spec^{\g}_{k_0}\left (\,_b\Zfrak_{f^*,\g^*}^{[k_0-2]} \right )=
\,_b\Zrm_{f^*,g_\alpha^*}^{[k_0-2]}.
\]
From Definition~\ref{definition second improved L-function} and Proposition~\ref{them:propertiestwovarPRlog} it follows that
\begin{multline}
\label{functional equation: formula 2}
L_p^{\mathrm{wt}}(\f^*,\g^*,k_0)=\Gamma (k_0-1)^{-1} B_b(\omega^{k_0-1},k_0,k_0,k_0)^{-1}
\left (1-\frac{\beta (f) \alpha (g)}{p^{k_0}} \right )
\times \\
\times
\left (1-\frac{p^{k_0-1}}{\beta (f) \alpha (g)} \right )^{-1}
\cdot
\left [\log \left (\,_b\Zrm_{f^*,g^*_\alpha}^{[k_0-2]}\right ), \spec^{\g}_{k_0}(\frak m_{f^*,\g^*}\otimes e_{-1})) \right ]_{\bN_{f^* ,g^*_\alpha }}.
\end{multline}
Since $C(f^*)=C(f),$ from (\ref{definition of the semistabilized generator m}) 
we have that 
\begin{equation}
\label{functional equation: formula 3}
\displaystyle\spec^{\g}_{k_0}(\frak m_{f^*,\g^*}\otimes e_{-1})=\frac{1}{C(f)\lambda_{N_f}(f^*)}n_{\alpha\beta}.
\end{equation}
 Proposition~\ref{proposition stabilization formulas} and (\ref{relation between BF and _bBF}) give
\begin{equation}
\label{functional equation: formula 4}
(\id, {\Pr}^\alpha_*) \left (\,_b\Zrm_{f^*,g^*_\alpha}^{[k_0-2]}\right )
=(b^2-\ep_f(b)\ep_g(b))\cdot 
\left (1-\frac{p^{k_0-1}}{\beta (f) \alpha (g)} \right )\cdot 
\left (1-\frac{p^{k_0-1}}{\alpha (f) \alpha (g)} \right ) \Zrm_{f^*,g^*}^{[k_0-2]}.
\end{equation}
Taking into account  (\ref{functional equation: formula 1}), (\ref{functional equation: formula 3}) and (\ref{functional equation: formula 4}), we can write   (\ref{functional equation: formula 2}) in the form
\begin{multline}
\label{computation of L(f*,g*)}
L_p^{\mathrm{wt}}(\f^*,\g^*,k_0)
=\frac{1}{(k_0-2)!\cdot C(f)G(\ep_f) G(\ep_g)}
\left (1-\frac{\beta (f) \alpha (g)}{p^{k_0}} \right )\times \\
\times
\left (1-\frac{p^{k_0-1}}{\alpha (f) \alpha (g)} \right ) 
\left [ \log \left (\Zrm_{f^*,g^*}^{[k_0-2]}\right ), n_{\alpha\beta}
\right ]_{\bN_{f^* ,g^*}}
.
\end{multline}
Now the theorem follows from (\ref{computation of L(f,g,k)}), (\ref{computation of L(f*,g*)}) and Corollary~\ref{functional equation for improved functions}.
\end{proof}
\begin{myremark}
\label{remark euler like factors}
{\rm
The explicit form of the  Euler-like factor $\CE (V_{f,g},D_{-1})$ is 
\[
\CE (V_{f,g},D_{-1})=\left (1-\frac {p^{k_0-1}}{\alpha (f) \alpha (g)} \right ) 
\left (1-\frac {\alpha (f) \beta (g)}{p^{k_0}} \right )
\left (1-\frac {\beta (f) \alpha (g)}{p^{k_0}} \right )
\left (1-\frac {\beta (f) \beta (g)}{p^{k_0}} \right ).
\]
}
\end{myremark}

\section{Extra zeros of Rankin--Selberg $L$-functions}

\subsection{The $p$-adic regulator}
\label{subsection Extra-zeros of $p$-adic Rankin--Selberg $L$-functions} 

\subsubsection{} In this section, we prove the main result of the paper. 
Let $f$ and $g$ be  two newforms of the same  weight $k_0\geqslant 2,$  levels $N_f$ and $N_g$ and nebentypus $\ep_f$ and $\ep_g$ respectively. Fix a prime number $p\geqslant 5$ such that $(p,N_fN_g)=1.$ As before, we denote by $\alpha (f)$ and 
$\beta (f)$ (respectively by $\alpha (g)$ and $\beta (g)$) the roots of the Hecke polynomial of $f$ (respectively $g$) at $p.$  We will always assume that  conditions {\bf M1-4)} hold, namely 

\begin{itemize}
\item[]{\bf M1)}  $\alpha (f)\neq \beta (f)$
and $\alpha (g)\neq \beta (g).$

\item[]{\bf M2)}  $v_p(\alpha (f))<k_0-1$ and $ v_p(\alpha (g)) <k_0-1.$

\item[]{\bf M3)} The characters  $\ep_f,$ $\ep_g$ and $\ep_f\ep_g$ are primitive modulo $N_f,$ $N_g$ and $\mathrm{lcm}(N_f,N_g)$ respectively.

\end{itemize}
We make also the  following additional assumption which will allow us to apply
Theorem~\ref{theorem functional equation for zeta elements}:

\begin{itemize}
\item[]{\bf M4)} $\ep_f(p) \ep_g(p)\neq 1.$
\end{itemize}
We maintain the notation of Section~\ref{subsection zeta elements}. 
Let  $V_g=W_g(k_0)$ and $V_{f,g}=W_{f,g}(k_0).$ 
The two-dimensional $E$-subspace
\[
D=E\eta_f^\alpha \otimes \Dc (V_g)\subset \Dc (V_{f,g})
\]
is stable under the action of $\Ph .$ Let $\f$ and $\g$ be Coleman families passing through $f_\alpha$ anf $g_\alpha$
and let $L_p(\f,\g, \omega^{k_0}) (x,y,s)$ denote the three-variable
$p$-adic $L$-function. 

\begin{mydefinition} We define the one-variable $p$-adic $L$-function 
$L_{p,\alpha}(f,g,s)$ by
\begin{equation}
\nonumber
L_{p,\alpha}(f,g,s)=L_p(\f,\g, \omega^{k_0}) (k_0,k_0,s).
\end{equation} 
\end{mydefinition}

Note that if $v_p (\beta (g)) < k_0-1$ and $\tilde\g$ denotes
the Coleman family passing through $g_\beta,$ then the density argument 
shows that
\[
L_p(\f,\g,\omega^{k_0})(x,k_0,s)=L_p(\f, {\tilde\g}, \omega^{k_0})(x,k_0,s),
\]
 and therefore our definition does not depend 
on the choice of the stabilization of $g$  (see \cite[Proposition~3.6.3]{BLLV}).

The Euler-like factor (\ref{definition of Euler-like factor}) takes
the form
\begin{equation}
\nonumber
\CE (V_{f,g}, D)=
\left (1-\frac{p^{k_0-1}}{\alpha (f) \alpha (g)}\right )\cdot
 \left (1-\frac{p^{k_0-1}}{\alpha (f) \beta (g)}\right ) 
\cdot \left (1-\frac{\beta (f) \alpha(g)}{p^{k_0}}\right ) \cdot
  \left (1-\frac{\beta (f) \beta(g)}{p^{k_0}}\right ).
\end{equation}
The weight argument shows that only the first two factors of this product 
can vanish and that they can not vanish simultaneously. Exchanging 
$\alpha (g)$ and $\beta (g)$ if necessary,  without loss of generality we can   assume that 
\begin{itemize}
\item[]{\bf M5)}
$\alpha (f) \alpha (g) \neq p^{k_0-1}.$ 
\end{itemize}

\subsubsection{} Let $\BFrm_{f^*,g^*}^{[k_0-2]}\in H^1_S(\Q, W_{f^*,g^*}^*(2-k_0))$
denote the element of Beilinson--Flach associated to the forms $f^*,$ $g^*$
(see Definition~\ref{definition of Beilinson-Flach}). Using the  canonical isomorphism $W_{f^*,g^*}^*(2-k_0)\simeq V_{f,g}$ we can  consider 
it as an element 
\[ 
\BFrm_{f^*,g^*}^{[k_0-2]}\in H^1_S(\Q, V_{f,g}).
\] 
For any prime $l,$ we denote by 
$\res_l \left (\BFrm_{f^*,g^*}^{[k_0-2]}\right ) \in H^1 (\Q_l,V_{f,g})$ the localization of this element at $l.$

\begin{mylemma}
\label{lemma properties Beilinson-Flach in H}
 The following holds true:

1) $\res_p \left (\BFrm_{f^*,g^*}^{[k_0-2]}\right ) \in H^1_f (\Qp,V_{f,g}).$ 

2) Assume that for each prime divisor $l \mid N_fN_g$ the factorization 
of $N_f$ or $N_g$ contains $l$ with  multiplicity $1.$ Then 
\[
\BFrm_{f^*,g^*}^{[k_0-2]}\in  H^1_{f,\{p\}} (\Q ,V_{f,g}).
\]
\end{mylemma}
\begin{proof} The first statement is proved in \cite[Proposition~5.4.1]{KLZb}
and was already mentioned in Section~\ref{subsection Local properties of Beilinson--Flach elements}. The second statement follows from the fact that 
$H^1_{f} (\Q_l ,V_{f,g})=H^1 (\Q_l ,V_{f,g})$ if and only if 
$H^0(\Q_l ,V^*_{f,g}(1))=0$ and the monodromy-weight conjecture
for modular forms \cite{Sa00}.
\end{proof}

\subsubsection{} Recall that $\Dc (V_{f,g})$ is equipped with the filtration 
(\ref{filtration on Dcris}). We denote by 
\linebreak
$\left (F_i\Ddagrig (V_{f,g})\right )_{i=-2}^2$ the associated triangulation of $\Ddagrig (V_{f,g}).$ 
The eigenvector $d_{\beta\alpha}=\omega_f^{\beta}\otimes\eta_g^{\alpha}\otimes e_{k_0}$ defined in (\ref{canonical basis of Dc (V)})
is a canonical basis of $\CDcris \left (\gr_1\Ddagrig (V_{f,g})\right ).$  
As in Section~\ref{subsection zeta elements}, we denote by 
\[
\log \,:\, H^1 \left (\gr_1\Ddagrig (V_{f,g})\right ) 
\rightarrow \CDcris  \left (\gr_1\Ddagrig (V_{f,g})\right )
\]
the logarithm map of Bloch and Kato. 
Let
\[
\Zrm_{f^*,g^*}^{[k_0-2]}\in H^1\left (\gr_1\Ddagrig (V_{f,g})\right )
\]
be the image of $\res_p \left (\BFrm_{f^*,g^*}^{[k_0-2]}\right )$ under the canonical projection  
\[
H^1(\Qp, V_{f,g}) \rightarrow H^1\left (\Ddagrig (V_{f,g})/
F_0\Ddagrig (V_{f,g})\right )
\] 
(see Section~\ref{subsubsection localisation of Beilinson-Flach}).
Denote by $\widetilde R_p  \left (V_{f,g}, D \right )   \in E$ the unique 
element of $E$ such that 
\begin{equation}
\nonumber
\log \left (\Zrm_{f^*,g^*}^{[k_0-2]}\right )= 
\widetilde R_p  \left (V_{f,g}, D \right ) \cdot
d_{\beta\alpha} .
\end{equation}
Since  $d_{\beta\alpha}\in \CDcris  \left (\gr_1\Ddagrig (V_{f,g})\right )$ is the dual basis of 
\[
n_{\alpha \beta}\in  \CDcris  \left (\bN_{f^*,g^*}\right )
\simeq
\CDcris  \left (\gr_0\Ddagrig (V_{f,g}^*(1))\right ) 
\]
(see Section~\ref{subsection zeta elements}),   we have
\begin{equation}
\label{regulator as pairing}
\widetilde R_p  \left (V_{f,g}, D \right )= 
\left [\log \left (\Zrm_{f^*,g^*}^{[k_0-2]}\right >, n_{\alpha \beta} \right ]_{\bN_{f^*,g^*}}.
\end{equation}

Let $\eta_g\in \Dc (W_g)$ be any vector such that $[\eta_g,\omega_{g^*}]=e_{1-k_0}.$
Set $b=\omega_f\otimes \eta_g\otimes e_{k_0}\in \Dc (V_{f,g}).$
Then the class
\[
\overline b_\alpha=b \mod{\left (\F^0\Dc (V_{f,g})+D \right )}
\]
is nonzero, does not depend on the choice of $\eta_g$ and therefore gives  a canonical basis
of the one-dimensional vecor space  $\Dc (V_{f,g})/\left (\F^0\Dc (V_{f,g})+D \right ).$ 

\begin{myproposition} 
\label{proposition about regulator}
1) The representation $V_{f,g}$ satisfies
 conditions {\bf C1-3)} of Section~\ref{subsection regilar submodules}.

2) Assume  that the representation $V_{f,g}$ satisfies conditions {\bf C4-5)},
namely that $H^1_f(\Q, V_{f,g}^*(1))=0$ and the localization map 
$H^1_f(\Q, V_{f,g})\rightarrow  H^1_f(\Qp, V_{f,g})$  is injective.
Then 

i) $D$ is a regular submodule if and only if  $\Zrm_{f^*,g^*}^{[k_0-2]}\neq 0.$

ii) If that is the case, then $\widetilde R_p  \left (V_{f,g}, D \right )    $ coincides
with the determinant of the regulator map 
\[
H^1_f(\Q, V_{f,g}) \rightarrow \Dc (V_{f,g})/\left (\F^0\Dc (V_{f,g})+D \right )
\]
computed in the bases 
$\BFrm_{f^*,g^*}^{[k_0-2]}\in H^1_f(\Q, V_{f,g})
$
and 
$
\overline{b}_{\alpha}\in  \Dc (V_{f,g})/(\F^0\Dc (V_{f,g})+D).
$
\end{myproposition}

\begin{proof}
1) The weight argument shows that $\Dc (V_{f,g})^{\Ph=1}=0$ and  $H^0(\Q, V_{f,g})=0.$ Since $V_{f,g}^*(1)=\Hom_E ( W_f,W_{g^*})$ and $f\neq g^*,$ 
we obtain that $H^0(\Q, V_{f,g})=\Hom_{E[G_{\Qp}]} ( W_f,W_{g^*})=0.$
The semisimplicity of $\Ph$ follows from {\bf M1)}. 

2) From the congruences $\eta_g\equiv \eta_g^{\alpha}\pmod{\F^{k_0-1}\Dc (W_{g})}$
and $\omega_f\equiv \omega_f^{\beta}\mod{\Dc (W_f)^{\Ph= \alpha (f)}}$
it is easy to see that $\overline{b}_\alpha =\overline{d}_{\beta\alpha}.$
Now the second statement follows directly from the definition of the regulator
map.
\end{proof}


\subsection{The $\mathcal L$-invariant}
\subsubsection{} 
Set 
\[
\,_b\widetilde\BFrm_{f,g}^{[k_0-1]}=
\left ({\Pr}^\alpha_*, {\Pr}^\alpha_*\right )\circ
\spec^{\f,\g,c}_{k_0,k_0,1-k_0}\left (
\,_b\BF_{\f,\g}^{\Iw}\right ) \in H^1_S(\Q, V^*_{f,g}(1)),
\]
where $\,_b\BF_{\f,\g}^{\Iw}$ is the class in Iwasawa cohomology constructed
in  Proposition~\ref{proposition specialization of two variable Beilinson Flach elements}. Note that, unlike  $\,_b\BFrm_{f^*,g^*}^{[k_0-2]},$
the element $\,_b\widetilde\BFrm_{f,g}^{[k_0-1]}$ is not a proper Beilinson--Flach class and its construction involves $p$-adic interpolation. 

\begin{mylemma} For all primes $l\neq p$ we have 
\[
\res_l\left (\,_b\widetilde\BFrm_{f,g}^{[k_0-1]}\right )\in H^1_f(\Q_l,
V^*_{f,g}(1)),
\]
and therefore 
\[
\,_b\widetilde\BFrm_{f,g}^{[k_0-1]}\in H^1_{f,\{p\}}(\Q, V^*_{f,g}(1)).
\] 
\end{mylemma}
\begin{proof} By \cite[Section~2.1.7]{PR92}, the image of the projection map
\[
H^1_{\Iw}(\Q_l, V^*_{f,g}(1))\rightarrow H^1(\Q_l, V^*_{f,g}(1))
\]
is contained in $H^1_{f}(\Q_l, V^*_{f,g}(1)).$ This implies the lemma.
\end{proof}

Choose $b$ such that $\ep_f(b)\ep_g(b)\neq 1.$ The element $\,_b\widetilde\Zrm_{f,g}^{[k_0-1]}$ constructed in Section~\ref{subsection zeta elements} is the image of $\,_b\widetilde\BFrm_{f,g}^{[k_0-1]}$
under 
the composition  
\[
H^1_S(\Q, V^*_{f,g}(1))\xrightarrow{\res_p} H^1(\Qp, V^*_{f,g}(1))
\rightarrow H^1 \left (\Ddagrig (V_{f,g}^*(1))/F_0\Ddagrig (V_{f,g}^*(1))
\right ). 
\]
From Proposition~\ref{proposition local property of  two-variable BF} it follows that
\[
\,_b\widetilde\Zrm_{f,g}^{[k_0-1]}\in H^1\left (\gr_1 \Ddagrig (V_{f,g}^*(1))
\right ).
\]
Note  that $\Ph$ acts on $\CDcris \left (\gr_1 \Ddagrig (V_{f,g}^*(1))
\right )$ as multiplication
by $\displaystyle\frac{p^{k_0-1}}{\alpha (f) \beta (g)}.$

\subsubsection{} 
To simplify notation, we set $\bM_{f,g}=\gr_0\Ddagrig (V_{f,g})$
(see Section~\ref{subsection zeta elements}). 
Then $\bM_{f,g}^*(\chi)\simeq \gr_1\Ddagrig (V_{f,g}^*(1)).$
Assume that $\alpha (f) \beta (g)=p^{k_0-1}.$ Then 
formulas (\ref{formulas for phi action on M}) show that $\bM_{f,g}\simeq \CR_E(\chi)$ and, dually, 
$\bM_{f,g}^*(\chi)=\CR_E.$ Therefore $\bM_{f,g}^*(\chi)$
is the $(\Ph,\Gamma)$-module associated to the trivial representation, and 
we have canonical isomorphisms $\CDcris (\bM_{f,g}^*(\chi))\simeq E$ and $H^1(\bM_{f,g}^*(\chi))\simeq H^1(\Qp,E).$
Clearly,  $\bM_{f,g}^*(\chi)$ satisfies  condition
(\ref{conditions for exceptional (phi,Gamma)-module})
and the decomposition (\ref{construction of i_(W^*(chi))})
has the following  interpretation in terms of Galois cohomology. 
Let $\mathrm{ord} \,:\,\Gal (\Qp^{\textrm{ur}}/\Qp) \rightarrow \Zp$
denote the unramified character defined by $\mathrm{ord} (\Fr_p)=-1,$
where $\Fr_p$ is the geometric Frobenius. Denote by $\log \chi$  the logarithm
of the cyclotomic character viewed as an additive character of the whole 
Galois group. We have a commuative diagram 
\[
\xymatrix{
\CDcris (\bM_{f,g}^*(\chi))\oplus \CDcris (\bM_{f,g}^*(\chi))
\ar[d]^{=}
\ar[rr]^(.6){i_{\bM_{f,g}^*(\chi)}} & &H^1(\bM_{f,g}^*(\chi)) \ar[d]^{=}\\
E\oplus E\ar[rr] & & H^1(\Qp,E),
}
\]
where the bottom horizontal arrow is given by $(x,y)\mapsto x\cdot \mathrm{ord}+y\cdot \log \chi .$ This follows  from the explicit description of the Galois cohomology in terms of $(\Ph,\Gamma)$-modules (see \cite[Proposition~1.3.2]{Ben00}
or \cite[Proposition~I.4.1]{CC99}). Under the right vertical map, the subspaces $H^1_f(\bM_{f,g}^*(\chi))$ and $H^1_c(\bM_{f,g}^*(\chi))$ are mapped onto 
the subspaces generated by the characters $\mathrm{ord}$ and $\log \chi$ respectively. We refer the reader to \cite[Section~1.5]{Ben11}
for further comments.

\begin{mydefinition} 
\label{definition of ad hoc L-invariant}
Assume that $\alpha (f) \beta (g)=p^{k_0-1}$ and 
$\,_b\widetilde\Zrm_{f,g}^{[k_0-1]}\notin H^1_f(\bM_{f,g}^*(\chi)).$ Then 
\[
\,_b\widetilde\Zrm_{f,g}^{[k_0-1]}= A \cdot \mathrm{ord}_p +B\cdot \log \chi
\]
for unique $A,B\in E$ such that $B\neq 0,$  and we define 
\[
\widetilde{\mathscr L} (V_{f,g}, D)=A/B.
\]

\end{mydefinition}

\begin{myproposition} Assume that the following holds:
\begin{itemize}

\item[]{a)} $\alpha (f)\beta (g)=p^{k_0-1}.$

\item[]{b)} The representation $V_{f,g}$ satisfies  conditions {\bf C4-5)}
of Section~\ref{subsection regilar submodules}.

\item[]{c)}  $Z_{f^*,g^*}^{[k_0-2]}\neq 0.$
\end{itemize}

Then $\widetilde{\mathcal L} (V_{f,g}, D)=\mathcal L(V_{f,g}, D),$
where $\mathcal L(V_{f,g}, D)$ is the invariant defined in Section~\ref{The L-invariant}.
\end{myproposition}
\begin{proof} From Proposition~\ref{proposition about regulator} it follows that
$D$ is regular. By Proposition~\ref{proposition local property of  two-variable BF},  the image of the element $\widetilde{\BFrm}_{f,g}^{[k_0-1]}\in H^1_{f,\{p\}}(\Q, V_{f,g}^*(1))$ under the  map
\[
H^1_{f,\{p\}}(\Q, V_{f,g}^*(1)) \rightarrow
\frac{H^1(\Qp, V^*(1))}{H^1\left (F_0\Ddagrig (V^*(1))\right )}
\]
is $\widetilde Z_{f,g}^{[k_0-1]}\in H^1(\bM_{f,g}^*(\chi)).$
Since $\ep_f(p)\ep_g(p)\neq 1,$ the condition a) implies that
$\beta (f)\alpha (g)\neq p^{k_0-1},$ and we can apply Theorem~\ref{theorem functional equation for zeta elements}. By assumption, $Z_{f^*,g^*}^{[k_0-2]}\neq 0.$ Therefore
$\left <\widetilde Z_{f,g}^{[k_0-1]}, \exp (d_{\alpha\beta})\right >_{\bM_{f,g}}\neq 0$ and $\widetilde Z_{f,g}^{[k_0-1]}\notin H^1_f(\bM_{f,g}^*(\chi)).$
Comparing this with the isomorphism (\ref{definition of dual L-inv:main isomorphism}),
we see that 
\[
\widetilde{\BFrm}_{f,g}^{[k_0-1]}\in H^1(D^\perp,V_{f,g}^*(1)).
\] 
Now it is easy to see that our construction of the ad hoc invariant 
concides (up to the sign) with the invariant defined in  Section~\ref{subsection second definition of L-invariant}:
\[
\widetilde{\mathscr L} (V_{f,g}, D)=-\mathscr L (V_{f,g}^*(1), D^\perp).
\]
Hence, the proposition follows from Proposition~\ref{proposition comparision l-inv}.
 \end{proof}

\subsection{The main theorem}

In this section, we prove Theorem~I.
We keep previous notation and conventions. 

\begin{mytheorem} 
\label{main theorem}
Assume that $\alpha (f) \beta (g)=p^{k_0-1}.$ Then

1) $L_{p,\alpha}(f,g,k_0)=0.$

2) The following conditions are equivalent:
\begin{itemize}
\item[i)]{} $\mathrm{ord}_{s=0}L_{p,\alpha}(f,g, s)=1$.
\item[ii)]{} $\,_b\widetilde\Zrm_{f,g}^{[k_0-1]}\notin 
 H^1_c\left (\gr_1 \Ddagrig (V_{f,g}^*(1))
\right ).$
\end{itemize}

3) In addition to the assumption that $\alpha (f) \beta (g)=p^{k_0-1},$ 
suppose that 
\[
\,_b\widetilde\Zrm_{f,g}^{[k_0-1]}\notin 
 H^1_f\left (\gr_1 \Ddagrig (V_{f,g}^*(1))
\right ).
\] 
Then 
\[
L_{p,\alpha}'(f,g,k_0)=\frac{\ep (f,g,k_0)\cdot \widetilde{\mathcal L}(V_{f,g},D) \cdot
\CE^+(V_{f,g},D)}{  C(f) \cdot   G(\ep_f) \cdot   G(\ep_g)
\cdot  (k_0-2)!}  \cdot \widetilde{R}_p (V_{f,g},D), 
\]
where 
\begin{equation}
\nonumber
\CE^+(V_{f,g},D)=\left (1-\frac {p^{k_0-1}}{\alpha (f) \alpha (g)} \right ) 
\left (1-\frac {\beta (f) \alpha (g)}{p^{k_0}} \right )
\left (1-\frac {\beta (f) \beta (g)}{p^{k_0}} \right ).
\end{equation}
\end{mytheorem}
\begin{proof}
1) 
From Theorem~\ref{theorem relation L with Beilinson-Flach element}, Lemma~\ref{lemma first improved L} and the identity (\ref{comparision generators of Mfg}) 
it follows that
\begin{multline}
\label{formula for cyclotomic L-function}
B_b(\omega^{k_0}, k_0, k_0, k_0+s) L_{p,\alpha}(f,g,s)=
(-1)^{k_0}\mathcal A_{\omega^0}\left (\frak{Log}_{\bM_{\f,\g}, m_{\f,\g}}
\left (\Tw_{\bchi_{\g}^{-1}} (\,_b\Z_{\f,\g}^{\Iw}))\right ) \right )(k_0,k_0,s)=\\
=
\frac {(-1)^{k_0}}{C(f)\lambda_{N_f}(f)}
\cdot\mathcal A_{\omega^0}\left (\frak{Log}_{\bM_{f,g}, d_{\alpha\beta}}
\left (\Tw_{1-k_0} (\,_b\Zrm_{f,g}^{\Iw})\right ) \right )(s).
\end{multline} 
By Proposition~\ref{them:propertiestwovarPRlog} we have 
\begin{multline}
\nonumber
B_b(\omega^{k_0}, k_0, k_0, k_0) L_p(f,g,k_0)=\\
=\frac {(-1)^{k_0}}{C(f)\lambda_{N_f}(f)}
\cdot
\left (
1-\frac{p^{k_0-1}}{\alpha (f)\beta (g)}
\right ) 
\cdot
\left (
1-\frac{\alpha (f)\beta (g)}{p^{k_0}}
\right )\cdot\left <\,_b\widetilde\Zrm_{f,g}^{[k_0-1]}, \exp_{\bM_{f,g}}(d_{\alpha\beta})
\right >_{\bM_{f,g}}.
\end{multline}
This proves 1).

2) The derivative of the large logarithm map in presence of trivial zeros
is computed in \cite[Propositions~1.3.6 and 2.2.2]{Ben14a}\footnote{In \cite{Ben14a}, only the case of $p$-adic representations is considered, but for $(\Ph,\Gamma)$-modules the proof is exactly the same.}. Applying it to our case (see especially
formulas (24) and (25) of op. cit.), we obtain
\begin{equation}
\label{proof of main theorem 1}
\left. \frac{d}{ds}\mathcal A_{\omega^0}\left (\frak{Log}_{\bM_{f,g}, d_{\alpha\beta}}
\left (\Tw_{1-k_0} (\,_b\Zrm_{f,g}^{\Iw})\right ) \right )(s)\right \vert_{s=0}
=
-\left (1-\frac{1}{p} \right )^{-1}\cdot \left <\,_b\widetilde\Zrm_{f,g}^{[k_0-1]}, i_{\bM_{f,g},c}(d_{\alpha\beta}) \right >_{\bM_{f,g}}.
\end{equation}
Write $\,_b\widetilde\Zrm_{f,g}^{[k_0-1]}= A \cdot \mathrm{ord}_p +B\cdot \log \chi.$
Then 
\begin{equation}
\label{proof of main theorem 2}
 \left <\,_b\widetilde\Zrm_{f,g}^{[k_0-1]}, i_{\bM_{f,g},c}(d_{\alpha\beta}) \right >_{\bM_{f,g}} = A \left <\mathrm{ord}_p,  i_{\bM_{f,g},c}(d_{\alpha\beta})  \right >_{\bM_{f,g}}=-A.
\end{equation}
This implies 2).

3) By \cite[Theorem~1.5.7]{Ben11}, $\exp_{\bM_{f,g}}(d_{\alpha\beta})= i_{\bM_{f,g},f}(d_{\alpha\beta}),$ and therefore
\begin{equation}
\nonumber
\left <\,_b\widetilde\Zrm_{f,g}^{[k_0-1]}, \exp_{\bM_{f,g}}(d_{\alpha\beta}) \right >_{\bM_{f,g}}=
\left <\,_b\widetilde\Zrm_{f,g}^{[k_0-1]}, i_{\bM_{f,g},f}(d_{\alpha\beta}) \right >_{\bM_{f,g}}=B.
\end{equation}
Taking into account  Definition~\ref{definition of ad hoc L-invariant} and  (\ref{proof of main theorem 2}), we obtain that
\begin{multline}
\label{proof of main theorem 4}
 \left <\,_b\widetilde\Zrm_{f,g}^{[k_0-1]}, i_{\bM_{f,g},c}(d_{\alpha\beta}) \right >_{\bM_{f,g}}=
 -
\widetilde{\mathcal L}(V_{f,g},D)\cdot B=
\\
=-
\widetilde{\mathcal L}(V_{f,g},D) \cdot
\left <\,_b\widetilde\Zrm_{f,g}^{[k_0-1]}, \exp_{\bM_{f,g}}(d_{\alpha\beta}) \right >_{\bM_{f,g}}.
\end{multline}
Formulas (\ref{formula for cyclotomic L-function}), (\ref{proof of main theorem 1})
and (\ref{proof of main theorem 4}) give
\begin{multline}
\nonumber
B_b(\omega^{k_0}, k_0, k_0, k_0) \cdot  L'_{p,\alpha}(f,g,k_0)=\\
=
\frac {(-1)^{k_0-1}}{C(f)\lambda_{N_f}(f)}
\left (1-\frac{1}{p}\right )^{-1} \cdot 
\widetilde{\mathcal L}(V_{f,g},D) \cdot 
\left <\,_b\widetilde\Zrm_{f,g}^{[k_0-1]}, \exp_{\bM_{f,g}}(d_{\alpha\beta}) \right >_{\bM_{f,g}}.
\end{multline}
Since $\alpha (f)\beta (g)=p^{k_0-1},$ the condition {\bf M4)} implies that 
$\beta (f) \alpha (g)\neq p^{k_0-1}$ and we can apply Theorem~\ref{theorem functional equation for zeta elements}. Taking
into account Remark~\ref{remark euler like factors}, (\ref{formula for factor B}) 
and (\ref{definition of Zrm independent on b}), we obtain that
\begin{equation}
\nonumber
L_{p,\alpha}'(f,g,k_0)=\frac{\ep (f,g,k_0)\cdot \widetilde{\mathcal L}(V_{f,g},D) \cdot
\CE^+(V_{f,g},D)}{  C(f) \cdot   G(\ep_f) \cdot   G(\ep_g)
\cdot  (k_0-2)!}  \cdot \left [ \log \left (\Zrm_{f^*,g^*}^{[k_0-2]}\right ), n_{\alpha\beta}
\right ]_{\bN_{f^* ,g^*}}.
\end{equation}
Using (\ref{regulator as pairing}), we can replace 
$\left [ \log \left (\Zrm_{f^*,g^*}^{[k_0-2]}\right ), n_{\alpha\beta}\right ]_{\bN_{f^* ,g^*}}$ by $\widetilde{R}_p(V_{f,g},D).$  The theorem is proved.
\end{proof}

\bibliographystyle{style}

\end{document}